\documentclass[11pt, a4paper]{article}
\usepackage{a4wide}
\usepackage{hyperref}
\usepackage[british]{babel}
\usepackage{enumerate, longtable}
\usepackage{amsmath, amscd, amsfonts, amsthm, amssymb, latexsym, comment, stmaryrd, graphicx}
\usepackage[all]{xypic}
\usepackage[T1]{fontenc}
\usepackage[utf8]{inputenc}

\newtheorem{theorem}{Theorem}[section]
\newtheorem{lemma}[theorem]{Lemma}
\newtheorem{definition}[theorem]{Definition}
\newtheorem{remark}[theorem]{Remark}
\newtheorem{example}[theorem]{Example}
\newtheorem{proposition}[theorem]{Proposition}
\newtheorem{corollary}[theorem]{Corollary}

\newtheorem{algorithm}[theorem]{Algorithm}

\newtheorem{exercise}[theorem]{Exercise}
\newtheorem{cexercise}[theorem]{Computer Exercise}

\title{Computational Arithmetic of Modular Forms}
\author{Gabor Wiese}
\date{}

\newcommand{\End}{{\rm End}}
\newcommand{\Hom}{{\rm Hom}}

\newcommand{\Div}{{\rm Div}}

\newcommand{\GL}{\mathrm{GL}}

\newcommand{\SL}{\mathrm{SL}}

\newcommand{\id}{\mathrm{id}}
\newcommand{\obj}{\mathrm{obj}}

\DeclareMathOperator{\Sym}{Sym}
\DeclareMathOperator{\Ext}{Ext}
\DeclareMathOperator{\Tor}{Tor}

\DeclareMathOperator{\Image}{im}
\DeclareMathOperator{\Imag}{Im}
\DeclareMathOperator{\Real}{Re}
\DeclareMathOperator{\Gal}{Gal}

\DeclareMathOperator{\Spec}{Spec}

\DeclareMathOperator{\Frob}{Frob}

\DeclareMathOperator{\Tr}{Tr}

\DeclareMathOperator{\Det}{det}
\DeclareMathOperator{\Mat}{Mat}
\DeclareMathOperator{\coker}{coker}

\DeclareMathOperator{\Stab}{Stab}
\DeclareMathOperator{\vol}{vol}

\newcommand{\Ind}{{\rm Ind}}

\newcommand{\Res}{{\rm Res}}

\newcommand{\Coind}{{\rm Coind}}
\newcommand{\PGL}{\mathrm{PGL}}
\newcommand{\PSL}{\mathrm{PSL}}

\newcommand{\im}{\mathrm{im}}

\newcommand{\diam}[1]{{\langle #1 \rangle}}

\newcommand{\donothing}[1]{}

\newcommand{\mat}[4]{
 \left(  \begin{smallmatrix} #1 & #2 \\ #3 & #4 \end{smallmatrix} \right)}

\newcommand{\vect}[2]{
 \left(  \begin{smallmatrix} #1 \\ #2 \end{smallmatrix} \right)}

\newcommand{\calC}{\mathcal{C}}
\newcommand{\calD}{\mathcal{D}}

\newcommand{\calF}{\mathcal{F}}
\newcommand{\calG}{\mathcal{G}}

\newcommand{\calM}{\mathcal{M}}

\newcommand{\calO}{\mathcal{O}}
\newcommand{\calP}{\mathcal{P}}

\newcommand{\calR}{\mathcal{R}}

\newcommand{\cB}{\mathcal{B}}
\newcommand{\cC}{\mathcal{C}}
\newcommand{\cE}{\mathcal{E}}
\newcommand{\cF}{\mathcal{F}}
\newcommand{\cG}{\mathcal{G}}
\newcommand{\cM}{\mathcal{M}}

\newcommand{\cO}{\mathcal{O}}

\newcommand{\fa}{\mathfrak{a}}

\newcommand{\fm}{\mathfrak{m}}
\newcommand{\fn}{\mathfrak{n}}

\newcommand{\fp}{\mathfrak{p}}

\newcommand{\fP}{\mathfrak{P}}

\newcommand{\CC}{\mathbb{C}}

\newcommand{\FF}{\mathbb{F}}

\newcommand{\HH}{\mathbb{H}}

\newcommand{\NN}{\mathbb{N}}

\newcommand{\PP}{\mathbb{P}}
\newcommand{\QQ}{\mathbb{Q}}
\newcommand{\RR}{\mathbb{R}}

\newcommand{\TT}{\mathbb{T}}

\newcommand{\ZZ}{\mathbb{Z}}

\newcommand{\Qbar}{\overline{\QQ}}
\newcommand{\Zbar}{\overline{\ZZ}}

\newcommand{\Fbar}{\overline{\FF}}
\newcommand{\Kbar}{{\overline{K}}}
\newcommand{\Khat}{{\widehat{K}}}
\newcommand{\Khatbar}{{\overline{\widehat{K}}}}

\newcommand{\chibar}{\overline{\chi}}

\newcommand{\zbar}{\overline{z}}
\newcommand{\gbar}{{\overline{g}}}
\newcommand{\fbar}{{\overline{f}}}
\newcommand{\Gammabar}{\overline{\Gamma}}

\newcommand{\matz}{\mathrm{Mat}_2(\ZZ)_{\det \neq 0}}

\newcommand{\conj}{\mathrm{conj}}
\newcommand{\res}{\mathrm{res}}
\newcommand{\Sh}{\mathrm{Sh}}
\newcommand{\cores}{\mathrm{cores}}
\newcommand{\infl}{\mathrm{infl}}
\newcommand{\modules}{\mathrm{modules}}

\newcommand{\tr}{\mathrm{tr}}

\newcommand{\ms}{\hspace*{1cm}}

\DeclareMathOperator{\h}{H}
\DeclareMathOperator{\Z}{Z}
\DeclareMathOperator{\B}{B}
\newcommand{\Hpar}{\h_{\mathrm{par}}}
\newcommand{\Zpar}{\Z_{\mathrm{par}}}

\newcommand{\Sk}[4]{{\mathrm{S}}_{#1}(#2,#3\,;\,#4)}
\newcommand{\Skg}[3]{{\mathrm{S}}_{#1}(#2\,;\,#3)}
\newcommand{\Mk}[4]{{\mathrm{M}}_{#1}(#2,#3\,;\,#4)}
\newcommand{\Mkg}[3]{{\mathrm{M}}_{#1}(#2\,;\,#3)}

\newcommand{\cCMk}[4]{{\mathcal{CM}}_{#1}(#2,#3\,;\,#4)}
\newcommand{\cMk}[4]{{\mathcal{M}}_{#1}(#2,#3\,;\,#4)}
\newcommand{\cBk}[4]{{\mathcal{B}}_{#1}(#2,#3\,;\,#4)}
\newcommand{\cEk}[4]{{\mathcal{E}}_{#1}(#2,#3\,;\,#4)}
\newcommand{\Skone}[3]{{\mathrm{S}}_{#1}(\Gamma_1(#2)\,;\,#3)}
\newcommand{\Mkone}[3]{{\mathrm{M}}_{#1}(\Gamma_1(#2)\,;\,#3)}
\newcommand{\cCMkone}[3]{{\mathcal{CM}}_{#1}(\Gamma_1(#2)\,;\,#3)}
\newcommand{\cMkone}[3]{{\mathcal{M}}_{#1}(\Gamma_1(#2)\,;\,#3)}

\newcommand{\cCM}{\mathcal{CM}}

\begin{document}

\selectlanguage{british}

\maketitle

\abstract{
These course notes are about computing modular forms and some of their arithmetic properties.
Their aim is to explain and prove the modular symbols algorithm in as elementary and as explicit terms as possible,
and to enable the devoted student to implement it over any ring
(such that a sufficient linear algebra theory is available in the chosen computer algebra system).
The chosen approach is based on group cohomology and along the way the needed tools from homological algebra are provided.

}

\tableofcontents

\section*{Preface}
All sections of this course are either labelled as `Theory' or as
`Algorithms and Implementations'.
It is possible to study only the theory parts. However, the algorithmic parts
depend heavily on the developed theory. Of course, if one is principally interested in implementations,
one need not understand each and every proof.
Accordingly, theoretical and computer exercises are provided.

The conception of this course is different from every treatment I
know, in particular, from William Stein's excellent book `Modular
Forms: A Computational Approach' (\cite{Stein}) and from~\cite{Cremona}.
We emphasize the central role of Hecke 
algebras and focus on the use of group
cohomology since on the one hand it can be described in very explicit
and elementary terms and on the other hand already allows the
application of the strong machinery of homological algebra.
We shall not discuss any geometric approaches.

The treatment of the (group cohomological) modular symbols algorithm given in this course is complete.
However, we did not include any treatment of Heilbronn matrices describing Hecke operators
on Manin symbols, which allow a speed-up of Hecke operators.

This course was originally held at the Universit\"at Duisburg-Essen in 2008
and its notes have been slightly reworked for publication in this volume.

{\bf Acknowledgements.}
I would like to thank the anonymous referees for a huge number of helpful suggestions and corrections that surely improved
the text. Thanks are also due to the students who followed the original course, among them Maite Aranés, Adam Mohamed and Ralf Butenuth,
for their helpful feedback. I would also like to thank Mariagiulia De Maria, Daniel Berhanu Mamo, Atin Modi, Luca Notarnicola and Emiliano Torti for
useful corrections.

\section{Motivation and Survey}\label{sec:1}

This section serves as an introduction to the topics of the course. We will briefly review the theory of modular forms and Hecke operators. Then we will define the modular symbols formalism and state a theorem by Eichler and Shimura establishing a link between modular forms and modular symbols. This link is the central ingredient since the modular symbols algorithm for the computation of modular forms is entirely based on it. In this introduction, we shall also be able to give an outline of this algorithm.

\subsection{Theory: Brief review of modular forms and Hecke operators}

\subsubsection*{Congruence subgroups}

We first recall the standard congruence subgroups of $\SL_2(\ZZ)$.
By $N$ we shall always denote a positive integer.

Consider the group homomorphism
$$ \SL_2(\ZZ) \to \SL_2(\ZZ/N\ZZ).$$
By Exercise~\ref{exsln} it is surjective.
Its kernel is called the principal congruence subgroup of level~$N$ and denoted $\Gamma(N)$.

The group $\SL_2(\ZZ/N\ZZ)$ acts naturally on $(\ZZ/N\ZZ)^2$ (by multiplying the matrix with a vector).
We look at the orbit and the stabiliser of $\vect 10$. The orbit is
$$ \SL_2(\ZZ/N\ZZ)\vect 10 = \{ \vect ac \;|\;  a,c \textnormal{ generate }\ZZ/N\ZZ\}$$
because the determinant is~$1$.
We also point out that the orbit of $\vect 10$ can and should be viewed as the set of elements in $(\ZZ/N\ZZ)^2$ which are of precise (additive) order~$N$. 
We now consider the stabiliser of $\vect 10$ and define the group $\Gamma_1(N)$ as the preimage of that stabiliser group in $\SL_2(\ZZ)$.
Explicitly, this means that $\Gamma_1(N)$ consists of those matrices in $\SL_2(\ZZ)$ whose reduction modulo $N$ is of the form $\mat 1*01$.

The group $\SL_2(\ZZ/N\ZZ)$ also acts on $\PP^1(\ZZ/N\ZZ)$, the projective line over $\ZZ/N\ZZ$, which one can define as the tuples $(a:c)$ with $a,c \in \ZZ/N\ZZ$ such that $\langle a,c\rangle = \ZZ/N\ZZ$ modulo the equivalence relation given by multiplication by an element of $(\ZZ/N\ZZ)^\times$.
The action is the natural one (we should actually view $(a:c)$ as a column vector, as above).
The orbit of $(1:0)$ for this action is $\PP^1(\ZZ/N\ZZ)$.
The preimage in $\SL_2(\ZZ)$ of the stabiliser group of $(1:0)$ is called $\Gamma_0(N)$.
Explicitly, it consists of those matrices in $\SL_2(\ZZ)$ whose reduction is of the form $\mat **0*$.
We also point out that the quotient of $\SL_2(\ZZ/N\ZZ)$ modulo the stabiliser of $(1:0)$ corresponds to the set of cyclic subgroups of precise order $N$ in $(\ZZ/N\ZZ)^2$.
These observations are at the base of defining level structures for elliptic curves.

It is clear that $\Gamma_1(N)$ is a normal subgroup of $\Gamma_0(N)$ and that the map
$$ \Gamma_0(N) / \Gamma_1(N) \xrightarrow{\mat abcd \mapsto a \mod N} (\ZZ/N\ZZ)^\times$$
is a group isomorphism.

The quotient $\Gamma_0(N) / \Gamma_1(N)$ will be important in the sequel because it will act on modular forms and modular symbols for $\Gamma_1(N)$.
For that purpose, we shall often consider characters (i.e.\ group homomorphisms) of the form
$$ \chi : (\ZZ/N\ZZ)^\times \to \CC^\times.$$
We shall also often extend $\chi$ to a map $(\ZZ/N\ZZ) \to \CC$ by imposing $\chi(r) = 0$ if $(r,N) \ne 1$.

On the number theory side, the group $(\ZZ/N\ZZ)^\times$ enters as the Galois group of a cyclotomic extension.
More precisely, by class field theory or Exercise~\ref{qzetan} we have the isomorphism
$$ \Gal(\QQ(\zeta_N)/\QQ) \xrightarrow{\Frob_\ell \mapsto \ell} (\ZZ/N\ZZ)^\times$$
for all primes $\ell \nmid N$.
By $\Frob_\ell$ we denote (a lift of) the Frobenius endomorphism $x \mapsto x^\ell$, and by $\zeta_N$ we denote any primitive $N$-th root of unity.
We shall, thus, later on also consider $\chi$ as a character of $\Gal(\QQ(\zeta_N)/\QQ)$. The name
{\em Dirichlet character} (here of {\em modulus $N$}) is common usage for both.

\subsubsection*{Modular forms}

We now recall the definitions of modular forms. Standard references are \cite{DS} and~\cite{CS}, but I still vividly recommend~\cite{DiamondIm},
which gives a concise and yet rather complete introduction. We denote by
$$\HH = \{ z \in \CC | \im(z) > 0\}$$
the upper half plane. The set of cusps is by definition $\PP^1(\QQ) = \QQ \cup \{\infty\}$.
The group $\PSL_2(\ZZ)$ acts on $\HH$ by Möbius transforms; more explicitly, for $M = \mat abcd \in \SL_2(\ZZ)$ and $z \in \HH \cup \PP^1(\QQ)$
one sets
\begin{equation}\label{eq:moebius}
M.z = \frac{az+b}{cz+d}.
\end{equation}

For $M = \mat abcd$ an integer matrix with non-zero determinant,
an integer $k$ and a function $f:\HH \to \CC$, we put
$$ (f|_k M)(z) = (f|M)(z) := f\big(M.z\big) \frac{\Det(M)^{k-1}}{(cz+d)^k}.$$

Fix integers $k\ge 1$ and $N \ge 1$. A function
$$ f: \HH \to \CC$$
given by a convergent power series (the $a_n(f)$ are complex numbers) 
$$f(z) = \sum_{n=0}^\infty a_n(f) (e^{2\pi i z})^n = \sum_{n=0}^\infty a_n(f) q^n
\;\; \text{ with } q(z) = e^{2 \pi i z}$$
is called a {\em modular form of weight $k$ for $\Gamma_1(N)$} if
\begin{enumerate}[(i)]
\item $(f|_k \mat abcd) (z) = f(\frac{az+b}{cz+d})(cz+d)^{-k} = f(z)$ for all 
$\mat abcd \in \Gamma_1(N)$, and
\item the function $(f|_k \mat abcd) (z) = f(\frac{az+b}{cz+d}) (cz+d)^{-k}$
admits a limit when $z$ tends to $i \infty$ (we often just write $\infty$) for all $\mat abcd \in \SL_2(\ZZ)$ 
(this condition is called {\em $f$ is holomorphic at the cusp $a/c$}).
\end{enumerate}
We use the notation $\Mkone kN\CC$. If we replace (ii) by
\begin{enumerate}[(i)]
\item [(ii)'] the function $(f|_k \mat abcd) (z) = f(\frac{az+b}{cz+d}) (cz+d)^{-k}$
is holomorphic and
the limit $f(\frac{az+b}{cz+d}) (cz+d)^{-k}$ is $0$
when $z$ tends to $i \infty$,
\end{enumerate}
then $f$ is called a {\em cusp form}. For these,
we introduce the notation $\Skone kN\CC$.

Let us now suppose that we are given a Dirichlet character~$\chi$
of modulus~$N$ as above. Then we replace (i) as follows:
\begin{enumerate}[(i)]
\item [(i)'] $f(\frac{az+b}{cz+d}) (cz+d)^{-k} = \chi(d) f(z)$ for all 
$\mat abcd \in \Gamma_0(N)$.
\end{enumerate}
Functions satisfying this condition are called {\em modular forms}
(respectively, {\em cusp forms} if they satisfy (ii)') {\em of weight $k$,
character $\chi$ and level $N$}. The notation
$\Mk kN\chi\CC$ (respectively, $\Sk kN\chi\CC$) will be used.

All these are finite dimensional $\CC$-vector spaces. For $k \ge 2$, there are dimension formulae, which one can look up in \cite{Stein}.
We, however, point the reader to the fact that for $k=1$ nearly nothing about the dimension is known (except that it is smaller than the respective dimension for $k=2$; it is believed to be much smaller, but only very weak results are known to date).

\subsubsection*{Hecke operators}

At the base of everything that we will do with modular forms are the
Hecke operators and the diamond operators. One should really define
them more conceptually (e.g.\ geometrically), but this takes some time.
Here is a definition by formulae.

If $a$ is an integer coprime to~$N$, by Exercise~\ref{exsigmaa} we may let $\sigma_a$ be a matrix in $\Gamma_0(N)$ such that
\begin{equation}\label{sigmaa}
\sigma_a \equiv \mat {a^{-1}}00a \mod N.
\end{equation}

We define the {\em diamond operator} $\diam a$ (you see the diamond in the notation, with some phantasy) by the formula
$$ \diam a f = f|_k \sigma_a.$$
If $f \in \Mk kN\chi\CC$, then we have by definition $\diam a f = \chi(a) f$.
The diamond operators give a group action of $(\ZZ/N\ZZ)^\times$ on $\Mkone kN\CC$ and on $\Skone kN\CC$,
and the $\Mk kN\chi\CC$ and $\Sk kN\chi\CC$ are the $\chi$-eigenspaces for this action.
We thus have the isomorphism
$$ \Mkone kN\CC \cong \bigoplus_\chi \Mk kN\chi\CC$$
for $\chi$ running through the characters of $(\ZZ/N\ZZ)^\times$ (and similarly for the cuspidal spaces).

Let $\ell$ be a prime. We let 
\begin{align}\label{setrp}
\calR_\ell & := \{ \mat 1r0\ell | 0 \le r \le \ell-1 \} \cup \{\sigma_\ell \mat \ell 001\}, & \text{ if }\ell\nmid N\\
\calR_\ell & := \{ \mat 1r0\ell | 0 \le r \le \ell-1 \}, & \text{ if }\ell\mid N
\end{align}
We use these sets to define the {\em Hecke operator $T_\ell$} acting
on $f$ as above as follows:
$$ f |_k T_\ell := T_\ell f := \sum_{\delta \in \calR_\ell} f|_k\delta.$$

\begin{lemma}\label{lemhecke}
Suppose $f \in \Mk kN\chi\CC$. 
Recall that we have extended $\chi$ so that $\chi(\ell) = 0$ if $\ell$ divides~$N$. 
We have the formula
$$ a_n(T_\ell f) = a_{\ell n}(f) + \ell^{k-1} \chi(\ell) a_{n/\ell}(f).$$
In the formula, $a_{n/\ell}(f)$ is to be read as $0$ if $\ell$ does not divide $n$. 
\end{lemma}

\begin{proof}
Exercise~\ref{exhecke}.
\end{proof}

The Hecke operators for composite $n$ can be defined as follows (we put $T_1$ to be the identity):
\begin{equation}\label{eq:Tn-composite}
\begin{array}{rll}
T_{\ell^{r+1}} &= T_\ell \circ T_{\ell^r} - \ell^{k-1} \diam \ell T_{\ell^{r-1}} & \textnormal{ for all primes~$\ell$ and $r \ge 1$},\\
T_{uv} &= T_u \circ T_v & \textnormal{ for coprime positive integers $u,v$}.
\end{array}
\end{equation}

We derive the very important formula (valid for every $n$)
\begin{equation}\label{aeins}
a_1(T_n f) = a_n(f).
\end{equation}
It is the only formula that we will really need.

From Lemma~\ref{lemhecke} and the above formulae it is also evident that the Hecke operators commute among one another. 
By Exercise~\ref{excommute} eigenspaces for a collection of operators (i.e.\ each element of a given set of Hecke operators acts by scalar multiplication) are respected by all Hecke operators.
Hence, it makes sense to consider modular forms which are eigenvectors for every Hecke operator. These are called {\em Hecke eigenforms}, or often just {\em eigenforms}.
Such an eigenform $f$ is called {\em normalised} if $a_1(f) = 1$.
We shall consider eigenforms in more detail in the following section.

Finally, let us point out the formula (for $\ell$ prime and $\ell \equiv d \mod N$)
\begin{equation}\label{diamondinhecke}
\ell^{k-1} \diam d = T_\ell^2 - T_{\ell^2}.
\end{equation}
Hence, the diamond operators can be expressed as $\ZZ$-linear combinations of Hecke operators.
Note that divisibility is no trouble since we may choose $\ell_1$, $\ell_2$, both congruent to $d$ modulo $N$ satisfying an equation $1 = \ell_1^{k-1}r + \ell_2^{k-1}s$ for appropriate $r,s \in \ZZ$.

\subsubsection*{Hecke algebras and the $q$-pairing}

We now quickly introduce the concept of Hecke algebras. It will be treated in more detail in later sections.
In fact, {\em when we claim to compute modular forms with the modular symbols algorithm, we are really computing Hecke algebras.}
In the couple of lines to follow, we show that the Hecke algebra is the dual of modular forms, and hence all knowledge about modular forms can
- in principle - be derived from the Hecke algebra.

For the moment, we define the {\em Hecke algebra} of $\Mkone kN\CC$ as the sub-$\CC$-algebra
inside the endomorphism ring of the $\CC$-vector space $\Mkone kN\CC$ generated by all Hecke operators and all diamond operators.
We make similar definitions for $\Skone kN\CC$, $\Mk kN\chi\CC$ and $\Sk kN\chi\CC$.
Let us introduce the pieces of notation
$$\TT_\CC(\Mkone kN\CC), \TT_\CC(\Skone{k}{N}{\CC}), \TT_\CC(\Mk{k}{N}{\chi}{\CC}) \text{ and } \TT_\CC(\Sk{k}{N}{\chi}{\CC}),$$
respectively.
We now define a bilinear pairing, which we call the {\em (complex) $q$-pairing}, as
$$ \Mk kN\chi\CC \times \TT_\CC(\Mk kN\chi\CC) \to \CC, \;\;(f,T) \mapsto a_1(Tf)$$
(compare with Equation~\ref{aeins}).

\begin{lemma}\label{qpair}
Suppose $k \ge 1$.
The complex $q$-pairing is perfect, as is the analogous pairing for $\Sk kN\chi\CC$. In particular, 
$$ \Mk kN\chi\CC \cong \Hom_\CC (\TT_\CC(\Mk kN\chi\CC),\CC), \;\; f \mapsto (T \mapsto a_1(Tf))$$
and similarly for $\Sk kN\chi\CC$. For $\Sk kN\chi\CC$, the inverse is given by sending $\phi$ to $\sum_{n=1}^\infty \phi(T_n) q^n$.
\end{lemma}

\begin{proof}
Let us first recall that a pairing over a field is perfect
if and only if it is non-degenerate. That is what we are going to
check. It follows from Equation~\ref{aeins} like this.
If for all $n$ we have $0 = a_1(T_n f) = a_n(f)$, then
$f = 0$ (this is immediately clear for cusp forms; for
general modular forms at the first place we can only conclude
that $f$ is a constant, but since $k \ge 1$, non-zero constants
are not modular forms). Conversely, if
$a_1 (Tf) = 0$ for all $f$, then $a_1(T (T_n f)) = a_1 (T_n T f)
= a_n (Tf) = 0$ for all $f$ and all $n$, whence $Tf = 0$ for all $f$. 
As the Hecke algebra is defined as a subring in the endomorphism
of $\Mk kN\chi\CC$ (resp.\ the cusp forms), we find $T=0$,
proving the non-degeneracy.
\end{proof}

The perfectness of the $q$-pairing is also called the {\em existence of a $q$-expansion principle}.
Due to its central role for this course,
we repeat and emphasize that the Hecke algebra is the linear dual of the space of modular forms.

\begin{lemma}\label{eigenf}
Let $f$ in $\Mkone kN\CC$ be a normalised eigenform. Then
$$ T_n f = a_n(f) f \;\;\; \text{ for all } n \in \NN.$$
Moreover, the natural map from the above duality gives a bijection
\begin{equation*}
\{ \textnormal{Normalised eigenforms in }\Mkone kN\CC\} \leftrightarrow
\Hom_{\CC-\textnormal{algebra}} (\TT_\CC(\Mkone kN\CC),\CC).
\end{equation*}
Similar results hold, of course, also in the presence of~$\chi$.
\end{lemma}

\begin{proof}
Exercise~\ref{exeigenf}.
\end{proof}

\subsection{Theory: The modular symbols formalism}

In this section we give a definition of formal modular symbols, as implemented in {\sc Magma} and like the one in \cite{MerelUniversal},
\cite{Cremona} and~\cite{Stein}, except that we do not factor out torsion, but intend a common treatment for all rings.

Contrary to the texts just mentioned, we prefer to work with the group
$$ \PSL_2(\ZZ) = \SL_2(\ZZ) / \langle -1 \rangle,$$
since it will make some of the algebra much simpler and since
it has a very simple description as a free product (see later).
The definitions of modular forms could have been formulated
using $\PSL_2(\ZZ)$ instead of $\SL_2(\ZZ)$, too (Exercise~\ref{expsl}).

We introduce some definitions and pieces of notation to be used in all the course.

\begin{definition}
Let $R$ be a ring, $\Gamma$ a group and $V$ a left $R[\Gamma]$-module.
The $\Gamma$-invariants of~$V$ are by definition 
$$ V^\Gamma = \{ v \in V | g.v = v \; \forall \; g \in \Gamma \} \subseteq V.$$
The $\Gamma$-coinvariants of~$V$ are by definition
$$ V_\Gamma = V / \langle v - g.v |  g \in \Gamma, v \in V \rangle.$$
If $H \le \Gamma$ is a finite subgroup, we define the norm of~$H$ as
$$ N_H = \sum_{h \in H} h \in R[\Gamma].$$
Similarly, if $g \in \Gamma$ is an element of finite order~$n$, we define the norm of~$g$ as
$$ N_g = N_{\langle g \rangle} = \sum_{i=0}^{n-1} g^i \in R[\Gamma].$$
\end{definition}

Please look at the important Exercise~\ref{exgp} for some properties of these definitions.
We shall make use of the results of this exercise in the section on group cohomology.

For the rest of this section, we let $R$ be a commutative ring with unit and $\Gamma$ be a subgroup of finite index in $\PSL_2(\ZZ)$.
For the time being we allow general modules; so we let $V$ be a left $R[\Gamma]$-module.
Recall that $\PSL_2(\ZZ)$ acts on $\HH \cup \PP^1(\QQ)$ by Möbius transformations, as defined earlier.
A generalised version of the definition below appeared in~\cite{MS}.

\begin{definition}\label{defMS}
We define the $R$-modules
$$ \cM_R := R[\{\alpha,\beta\}| \alpha,\beta \in \PP^1(\QQ)]/
\langle \{\alpha,\alpha\}, \{\alpha,\beta\} + \{\beta,\gamma\} + \{\gamma,\alpha\}
| \alpha,\beta,\gamma \in \PP^1(\QQ)\rangle$$
and
$$ \cB_R := R[\PP^1(\QQ)].$$ 
We equip both with the natural left $\Gamma$-action. 
Furthermore, we let
$$ \cM_R(V) := \cM_R \otimes_R V \;\;\;\; \text{ and } \;\;\;\; \cB_R(V) := \cB_R \otimes_R V$$
for the left diagonal $\Gamma$-action.

\begin{enumerate}[(a)]
\item We call the $\Gamma$-coinvariants
$$ \cM_R (\Gamma,V) :=  \cM_R(V)_\Gamma = 
\cM_R(V)/ \langle (x - g x) | g \in \Gamma, x \in \cM_R(V) \rangle$$ 
{\em the space of $(\Gamma,V)$-modular symbols.}

\item We call the $\Gamma$-coinvariants
$$ \cB_R(\Gamma,V) :=  \cB_R(V)_\Gamma = 
\cB_R(V)/ \langle (x - g x) | g \in \Gamma, x \in \cB_R(V) \rangle$$ 
{\em the space of $(\Gamma,V)$-boundary symbols.}

\item We define the {\em boundary map} as the map
$$ \cM_R(\Gamma,V) \to \cB_R(\Gamma,V)$$
which is induced from the map $\cM_R \to \cB_R$ sending $\{\alpha, \beta\}$
to $\{\beta\} - \{\alpha\}$.

\item The kernel of the boundary map is denoted by $\cCM_R(\Gamma,V)$ and is called 
{\em the space of cuspidal $(\Gamma,V)$-modular symbols.}

\item The image of the boundary map inside $\cB_R(\Gamma,V)$ is
denoted by $\cE_R(\Gamma,V)$ and is called
{\em the space of $(\Gamma,V)$-Eisenstein symbols.}
\end{enumerate}
\end{definition}

The reader is now invited to prove that the definition of
$\cM_R(\Gamma,V)$ behaves well with respect to base change (Exercise~\ref{basechange}).

\subsubsection*{The modules $V_n(R)$ and $V_n^\chi(R)$}

Let $R$ be a ring.  We put $V_n(R) = R[X,Y]_n \cong \Sym^{n}(R^2)$ (see Exercise~\ref{exsym}).
By $R[X,Y]_n$ we mean the homogeneous polynomials of degree~$n$ in two variables with coefficients in the ring~$R$.
By $\matz$ we denote the monoid of integral $2\times 2$-matrices with non-zero determinant (for matrix multiplication),
{\em i.e.}, $\matz = \GL_2(\QQ) \cap \ZZ^{2 \times 2}$.
Then $V_n(R)$ is a $\matz$-module in several natural ways.

One can give it the structure of a left $\matz$-module via the polynomials by putting
$$ (\mat abcd .f) (X,Y) = f \big( (X,Y) \mat abcd \big) = f \big( (aX+cY, bX+dY) \big).$$
Merel and Stein, however, consider a different one, and that is the one implemented in {\sc Magma}, namely
$$ (\mat abcd .f) (X,Y) = f \big( (\mat abcd)^\iota \vect XY \big)
= f \big( \mat d{-b}{-c}a \vect XY \big) = f \big( \vect {dX-bY}{-cX+aY} \big).$$
Here, $\iota$ denotes Shimura's main involution whose definition
can be read off from the line above (note that $M^\iota$ is the inverse of $M$ if $M$ has determinant~$1$).
Fortunately, both actions are isomorphic due to the fact that the transpose of $(\mat abcd)^\iota \vect XY$ is equal
to $(X,Y) \sigma^{-1} \mat abcd \sigma$, where $\sigma = \mat 01{-1}0$.
More precisely, we have the isomorphism $V_n(R) \xrightarrow{f \mapsto \sigma^{-1}.f} V_n(R)$, 
where the left hand side module carries "our" action and the right hand side
module carries the other one. By $\sigma^{-1}.f$ we mean "our" $\sigma^{-1}.f$.

Of course, there is also a natural right action by $\matz$, namely
$$ (f.\mat abcd) (\vect XY) = f ( \mat abcd \vect XY ) = f(\vect{aX+bY}{cX+dY}).$$
By the standard inversion trick, also both left actions described above can be turned into right ones.

Let now $(\ZZ/N\ZZ)^\times \to R^\times$ be a Dirichlet character, which we shall also consider as a character
$\chi: \Gamma_0(N) \xrightarrow{\mat abcd \mapsto a} (\ZZ/N\ZZ)^\times \xrightarrow{\chi} R^\times$.
By $R^\chi$ we denote the $R[\Gamma_0(N)]$-module which
is defined to be $R$ with the $\Gamma_0(N)$-action through~$\chi$,
i.e.\ $\mat abcd .r = \chi(a) r = \chi^{-1}(d) r$ for $\mat abcd \in \Gamma_0(N)$ and $r \in R$.

For use with Hecke operators, we extend this action to matrices $\mat abcd \in \ZZ^{2 \times 2}$ which are
congruent to an upper triangular matrix modulo~$N$ (but not necessarily of determinant~$1$). Concretely,
we also put $\mat abcd .r = \chi(a) r$ for $r \in R$ in this situation. Sometimes, however, we want to use the coefficient~$d$
in the action. In order to do so, we let $R^{\iota,\chi}$ be $R$ with the action $\mat abcd .r = \chi(d) r$
for matrices as above.
In particular, the $\Gamma_0(N)$-actions on $R^{\iota,\chi}$ and $R^{\chi^{-1}}$ coincide.

Note that due to $(\ZZ/N\ZZ)^\times$ being an abelian group,
the same formulae as above make $R^\chi$ also into a right $R[\Gamma_0(N)]$-module.
Hence, putting
$$(f \otimes r).\mat abcd = (f|_k \mat abcd) \otimes \mat abcd r$$
makes $\Mkone kN\CC \otimes_\CC \CC^\chi$ into a right $\Gamma_0(N)$-module and we have the
description (Exercise~\ref{exchar})
\begin{equation}\label{eqchar}
\Mk kN\chi\CC = (\Mkone kN\CC \otimes_\CC \CC^\chi)^{(\ZZ/N\ZZ)^\times}
\end{equation}
and similarly for $\Sk kN\chi\CC$.

We let 
$$V_n^\chi(R) := V_n(R) \otimes_R R^\chi \textnormal{ and }V_n^{\iota,\chi}(R) := V_n(R) \otimes_R R^{\iota,\chi}$$
equipped with the natural diagonal left $\Gamma_0(N)$-actions.
Note that unfortunately these modules are in general not $\SL_2(\ZZ)$-modules, 
but we will not need that.
Note, moreover, that if $\chi(-1) = (-1)^n$, then minus the identity
acts trivially on $V_n^\chi(R)$ and $V_n^{\iota,\chi}(R)$, whence we consider these
modules also as $\Gamma_0(N)/\{\pm 1\}$-modules.

\subsubsection*{The modular symbols formalism for standard congruence subgroups}

We now specialise the general set-up on modular symbols that we
have used so far to the precise situation needed for establishing
relations with modular forms.

So we let $N \ge 1$, $k \ge 2$ be integers and fix a character
$\chi: (\ZZ/N\ZZ)^\times \to R^\times$, which we also sometimes
view as a group homomorphism $\Gamma_0(N) \to R^\times$ as above. We
impose that $\chi(-1) = (-1)^k$.

We define 
$$\cMk kN\chi R := \cM_R(\Gamma_0(N)/\{\pm 1\},V_{k-2}^\chi(R)),$$
$$\cCMk kN\chi R := \cCM_R(\Gamma_0(N)/\{\pm 1\},V_{k-2}^\chi(R)),$$
$$\cBk kN\chi R := \cB_R(\Gamma_0(N)/\{\pm 1\},V_{k-2}^\chi(R))$$
and
$$\cEk kN\chi R := \cE_R(\Gamma_0(N)/\{\pm 1\},V_{k-2}^\chi(R)).$$
We make the obvious analogous definitions for $\cMkone kNR$ etc.

Let
\begin{equation}\label{defeta}
\eta := \mat {-1}001. 
\end{equation}
Because of
$$ \eta \mat abcd \eta = \mat a {-b}{-c} d$$
we have
$$ \eta \Gamma_1(N) \eta = \Gamma_1(N) \;\;\; \text{ and } \;\;\;\eta \Gamma_0(N) \eta = \Gamma_0(N).$$
We can use the matrix $\eta$ to define an involution (also denoted by $\eta$) on the various modular symbols spaces. We just use the diagonal action on $\cM_R(V) := \cM_R \otimes_R V$, provided, of course, that $\eta$ acts on~$V$. On $V_{k-2}(R)$ we use the usual $\matz$-action, and on $V_{k-2}^\chi(R) = V_{k-2}(R) \otimes R^\chi$ we let $\eta$ only act on the first factor.
We will denote by the superscript ${}^+$ the subspace invariant under this involution, and by the superscript ${}^-$ the anti-invariant one. We point out that there are other very good
definitions of $+$-spaces and $-$-spaces. For instance, in many applications it can
be of advantage to define the $+$-space as the $\eta$-coinvariants, rather than
the $\eta$-invariants. In particular, for modular symbols, where we are using quotients
and coinvariants all the time, this alternative definition is more suitable. The reader
should just think about the differences between these two definitions.
Note that here we are not following the conventions of \cite{Stein}, p.~141. Our action just seems more natural than adding an extra minus sign.

\subsubsection*{Hecke operators}

The aim of this part is to state the definition of Hecke operators
and diamond operators on formal modular symbols $\cMk kN\chi R$
and $\cCMk kN\chi R$. One immediately sees that it is very similar
to the one on modular forms. One can get a different insight in the
defining formulae by seeing how they are derived from a double coset formulation
in section~\ref{sec:Hecke}.

The definition given here is also explained in detail in \cite{Stein}.
We should also mention the very important fact that one can transfer Hecke 
operators in an explicit way to Manin symbols using Heilbronn matrices.
We shall not do this explicitly in this course. This point is discussed in detail in \cite{Stein} and \cite{MerelUniversal}.

We now give the definition only for $T_\ell$ for a prime~$\ell$ and
the diamond operators. The $T_n$ for composite $n$ can be computed
from those by the formulae already stated in~\eqref{eq:Tn-composite}.
Notice that the $R[\Gamma_0(N)]$-action on $V_{k-2}^\chi(R)$
(for the usual conventions, in particular, $\chi(-1) = (-1)^k$)
extends naturally to an action of the semi-group generated
by $\Gamma_0(N)$ and $\calR_\ell$ (see Equation~\ref{setrp}).
Thus, this semi-group acts on $\cMk kN\chi R$ (and the
cusp space) by the diagonal action on the tensor product.
Let $x \in \cMkone kN R$ or $x \in \cMk kN \chi R$. We put
$$ T_\ell x = \sum_{\delta \in \calR_\ell} \delta.x.$$
If $a$ is an integer coprime to~$N$, we define the diamond operator as
$$ \diam a x = \sigma_a x $$
with $\sigma_a$ as in equation~\eqref{sigmaa}.
When $x = (m \otimes v \otimes 1)_{\Gamma_0(N)/\{\pm 1\}} \in \cMk kN \chi R$ for $m \in \calM_R$ and $v\in V_{k-2}$, we have
$\diam a x = (\sigma_a m \otimes \sigma_a v) \otimes \chi(a^{-1}))_{\Gamma_0(N)/\{\pm 1\}} = x$,
thus $(\sigma_a(m \otimes v) \otimes 1)_{\Gamma_0(N)/\{\pm 1\}} = \chi(a) (m \otimes v \otimes 1)_{\Gamma_0(N)/\{\pm 1\}}$.

As in the section on Hecke operators on modular forms, we define Hecke algebras on modular symbols in a very similar way. We will take the freedom of taking arbitrary base rings (we will do that for modular forms in the next section, too).

Thus for any ring $R$ we let $\TT_R (\cMkone knR)$ be the $R$-subalgebra
of the $R$-endomorphism algebra of the $R$-module $\cMkone knR$ generated by the Hecke
operators $T_n$. For a character $\chi: \ZZ/N\ZZ \to R^\times$, we make
a similar definition. We also make a similar definition for the
cuspidal subspace and the $+$- and $-$-spaces.

The following fact will be obvious from the description of modular symbols
as Manin symbols (see Theorem~\ref{ManinSymbols}), which will be derived in a later chapter.
Here, we already want to use it.

\begin{proposition}\label{factfg}
The $R$-modules $\cMkone kN R$, $\cCMkone kNR$, $\cMk kN\chi R$, $\cCMk kN\chi R$
are finitely presented.
\end{proposition}

\begin{corollary}
Let $R$ be Noetherian.
The Hecke algebras $\TT_R(\cMkone kN R)$, $\TT_R(\cCMkone kNR)$, 
$\TT_R(\cMk kN\chi R)$  and $\TT_R(\cCMk kN\chi R)$
are finitely presented $R$-modules.
\end{corollary}

\begin{proof}
This follows from Proposition~\ref{factfg} since
the endomorphism ring of a finitely generated module is finitely generated and
submodules of finitely generated modules over Noetherian rings are finitely generated.
Furthermore, over a Noetherian ring, finitely generated implies finitely presented.
\end{proof}

This very innocent looking corollary will give - together with the Eichler-Shimura
isomorphism - that coefficient fields of normalised eigenforms are number fields.
We next prove that the formation of Hecke algebras for modular symbols
behaves well with respect to flat base change.
We should have in mind the example $R=\ZZ$ or $R=\ZZ[\chi]:=\ZZ[\chi(n) : n \in \NN]$ (i.e., the ring of integers
of the cyclotomic extension of~$\QQ$ generated by the values of~$\chi$ or, equivalently, $\ZZ[e^{2 \pi i/r}]$ where $r$
is the order of~$\chi$) and $S=\CC$.

\begin{proposition}\label{hamsymbc}
Let $R$ be a Noetherian ring and $R \to S$ a flat ring
homomorphism. 
\begin{enumerate}[(a)]
\item The natural map
$$ \TT_R (\cMkone kNR) \otimes_R S \cong \TT_S (\cMkone kNS)$$
is an isomorphism of $S$-algebras.
\item The natural map
$$ \Hom_R(\TT_R(\cMkone kNR),R) \otimes_R S \cong \Hom_S(\TT_S(\cMkone kNS),S)$$
is an isomorphism of $S$-modules.
\item The map
$$ \Hom_R(\TT_R(\cMkone kNR) ,S) 
\xrightarrow{\phi \mapsto (T \otimes s \mapsto \phi(T)s)} \Hom_S(\TT_R(\cMkone kNR) \otimes_R S,S)$$
is also an isomorphism of $S$-modules.
\item Suppose in addition that $R$ is an integral domain and $S$ a field
extension of the field of fractions of~$R$. Then the natural map
$$ \TT_R(\cMkone kNR) \otimes_R S \to \TT_R(\cMkone kNS) \otimes_R S$$
is an isomorphism of $S$-algebras.
\end{enumerate}
For a character $\chi: (\ZZ/N\ZZ)^\times \to R^\times$, similar results hold.
Similar statements also hold for the cuspidal subspace.
\end{proposition}

\begin{proof}
We only prove the proposition for $M := \cMkone kNR$. The arguments
are exactly the same in the other cases.

(a) By Exercise~\ref{basechange} it suffices to prove
$$ \TT_R (M) \otimes_R S \cong \TT_S (M \otimes_R S).$$
Due to flatness and the finite presentation of~$M$ the natural homomorphism
$$ \End_R(M) \otimes_R S \to \End_S(M\otimes_R S)$$
is an isomorphism (see \cite{Eisenbud}, Prop.~2.10). By definition,
the Hecke algebra $\TT_R(M)$ is an $R$-submodule of $\End_R(M)$. As injections are preserved
by flat morphisms, we obtain the injection
$$\TT_R(M) \otimes_R S \hookrightarrow \End_R(M) \otimes_R S \cong \End_S(M\otimes_R S).$$
The image is equal to $\TT_S(M\otimes_R S)$, since all Hecke operators
are hit, establishing~(a).

(b) follows from the same citation from \cite{Eisenbud} as above. 

(c) Suppose that under the map from Statement~(c)
$\phi \in \Hom_R(\TT_R(M) ,S)$ is mapped to the zero map. 
Then $\phi(T)s=0$ for all $T$ and all $s \in S$. 
In particular with $s=1$ we get $\phi(T)=0$ for all $T$,
whence $\phi$ is the zero map, showing injectivity.
Suppose now that $\psi \in \Hom_S(\TT_R(M) \otimes_R S,S)$ is given. 
Call $\phi$ the composite 
$\TT_R(M) \to \TT_R(M) \otimes_R S \xrightarrow{\psi} S$. Then $\psi$
is the image of~$\phi$, showing surjectivity.

(d) We first define 
$$N := \ker \big( M \xrightarrow{\pi: m \mapsto m \otimes 1} M \otimes_R S\big).$$
We claim that $N$ consists only of $R$-torsion elements.
Let $x \in N$. Then $x \otimes 1 = 0$. If $rx \neq 0$ for all $r \in R - \{0\}$,
then the map $R \xrightarrow {r \mapsto rx} N$ is injective. We call $F$ the image
to indicate that it is a free $R$-module. Consider the exact sequence of $R$-modules:
$$ 0 \to F \to M \to M/F \to 0.$$
From flatness we get the exact sequence
$$ 0 \to F \otimes_R S \to M \otimes_R S \to M/F \otimes_R S \to 0.$$
But, $F \otimes_R S$ is $0$, since it is generated by $x\otimes 1 \in M \otimes_R S$.
However, $F$ is free, whence $F \otimes_R S$ is also~$S$. This contradiction shows
that there is some $r \in R-\{0\}$ with $rx = 0$.

As $N$ is finitely generated, there is some $r \in R-\{0\}$ such that $rN = 0$.
Moreover, $N$ is characterised as the set of elements $x \in M$ such that $rx=0$.
For, we already know that $x \in N$ satisfies $rx = 0$. If, conversely, $rx = 0$
with $x \in M$, then $0 = rx \otimes 1/r = x \otimes 1 \in M \otimes_R S$. 

Every $R$-linear (Hecke) operator $T$ on $M$ clearly restricts to~$N$, 
since $r Tn = T rn=T0=0$.
Suppose now that $T$ acts as~$0$ on~$M \otimes_R S$. 
We claim that then $rT = 0$ on all of~$M$.
Let $m \in M$. We have $0 = T \pi m = \pi Tm$. Thus $Tm \in N$ and, so, 
$rTm = 0$, as claimed. In other words, the kernel of the homomorphism
$\TT_R(M) \to \TT_R(M \otimes_R S)$ is killed by~$r$. 
This homomorphism is surjective, 
since by definition $\TT_R(M \otimes_R S)$ is generated by all Hecke operators
acting on~$M \otimes_R S$. Tensoring with $S$ kills the torsion and the statement follows.
\end{proof}

Some words of warning are necessary. 
It is essential that $R \to S$ is a flat homomorphism. A similar result for
$\ZZ \to \FF_p$ is not true in general. I call this a "faithfulness problem",
since then $\cMkone kN {\FF_p}$ is not a faithful module for 
$\TT_\ZZ(\cMkone kN\CC)  \otimes_\ZZ {\FF_p}$. Some effort goes into finding
$k$ and $N$, where this module is faithful. See, for instance, \cite{faithful}.
Moreover, $\cMkone kN R$ need not be a free $R$-module and can contain torsion.
Please have a look at Exercise~\ref{exhamsymbc} now to find out whether one
can use the $+$- and the $-$-space in the proposition.

\subsection{Theory: The modular symbols algorithm}

\subsubsection*{The Eichler-Shimura theorem}

At the basis of the modular symbols algorithm is the following theorem
by Eichler, which was extended by Shimura. One of our aims in this lecture
is to provide a proof for it. In this introduction, however, we only state
it and indicate how the modular symbols algorithm can be derived from it.

\begin{theorem}[Eichler-Shimura]\label{thmes}
There are isomorphisms respecting the Hecke operators
\begin{enumerate}[(a)]
\item $\Mk kN\chi\CC) \oplus \Sk kN\chi\CC^\vee \cong \cMk kN\chi\CC,$
\item $\Sk kN\chi\CC) \oplus \Sk kN\chi\CC^\vee \cong \cCMk kN\chi\CC,$
\item $\Sk kN\chi\CC \cong \cCMk kN\chi\CC^+.$
\end{enumerate}
Similar isomorphisms hold for modular forms and modular symbols on $\Gamma_1(N)$
and $\Gamma_0(N)$.
\end{theorem}

\begin{proof}
Later in this lecture (Theorems \ref{compthm} and~\ref{esgammaeins}, Corollary~\ref{esgammanull}).
\end{proof}

\begin{corollary}\label{algz1}
Let $R$ be a subring of $\CC$ and $\chi: (\ZZ/N\ZZ)^\times \to R^\times$
a character. Then there is the natural isomorphism
$$ \TT_R(\Mk kN\chi\CC) \cong \TT_R(\cMk kN\chi\CC).$$
A similar result holds cusp forms, and also for $\Gamma_1(N)$ without a character as well as for $\Gamma_0(N)$.
\end{corollary}

\begin{proof}
We only prove this for the full space of modular forms. The arguments in the other cases are very similar.
Theorem~\ref{thmes} tells us that the $R$-algebra generated by the
Hecke operators inside the endomorphism ring of $\Mk kN\chi\CC$ equals the $R$-algebra 
generated by the Hecke operators inside the endomorphism ring of $\cMk kN\chi\CC$, {\em i.e.}
the assertion to be proved.
To see this, one just needs to see that the algebra generated by all Hecke operators on 
$\Mk kN\chi\CC \oplus \Sk kN\chi\CC^\vee$ 
is the same as the one generated by all Hecke operators on $\Mk kN\chi\CC$, which follows
from the fact that if some Hecke operator $T$ annihilates the full space of modular forms, then
it also annihilates the dual of the cusp space.
\end{proof}

The following corollary of the Eichler-Shimura theorem is of utmost importance
for the theory of modular forms. It says that Hecke algebras of modular forms
have an integral structure (take $R=\ZZ$ or $R =\ZZ [\chi]$). We will say more
on this topic in the next section.

\begin{corollary}\label{algz}
Let $R$ be a subring of $\CC$ and $\chi: (\ZZ/N\ZZ)^\times \to R^\times$
a character. Then the natural map
$$\TT_R(\Mk kN\chi\CC) \otimes_R \CC  \cong \TT_\CC (\Mk kN\chi \CC)$$
is an isomorphism.
A similar result holds cusp forms, and also for $\Gamma_1(N)$ without a character as well as for $\Gamma_0(N)$.
\end{corollary}

\begin{proof}
We again stick to the full space of modular forms. Tensoring the isomorphism from Corollary~\ref{algz1} with $\CC$ we get
$$\TT_R(\Mk kN\chi\CC) \otimes_R \CC \cong \TT_R(\cMk kN\chi\CC) \otimes_R \CC 
\cong \TT_\CC(\cMk kN\chi\CC) \cong \TT_\CC(\Mk kN\chi\CC),$$
using Proposition~\ref{hamsymbc}~(d) and again Theorem~\ref{thmes}.
\end{proof}

The next corollary is at the base of the modular symbols algorithm, since it describes modular forms in linear algebra
terms involving only modular symbols.

\begin{corollary}\label{cores}
Let $R$ be a subring of $\CC$ and $\chi: (\ZZ/N\ZZ)^\times \to R^\times$
a character. Then we have the isomorphisms
\begin{align*}
\Mk kN\chi\CC  & \cong \Hom_R(\TT_R (\cMk kN\chi R), R) \otimes_R \CC \\
               & \cong \Hom_R(\TT_R (\cMk kN\chi R), \CC) & \textnormal{ and} \\
\Sk kN\chi\CC  & \cong \Hom_R(\TT_R (\cCMk kN\chi R), R) \otimes_R \CC \\
               & \cong \Hom_R(\TT_R (\cCMk kN\chi R), \CC).
\end{align*}
Similar results hold for $\Gamma_1(N)$ without a character and also for $\Gamma_0(N)$.
\end{corollary}

\begin{proof}
This follows from Corollaries~\ref{algz1}, \ref{algz}, Proposition~\ref{hamsymbc} and Lemma~\ref{qpair}.
\end{proof}

Please look at Exercise~\ref{excores} to find out which statement
should be included into this corollary concerning the $+$-spaces.
Here is another important consequence of the Eichler-Shimura theorem.

\begin{corollary}
Let $f = \sum_{n=1}^\infty a_n(f) q^n \in \Skone kN\CC$ be a normalised Hecke ei\-genform. Then $\QQ_f := \QQ(a_n(f) | n \in \NN)$ is
a number field of degree less than or equal to $\dim_\CC \Skone kN\CC$.

If $f$ has Dirichlet character~$\chi$, then $\QQ_f$ is a finite field
extension of $\QQ(\chi)$ of degree less than or equal to $\dim_\CC \Sk kN\chi\CC$.
Here $\QQ(\chi)$ is the extension of $\QQ$ generated by all the values of~$\chi$.
\end{corollary}

\begin{proof}
It suffices to apply the previous corollaries with $R = \QQ$ or $R = \QQ(\chi)$ and
to remember that normalised Hecke eigenforms correspond to algebra homomorphisms
from the Hecke algebra into~$\CC$.
\end{proof}

\subsubsection*{Sketch of the modular symbols algorithm}

It may now already be quite clear how the modular symbols algorithm for computing
cusp forms proceeds. We give a very short sketch.

\begin{algorithm}\label{algmodsymsketch}
\noindent \underline{Input:} A field $K \subset \CC$, integers $N \ge 1$, $k \ge 2$, $P$,
a character $\chi: (\ZZ/N\ZZ)^\times \to K^\times$.

\noindent \underline{Output:} A basis of the space of cusp forms
$\Sk kN\chi\CC$; each form is given by its standard $q$-expansion with
precision~$P$.

\begin{enumerate}[(1)]
\item create $M := \cCMk kN\chi K$.
\item $L \leftarrow []$ (empty list), $n \leftarrow 1$.
\item repeat
\item \ms compute $T_n$ on $M$.
\item \ms join $T_n$ to the list $L$.
\item \ms $\TT \leftarrow$ the $K$-algebra generated by all $T \in L$.
\item \ms $n \leftarrow n+1$
\item until $\dim_K (\TT) = \dim_\CC \Sk kN\chi\CC$
\item compute a $K$-basis $B$ of $\TT$.
\item compute the basis $B^\vee$ of $\TT^\vee$ dual to~$B$.
\item for $\phi$ in $B^\vee$ do
\item \ms output $\sum_{n=1}^P \phi (T_n) q^n \in K[q]$.
\item end for.
\end{enumerate}
\end{algorithm}

We should make a couple of remarks concerning this algorithm.
Please remember that there are dimension formulae for $\Sk kN\chi\CC$, which can be looked up in~\cite{Stein}.
It is clear that the repeat-until loop will stop, due to Corollary~\ref{cores}.
We can even give an upper bound as to when it stops at the latest.
That is the so-called Sturm bound, which is the content of the following proposition.

\begin{proposition}[Sturm]\label{sturm}
Let $f \in \Mk kN\chi\CC$ such that $a_n(f) = 0$ for all $n \le \frac{k\mu}{12}$,
where $\mu = N \prod_{l \mid N \textnormal{ prime}} (1 + \frac{1}{l})$.

Then $f=0$.
\end{proposition}

\begin{proof}
Apply Corollary 9.20 of \cite{Stein} with $\fm = (0)$.
\end{proof}

\begin{corollary}\label{corsturm}
Let $K, N,\chi$ etc.\ as in the algorithm. Then $\TT_K(\cCMk kN\chi K)$
can be generated as a $K$-vector space by the operators $T_n$ for $1  \le n \le \frac{k \mu}{12}$.
\end{corollary}

\begin{proof}
Exercise~\ref{excorsturm}.
\end{proof}

We shall see later how to compute eigenforms and how to decompose the space
of modular forms in a "sensible" way.

\subsection{Theory: Number theoretic applications}

We close this survey and motivation section by sketching some number theoretic applications.

\subsubsection*{Galois representations attached to eigenforms}

We mention the sad fact that until 2006 only the one-dimensional representations
of $\Gal(\Qbar/\QQ)$ were well understood. In the case of finite image
one can use the Kronecker-Weber theorem, which asserts that any cyclic
extension of $\QQ$ is contained in a cyclotomic field. This is
generalised by global class field theory to one-dimensional
representations of $\Gal(\Qbar/K)$ for each number field $K$.
Since we now have a proof of Serre's modularity conjecture \cite{Serre}
(a theorem by Khare, Wintenberger \cite{KhareWintenberger}),
we also know a little bit about $2$-dimensional representations
of $\Gal(\Qbar/\QQ)$, but, replacing $\QQ$ by any other number field, all
one has is conjectures.

The great importance of modular forms for modern number theory is due
to the fact that one may attach a $2$-dimensional representation of
the Galois group of the rationals to each normalised cuspidal
eigenform. The following theorem is due to Shimura for $k=2$ and due
to Deligne for $k \ge 2$.

Until the end of this section, we shall use the language of Galois representations
(e.g.\ irreducible, unramified, Frobenius element, cyclotomic character) without
introducing it. It will not be used elsewhere.
The meanwhile quite old lectures by Darmon, Diamond and Taylor
are still an excellent introduction to the subject \cite{DarmonDiamondTaylor}.

\begin{theorem}\label{deligneqp}
Let $k \ge 2$, $N \ge 1$, $p$ a prime,
and $\chi: (\ZZ/N\ZZ)^\times \to \CC^\times$ a character.

Then to any normalised eigenform $f \in \Sk kN\chi\CC$
with $f = \sum_{n\ge 1} a_n(f) q^n$ one can attach a Galois
representation, i.e.\ a continuous group homomorphism,
$$\rho_f: \Gal(\Qbar/\QQ) \to \GL_2(\Qbar_p)$$
such that
\begin{enumerate}[(i)]
\item $\rho_f$ is irreducible,
\item $\det(\rho_f (c)) = -1$ for any complex conjugation $c \in \Gal(\Qbar/\QQ)$
(one says that $\rho_f$ is {\em odd}),
\item for all primes $\ell \nmid Np$ the representation $\rho_f$
is unramified at~$\ell$, 
$$\tr(\rho_f(\Frob_\ell)) = a_\ell(f) \;\; \text{ and } \;\;
\Det(\rho_f(\Frob_\ell)) = \ell^{k-1} \chi(\ell).$$
In the statement, $\Frob_\ell$ denotes a Frobenius element at~$\ell$.
\end{enumerate}
\end{theorem}

By choosing a $\rho(\Gal(\Qbar/\QQ))$-stable lattice in $\Qbar_p^2$
and applying reduction and semi-simplification one obtains
the following consequence.

\begin{theorem}\label{delignefp}
Let $k \ge 2$, $N \ge 1$, $p$ a prime,
and $\chi: (\ZZ/N\ZZ)^\times \to \CC^\times$ a character.

Then to any normalised eigenform $f \in \Sk kN\chi\CC$
with $f = \sum_{n\ge 1} a_n(f) q^n$ and to any prime ideal $\fP$
of the ring of integers $\calO_f$ of $\QQ_f = \QQ(a_n(f) : n \in \NN)$ with
residue characteristic~$p$ (and silently a fixed embedding $\calO_f/\fP \hookrightarrow \Fbar_p$), one can attach a Galois
representation, i.e.\ a continuous group homomorphism
(for the discrete topology on $\GL_2(\Fbar_p)$),
$$\rho_f: \Gal(\Qbar/\QQ) \to \GL_2(\Fbar_p)$$
such that
\begin{enumerate}[(i)]
\item $\rho_f$ is semi-simple,
\item $\det(\rho_f (c)) = -1$ for any complex conjugation $c \in \Gal(\Qbar/\QQ)$
(one says that $\rho_f$ is {\em odd}),
\item for all primes $\ell \nmid Np$ the representation $\rho_f$
is unramified at~$\ell$, 
$$\tr(\rho_f(\Frob_\ell)) \equiv a_\ell(f) \mod \fP \;\; \text{ and } \;\;
\det(\rho_f(\Frob_\ell)) \equiv \ell^{k-1} \chibar(\ell) \mod \fP.$$
\end{enumerate}
\end{theorem}

\subsubsection*{Translation to number fields}

\begin{proposition}\label{overnf}
Let $f$, $\QQ_f$, $\fP$ and $\rho_f$ be as in Theorem~\ref{delignefp}.
Then the following hold:
\begin{enumerate}[(a)]
\item The image of $\rho_f$ is finite and its image is contained in $\GL_2(\FF_{p^r})$ for some~$r$.
\item The kernel of $\rho_f$ is an open subgroup of $\Gal(\Qbar/\QQ)$ and is hence
of the form $\Gal(\Qbar/K)$ for some Galois number field~$K$. Thus, we can and do
consider $\Gal(K/\QQ)$ as a subgroup of $\GL_2(\FF_{p^r})$.
\item The characteristic polynomial of $\Frob_\ell$ (more precisely, of $\Frob_{\Lambda/\ell}$
for any prime $\Lambda$ of $K$ dividing $\ell$) is equal to $X^2 - a_\ell(f) X + \chi(\ell) \ell^{k-1} \mod \fP$
for all primes $\ell \nmid Np$.
\end{enumerate}
\end{proposition}

\begin{proof}
Exercise~\ref{exovernf}.
\end{proof}

To appreciate the information obtained from the $a_\ell(f) \mod \fP$, the reader is invited to do Exercise~\ref{exinfochar} now.

\subsubsection*{Images of Galois representations}

One can also often tell what the Galois group $\Gal(K/\QQ)$ is as an abstract group.
There are not so many possibilities, as we see from the following theorem.

\begin{theorem}[Dickson]
Let $p$ be a prime and $H$ a finite subgroup of $\PGL_2(\Fbar_p)$.
Then a conjugate of $H$ is isomorphic to one of the following groups:
\begin{itemize}
\item finite subgroups of the upper triangular matrices,
\item $\PSL_2(\FF_{p^r})$ or $\PGL_2(\FF_{p^r})$ for $r \in \NN$,
\item dihedral groups $D_r$ for $r \in \NN$ not divisible by $p$,
\item $A_4$, $A_5$ or $S_4$.
\end{itemize}
\end{theorem}

For modular forms there are several results mostly by Ribet concerning the
groups that occur as images \cite{Ribet}.
Roughly speaking, they say that the image is
`as big as possible' for almost all $\fP$ (for a given~$f$). For modular forms
without CM and inner twists (we do not define these notions in this course)
this means that if $G$ is the image, then $G$ modulo scalars
is equal to $\PSL_2(\FF_{p^r})$ or $\PGL_2(\FF_{p^r})$, where $\FF_{p^r}$
is the extension of $\FF_p$ generated by the $a_n(f) \mod \fP$.

An interesting question is to study which groups (i.e.\ which $\PSL_2(\FF_{p^r})$) actually occur.
It would be nice to prove that all of them do, since - surprisingly -
the simple groups $\PSL_2(\FF_{p^r})$ are still resisting a lot to all efforts to
realise them as Galois groups over $\QQ$ in the context of inverse Galois theory.

\subsubsection*{Serre's modularity conjecture}

Serre's modularity conjecture is the following.
Let $p$ be a prime and $\rho: \Gal(\Qbar/\QQ) \to \GL_2(\Fbar_p)$
be a continuous, odd, irreducible representation.

\begin{itemize}
\item Let $N_\rho$ be the (outside of~$p$) conductor of~$\rho$ (defined by
a formula analogous to the formula for the Artin conductor, except
that the local factor for~$p$ is dropped).
\item Let $k_\rho$ be the integer defined by~\cite{Serre}.
\item Let $\chi_\rho$ be the prime-to-$p$ part of $\det \circ \rho$
considered as a character 
$(\ZZ/N_\rho\ZZ)^\times \times (\ZZ/p\ZZ)^\times \to \Fbar_p^\times$.
\end{itemize}

\begin{theorem}[Khare, Wintenberger, Kisin: Serre's Modularity Conjecture]
Let $p$ be a prime and $\rho: \Gal(\Qbar/\QQ) \to \GL_2(\Fbar_p)$
be a continuous, odd, irreducible representation.

Then there exists a normalised eigenform
$$f \in \Sk {k_\rho}{N_\rho}{\chi_\rho}{\CC}$$
such that $\rho$ is isomorphic to the Galois representation
$$\rho_f: \Gal(\Qbar/\QQ) \to \GL_2(\Fbar_p)$$
attached to~$f$ by Theorem~\ref{delignefp}.
\end{theorem}

Serre's modularity conjecture implies that we can compute
(in principle, at least) arithmetic properties of all Galois
representations of the type in Serre's conjecture by 
computing the mod~$p$ Hecke eigenforms they come from.
Conceptually, Serre's modularity conjecture gives an explicit description of
all irreducible, odd and continuous `mod $p$' representations
of $\Gal(\Qbar/\QQ)$ and, thus, in a sense generalises class field theory.

Edixhoven et al.\ \cite{Edixhoven} have succeeded in giving
an algorithm which computes the actual Galois representation
attached to a mod~$p$ modular form. Hence, with Serre's conjecture
we have a way of - in principle - obtaining all information on
$2$-dimensional irreducible, odd continuous representations of $\Gal(\Qbar/\QQ)$.

\subsection{Theory: Exercises}

\begin{exercise}\label{exsln} 
\begin{enumerate}[(a)]
\item The group homomorphism
$$ \SL_2(\ZZ) \to \SL_2(\ZZ/N\ZZ)$$
given by reducing the matrices modulo~$N$ is surjective.
\item Check the bijections 
$$\SL_2(\ZZ)/\Gamma_1(N) = \{ \vect a c | \langle a,c \rangle = \ZZ/N\ZZ\}$$
and
$$\SL_2(\ZZ)/\Gamma_0(N) = \PP^1 (\ZZ/N\ZZ),$$
which were given in the beginning.
\end{enumerate}
\end{exercise}

\begin{exercise}\label{qzetan}
Let $N$ be an integer and $\zeta_N \in \CC$ any primitive $N$-th root of unity.
Prove that the map
$$ \Gal(\QQ(\zeta_N)/\QQ) \xrightarrow{\Frob_\ell \mapsto \ell} (\ZZ/N\ZZ)^\times$$
(for all primes $\ell \nmid N$) is an isomorphism. 
\end{exercise}

\begin{exercise}\label{exsigmaa}
Prove that a matrix $\sigma_a$ as in Equation~\ref{sigmaa} exists.
\end{exercise}

\begin{exercise}\label{exhecke}
Prove Lemma~\ref{lemhecke}. See also \cite[Proposition~5.2.2]{DS}.
\end{exercise}

\begin{exercise}\label{excommute}
\begin{enumerate}[(a)]
\item Let $K$ be a field, $V$ a vector space and $T_1,T_2$ two
commuting endomorphisms of~$V$, i.e.\ $T_1 T_2 = T_2 T_1$.
Let $\lambda_1 \in K$ and consider the $\lambda_1$-eigenspace
of~$T_1$, i.e.\ $V_1 = \{ v | T_1 v = \lambda_1 v\}$. 
Prove that $T_2 V_1 \subseteq V_1$.
\item Suppose that $\Mkone Nk\CC$ is non-zero. Prove that it contains
a Hecke eigenform.
\end{enumerate}
\end{exercise}

\begin{exercise}\label{exeigenf}
Prove Lemma~\ref{eigenf}.

Hint: use the action of Hecke operators explicitly described on $q$-expansions.
\end{exercise}

\begin{exercise} \label{expsl}
Check that it makes sense to replace $\SL_2(\ZZ)$ by $\PSL_2(\ZZ)$
in the definition of modular forms.

Hint: for the transformation rule: if $-1$ is not in the congruence subgroup in question, there is nothing to show;
if $-1$ is in it, one has to verify that it acts trivially. Moreover convince yourself that the holomorphy at the cusps does not depend
on replacing a matrix by its negative.
\end{exercise}

\begin{exercise}\label{exgp}
Let $R$ be a ring, $\Gamma$ a group and $V$ a left $R[\Gamma]$-module.
\begin{enumerate}[(a)]
\item Define the augmentation ideal $I_\Gamma$ by the exact sequence
$$ 0 \to I_\Gamma \to R[\Gamma] \xrightarrow{\gamma \mapsto 1} R \to 1.$$
Prove that $I_\Gamma$ is the ideal in $R[\Gamma]$ generated by the
elements $1-g$ for $g \in \Gamma$.
\item Conclude that $V_\Gamma = V / I_\Gamma V$.
\item Conclude that $V_\Gamma \cong R \otimes_{R[\Gamma]} V$.
\item Suppose that $\Gamma = \langle T \rangle$ is a cyclic group
(either finite or infinite (isomorphic to $(\ZZ,+)$)).
Prove that $I_\Gamma$ is the ideal generated by $(1-T)$.
\item Prove that $V^\Gamma \cong \Hom_{R[\Gamma]}(R,V)$.
\end{enumerate}
\end{exercise}

\begin{exercise}\label{basechange}
Let $R$, $\Gamma$ and $V$ as in Definition~\ref{defMS} and let $R \to S$ be a ring homomorphism.
\begin{enumerate}[(a)]
\item Prove that
$$\cM_R(\Gamma,V) \otimes_R S \cong \cM_S(\Gamma,V \otimes_R S).$$
\item Suppose $R \to S$ is flat. Prove a similar statement for the cuspidal subspace.
\item Are similar statements true for the boundary or the Eisenstein space?
What about the $+$- and the $-$-spaces?
\end{enumerate}
\end{exercise}

\begin{exercise}\label{exsym}
Prove that the map
$$ \Sym^{n}(R^2) \to R[X,Y]_n, \;\;\;
\vect {a_1}{b_1} \otimes \dots \otimes \vect{a_n}{b_n} \mapsto (a_1X + b_1Y) \cdots (a_nX + b_nY)$$
is an isomorphism,
where $\Sym^n(R^2)$ is the $n$-th symmetric power of~$R^2$, which is defined as
the quotient of $\underbrace{R^2 \otimes_R \dots \otimes_R R^2}_{n\textnormal{-times}}$
by the span of all elements 
$v_1 \otimes \dots \otimes v_n - v_{\sigma(1)} \otimes \dots \otimes v_{\sigma(n)}$
for all $\sigma$ in the symmetric group on the letters $\{1,2,\dots,n\}$.
\end{exercise}

\begin{exercise}\label{exchar}
Prove Equation~\ref{eqchar}.
\end{exercise}

\begin{exercise}\label{exhamsymbc}
Can one use $+$- or $-$-spaces in Proposition~\ref{hamsymbc}?
What could we say if we defined the $+$-space as $M/(1-\eta) M$ with $M$
standing for some space of modular symbols?
\end{exercise}

\begin{exercise}\label{excores}
Which statements in the spirit of Corollary~\ref{cores}~(b) are true for the
$+$-spaces?
\end{exercise}

\begin{exercise}\label{excorsturm}
Prove Corollary~\ref{corsturm}.
\end{exercise}

\begin{exercise}\label{exovernf}
Prove Proposition~\ref{overnf}.
\end{exercise}

\begin{exercise}\label{exinfochar}
In how far is a conjugacy class in $\GL_2(\FF_{p^r})$ determined by
its characteristic polynomial?
Same question as above for a subgroup $G \subset \GL_2(\FF_{p^r})$.
\end{exercise}

\subsection{Computer exercises}

\begin{cexercise}
\begin{enumerate}[(a)]
\item Create a list $L$ of all primes in between 234325 and 3479854?
How many are there?
\item For $n=2,3,4,5,6,7,997$ compute for each $a \in \ZZ/ n\ZZ$ how often
it appears as a residue in the list $L$.
\end{enumerate}
\end{cexercise}

\begin{cexercise}
In this exercise you verify the validity of the prime number theorem.
\begin{enumerate}[(a)]
\item Write a function {\tt NumberOfPrimes} with the following specifications.
Input: Positive integers $a,b$ with $a \le b$.
Output: The number of primes in $[a,b]$.
\item Write a function {\tt TotalNumberOfPrimes} with the following specifications.
Input: Positive integers $x, s$.
Output: A list $[n_1,n_2,n_3,\dots,n_m]$ such that $n_i$ is the number of primes
between $1$ and $i\cdot s$ and $m$ is the largest integer smaller than or equal to
$x/s$.
\item Compare the output of {\tt TotalNumberOfPrimes} with the predictions
of the prime number theorem: Make a function that returns the list $[r_1,r_2,\dots,r_m]$ 
with $r_i = \frac{si}{\log{si}}$. Make a function that computes the quotient of
two lists of "numbers".
\item Play with these functions. What do you observe?
\end{enumerate}
\end{cexercise}

\begin{cexercise}
Write a function {\tt ValuesInField} with: Input: a unitary polynomial $f$ with integer
coefficients and $K$ a finite field. Output: the set of values of $f$ in $K$.
\end{cexercise}

\begin{cexercise}
\begin{enumerate}[(a)]
\item Write a function {\tt BinaryExpansion} that computes the binary expansion
of a positive integer. Input: positive integer~$n$. Output: list of $0$'s and $1$'s
representing the binary expansion.
\item Write a function {\tt Expo} with: Input: two positive integers $a,b$.
Output $a^b$. You must not use the in-built function $a^b$, but write a sensible
algorithm making use of the binary expansion of~$b$. The only arithmetic operations allowed
are multiplications.
\item Write similar functions using the expansion with respect to a general base~$d$.
\end{enumerate}
\end{cexercise}

\begin{cexercise}
In order to contemplate recursive algorithms, the monks in Hanoi used
to play the following game. First they choose a degree of contemplation,
i.e.\ a positive integer~$n$. Then they create three lists:
$$L_1 := [n,n-1,\dots,2,1]; L_2 := []; L_3 := [];$$
The aim is to exchange $L_1$ and $L_2$. However, the monks may only perform
the following step: Remove the last element from one of the lists and
append it to one of the other lists, subject to the important condition
that in all steps all three lists must be descending.

Contemplate how the monks can achieve their goal. Write a procedure
with input~$n$ that plays the game. After each step,
print the number of the step, the three lists and test whether all
lists are still descending.

[Hint: For recursive procedures, i.e.\ procedures calling themselves,
in {\sc Magma} one must put the command {\tt forward my\_procedure} in front of the
definition of {\tt my\_procedure}.]
\end{cexercise}

\begin{cexercise}
This exercise concerns the normalised cuspidal eigenforms in weight~$2$ and level~$23$.
\begin{enumerate}[(a)]
\item What is the number field $K$ generated by the coefficients of each of the two
forms?
\item Compute the characteristic polynomials of the first 100 Fourier coefficients
of each of the two forms.
\item Write a function that for a given prime $p$ computes the reduction modulo~$p$
of the characteristic polynomials from the previous point and their factorisation.
\item Now use modular symbols over $\FF_p$ for a given $p$. Compare the results.
\item Now do the same for weight~$2$ and level~$37$. In particular, try $p=2$.
What do you observe? What could be the reason for this behaviour?
\end{enumerate}
\end{cexercise}

\begin{cexercise}\label{cexalgmodsymsketch}
Implement Algorithm~\ref{algmodsymsketch}.
\end{cexercise}

\section{Hecke algebras}

An important point made in the previous section is that for computing modular forms,
one computes Hecke algebras. This perspective puts Hecke algebras in its centre.
The present section is written from that point of view.
Starting from Hecke algebras, we define modular forms with coefficients in arbitrary rings,
we study integrality properties and also present results on the structure of Hecke algebras,
which are very useful for studying the arithmetic of modular forms.

It is essential for studying arithmetic properties of modular forms
to have some flexibility for the coefficient rings. For instance, when
studying mod~$p$ Galois representations attached to modular forms, it is
often easier and sometimes necessary to work with modular forms whose $q$-expansions
already lie in a finite field. Moreover, the concept of congruences of
modular forms only gets its seemingly correct framework when working over
rings such as extensions of finite fields or rings like $\ZZ/p^n\ZZ$.

There is a very strong theory of modular forms over a general ring~$R$
that uses algebraic geometry over~$R$. One can, however, already get
very far if one just defines modular forms over~$R$ as the $R$-linear
dual of the $\ZZ$-Hecke algebra of the holomorphic modular forms, i.e.\
by taking $q$-expansions with coefficients in~$R$. In this course we shall only
use this. Precise definitions will be given in a moment. A priori it is
maybe not clear whether non-trivial modular forms with $q$-expansions in the
integers exist at all. The situation is as good as it could possibly be:
the modular forms with $q$-expansion in the integers form a lattice in
the space of all modular forms (at least for $\Gamma_1(N)$
and $\Gamma_0(N)$; if we are working with a Dirichlet character, the
situation is slightly more involved). This is an extremely useful and important
fact, which we shall derive from the corollaries of the Eichler-Shimura 
isomorphism given in the previous section. 

Hecke algebras of modular forms 
over~$R$ are finitely generated $R$-modules. This leads us to a study, 
belonging to the theory
of Commutative Algebra, of finite $R$-algebras, that is, $R$-algebras
that are finitely generated as $R$-modules. We shall prove structure
theorems when $R$ is a discrete valuation ring or a finite field.
Establishing back the connection with modular forms, we will for example 
see that the maximal ideals of Hecke algebras correspond to Galois conjugacy
classes of normalised eigenforms, and, for instance, the notion of a
congruence can be expressed as a maximal prime containing two minimal ones.

\subsection{Theory: Hecke algebras and modular forms over rings}

We start by recalling and slightly extending
the concept of Hecke algebras of modular forms. 
It is of utmost importance for our treatment
of modular forms over general rings and their computation. In fact, as pointed
out a couple of times, we will compute Hecke algebras and not modular forms.
We shall assume that $k \ge 1$ and $N \ge 1$.

As in the introduction, we define the {\em Hecke algebra}
of $\Mkone kN\CC$ as the subring (i.e.\ the $\ZZ$-algebra) 
inside the endomorphism ring of the $\CC$-vector space $\Mkone kN\CC$
generated by all Hecke operators. Remember that due to Formula~\ref{diamondinhecke} 
all diamond operators are contained in the Hecke algebra.
Of course, we make similar definitions for $\Skone kN\CC$ and use the notations
$\TT_\ZZ(\Mkone kN\CC)$ and $\TT_\ZZ(\Skone{k}{N}{\CC})$.

If we are working with modular forms with a character, we essentially
have two possibilities for defining the Hecke algebra, namely,
firstly as above as the $\ZZ$-algebra generated by all Hecke operators
inside the endomorphism ring of the $\CC$-vector space $\Mk kN\chi\CC$
(notation $\TT_\ZZ (\Mk kN\chi\CC)$) or, secondly, as the $\ZZ[\chi]$-algebra
generated by the Hecke operators inside $\End_\CC(\Mk kN\chi\CC)$
(notation $\TT_{\ZZ[\chi]} (\Mk kN\chi\CC)$); similarly for the
cusp forms. Here $\ZZ[\chi]$ is the ring extension of $\ZZ$ generated
by all values of~$\chi$, it is the integer ring of $\QQ(\chi)$.
For two reasons we prefer the second variant. The first reason is
that we needed to work over $\ZZ[\chi]$ (or its extensions) for modular
symbols. The second reason is that on the natural $\ZZ$-structure
inside $\Mkone kN\CC$ the decomposition into $(\ZZ/N\ZZ)^\times$-eigenspaces
can only be made after a base change to $\ZZ[\chi]$. So, the $\CC$-dimension of
$\Mk kN\chi\CC$ equals the $\QQ[\chi]$-dimension
of $\TT_{\QQ[\chi]}(\Mk kN\chi\CC)$ and not the $\QQ$-dimension of
$\TT_{\QQ}(\Mk kN\chi\CC)$.

\begin{lemma}\label{freealg}
\begin{enumerate}[(a)]
\item The $\ZZ$-algebras $\TT_\ZZ (\Mkone kN\CC)$ and $\TT_\ZZ (\Mk kN\chi\CC)$
are free $\ZZ$-modules of finite rank; the same holds for the cuspidal Hecke algebras.
\item The $\ZZ[\chi]$-algebra $\TT_{\ZZ[\chi]} (\Mk kN\chi\CC)$
is a torsion-free finitely generated $\ZZ[\chi]$-module; 
the same holds for the cuspidal Hecke algebra.
\end{enumerate}
\end{lemma}

\begin{proof}
(a) Due to the corollaries of the Eichler-Shimura theorem 
(Corollary~\ref{algz}) we know that
these algebras are finitely generated as $\ZZ$-modules. As they lie
inside a vector space, they are free (using the structure theory
of finitely generated modules over principal ideal domains).

(b) This is like (a), except that $\ZZ[\chi]$ need not be a principal
ideal domain, so that we can only conclude torsion-freeness, but not
freeness.
\end{proof}

\subsubsection*{Modular forms over rings}

Let $k \ge 1$ and $N \ge 1$. Let $R$ be any $\ZZ$-algebra (ring).
We now use the $q$-pairing to define modular (cusp) forms over~$R$.
We let
\begin{align*}
 \Mkone kN R := &\Hom_\ZZ (\TT_\ZZ(\Mkone kN\CC),R)\\
 \cong&\Hom_R (\TT_\ZZ(\Mkone kN\CC) \otimes_\ZZ R,R).
\end{align*}
We stress the fact that $\Hom_R$ denotes the homomorphisms as $R$-modules (and not as $R$-algebras; those will appear later).
The isomorphism is proved precisely as in Proposition~\ref{hamsymbc}~(c),
where we did not use the flatness assumption.
Every element $f$ of $\Mkone kN R$ thus corresponds
to a $\ZZ$-linear function $\Phi: \TT_\ZZ(\Mkone kN\CC) \to R$ and is uniquely
identified by its {\em formal $q$-expansion}
$$f = \sum_n \Phi(T_n) q^n = \sum_n a_n(f) q^n \in R[[q]].$$
We note that $\TT_\ZZ(\Mkone kN\CC)$ acts naturally on
$\Hom_\ZZ (\TT_\ZZ(\Mkone kN\CC),R)$, namely by
\begin{equation}\label{eq:hecke}
(T.\Phi)(S) = \Phi(TS) = \Phi(ST).
\end{equation}
This means that the action of $\TT_\ZZ(\Mkone kN\CC)$ on
$\Mkone kN R$ gives the same formulae as usual on formal
$q$-expansions.
For cusp forms we make the obvious analogous definition, i.e.\
\begin{align*}
 \Skone kN R := &\Hom_\ZZ (\TT_\ZZ(\Skone kN\CC),R)\\
 \cong & \Hom_R (\TT_\ZZ(\Skone kN\CC) \otimes_\ZZ R,R).
\end{align*}

We caution the reader that for modular forms which are not cusp forms
there also ought to be some $0$th coefficient in the formal $q$-expansion, 
for example, for recovering
the classical holomorphic $q$-expansion. Of course, for cusp forms we do not
need to worry.

Now we turn our attention to modular forms with a character.
Let $\chi: (\ZZ/N\ZZ)^\times \to \CC^\times$ be a
Dirichlet character and $\ZZ[\chi] \to R$ a ring homomorphism.
We now proceed analogously to the treatment of modular symbols
for a Dirichlet character. We work with $\ZZ[\chi]$ as the base ring
(and not $\ZZ$). We let
\begin{align*}
\Mk kN\chi R := &\Hom_{\ZZ[\chi]} (\TT_{\ZZ[\chi]}(\Mk kN\chi\CC),R) \\
\cong &\Hom_R (\TT_{\ZZ[\chi]}(\Mk kN\chi\CC) \otimes_{\ZZ[\chi]} R,R)
\end{align*}
and similarly for the cusp forms.

We remark that these definitions of $\Mkone kN\CC$, $\Mk k N \chi \CC$ etc.\
agree with those from section~\ref{sec:1}; thus, it is justified to use the same pieces of notation.
As a special case, we get that 
$\Mkone k N \ZZ$ precisely consists of those holomorphic modular
forms in $\Mkone kN\CC$ whose $q$-expansions take values in~$\ZZ$.

If $\ZZ[\chi] \xrightarrow{\pi} R=\FF$ with $\FF$
a finite field of characteristic~$p$ or $\Fbar_p$, we call
$\Mk k N \chi \FF$ the space of {\em mod $p$ modular forms
of weight $k$, level $N$ and character $\chi$}.
Of course, for the cuspidal space similar statements are made and we use similar notation.

We furthermore extend the notation for Hecke algebras introduced in section~\ref{sec:1} as follows.
If $S$ is an $R$-algebra and $M$ is an $S$-module admitting the action of Hecke operators $T_n$ for $n \in \NN$,
then we let $\TT_R(M)$ be the $R$-subalgebra of $\End_S(M)$ generated by all $T_n$ for $n \in \NN$.

We now study base change properties of modular forms over~$R$.

\begin{proposition}
\begin{enumerate}[(a)]
\item Let $\ZZ \to R \to S$ be ring homomorphisms. Then the following
statements hold.
\begin{enumerate}[(i)]
\item The natural map
$$ \Mkone kNR \otimes_R S \to \Mkone kNS$$
is an isomorphism.
\item The evaluation pairing
$$ \Mkone kNR \times \TT_\ZZ(\Mkone kN\CC) \otimes_\ZZ R \to R$$
 is the $q$-pairing and it is perfect.
\item The Hecke algebra $\TT_R(\Mkone kNR)$ is naturally isomorphic to\\ 
$\TT_\ZZ(\Mkone kN\CC) \otimes_\ZZ R$.
\end{enumerate}
\item If $\ZZ[\chi] \to R \to S$ are flat, then Statement~(i) holds
for $\Mk kN\chi R$.
\item If $\TT_{\ZZ[\chi]} (\Mk kN\chi\CC)$ is a free $\ZZ[\chi]$-module
and $\ZZ[\chi] \to R \to S$ are ring homomorphisms,
statements (i)-(iii) hold for $\Mk kN\chi R$.
\end{enumerate}
\end{proposition}

\begin{proof}
(a) We use the following general statement, in which $M$ is assumed to
be a free finitely generated $R$-module and $N, T$ are $R$-modules:
$$ \Hom_R (M,N) \otimes_R T \cong \Hom_R (M, N\otimes_R T).$$
To see this, just see $M$ as $\bigoplus R$ and pull the direct sum out
of the $\Hom$, do the tensor product, and put the direct sum back
into the $\Hom$.

(i) Write $\TT_\ZZ$ for $\TT_\ZZ(\Mkone kN\CC)$. It is a free $\ZZ$-module by
Lemma~\ref{freealg}.
We have 
$$\Mkone kNR \otimes_R S = \Hom_\ZZ(\TT_\ZZ,R) \otimes_R S,$$
which by the above is isomorphic to $\Hom_\ZZ(\TT_\ZZ,R \otimes_R S)$ and
hence to $\Mkone kNS$.

(ii) The evaluation pairing $\Hom_\ZZ(\TT_\ZZ,\ZZ) \times \TT_\ZZ \to \ZZ$
is perfect, since $\TT_\ZZ$ is free as a $\ZZ$-module. The result
follows from~(i) by tensoring with~$R$.

(iii) We consider the natural map
$$ \TT_\ZZ \otimes_\ZZ R \to \End_R (\Hom_R(\TT_\ZZ \otimes_\ZZ R, R))$$
and show that it is injective. Its image is by definition $\TT_R(\Mkone kNR)$.
Let $T$ be in the kernel. Then $\phi(T) = 0$ for all 
$\phi \in \Hom_R(\TT_\ZZ \otimes_\ZZ R,R)$.
As the pairing in (ii) is perfect and, in particular, non-degenerate,
$T=0$ follows.

(b) Due to flatness we have 
$$\Hom_R(\TT_{\ZZ[\chi]} \otimes_{\ZZ[\chi]} R,R) \otimes_R S \cong \Hom_S (\TT_{\ZZ[\chi]} \otimes_{\ZZ[\chi]} S, S),$$
as desired.

(c) The same arguments as in~(a) work.
\end{proof}

\subsubsection*{Galois conjugacy classes}

By the definition of the Hecke action in equation~\eqref{eq:hecke}, the normalised Hecke eigenforms in 
the $R$-module $\Mkone k N R$ are precisely 
the $\ZZ$-algebra homomorphisms in
$\Hom_\ZZ (\TT_\ZZ(\Mkone kN\CC),R)$,
where the normalisation means that the identity operator $T_1$ is sent to~$1$.
Such an algebra homomorphism $\Phi$ is often referred to as a
{\em system of eigenvalues}, since the image of each $T_n$
corresponds to an eigenvalue of~$T_n$, namely to $\Phi(T_n) = a_n(f)$
(if $f$ corresponds to $\Phi$).

Let us now consider a perfect field $K$ (if we are working with a Dirichlet character,
we also want that $K$ admits a ring homomorphism $\ZZ[\chi] \to K$). 
Denote by $\Kbar$ an algebraic closure, so that we have
\begin{align*}
\Mkone kN \Kbar& = \Hom_\ZZ (\TT_\ZZ(\Mkone kN\CC),\Kbar)\\
& \cong \Hom_K (\TT_\ZZ(\Mkone kN\CC) \otimes_\ZZ K,\Kbar).
\end{align*}
We can compose any $\Phi \in \Hom_\ZZ (\TT_\ZZ(\Mkone kN\CC),\Kbar)$ by
any field automorphism $\sigma: \Kbar \to \Kbar$ fixing~$K$.
Thus, we obtain an action of the absolute Galois group
$\Gal(\Kbar/K)$ on $\Mkone kN \Kbar$ (on formal $q$-expansions,
we only need to apply $\sigma$ to the coefficients).
All this works similarly for the cuspidal subspace, too.

Like this, we also obtain a $\Gal(\Kbar/K)$-action on the
normalised eigenforms, and can hence speak about {\em Galois
conjugacy classes of eigenforms}.

\begin{proposition}\label{galoisconjugacy}
We have the following bijective correspondences:
\begin{align*}
 \Spec(\TT_K(\cdot)) &\overset{1-1}{\leftrightarrow}
\Hom_{K \textnormal{-alg}}(\TT_K(\cdot),\Kbar)/\Gal(\Kbar/K)\\
&\overset{1-1}{\leftrightarrow} \{\textnormal{ normalised eigenf.\ in $\cdot$ }\}/\Gal(\Kbar/K)
\end{align*}
and with $K = \Kbar$
$$ \Spec(\TT_\Kbar(\cdot)) \overset{1-1}{\leftrightarrow}
\Hom_{\Kbar \textnormal{-alg}}(\TT_\Kbar(\cdot),\Kbar)
\overset{1-1}{\leftrightarrow} \{\textnormal{ normalised eigenforms in $\cdot$ }\}.$$
Here, $\cdot$ stands for either $\Mkone kN\Kbar$, $\Skone kN\Kbar$
or the respective spaces with a Dirichlet character.
\end{proposition}

We recall that $\Spec$ of a ring is the set
of prime ideals. In the next section we will see that in
$\TT_K(\cdot)$ and $\TT_\Kbar(\cdot)$ all prime ideals are
already maximal (it is an easy consequence of the finite
dimensionality).

\begin{proof}
Exercise~\ref{exgaloisconjugacy}.
\end{proof}

We repeat that the coefficients of any eigenform~$f$
in $\Mk kN\chi \Kbar$ lie in a finite extension of $K$, namely
in $\TT_K(\Mk kN\chi K)/ \fm$, when $\fm$ is the maximal
ideal corresponding to the conjugacy class of~$f$.

Let us note that the above discussion applies to 
$\Kbar=\CC$, $\Kbar = \Qbar$, $\Kbar = \Qbar_p$, as well
as to $\Kbar = \Fbar_p$.
In the next sections we will also take into account the finer
structure of Hecke algebras over $\cO$, or rather over
the completion of $\cO$ at one prime.

\subsubsection{Some commutative algebra}

In this section we leave the special context of modular forms
for a moment and provide quite
useful results from commutative algebra that will be applied
to Hecke algebras in the sequel.

We start with a simple case which we will prove directly.
Let $\TT$ be an {\em Artinian} algebra, i.e.\ an algebra in which
every descending chain
of ideals becomes stationary. Our main example will be finite
dimensional algebras over a field. That those are Artinian is obvious, since in every
proper inclusion of ideals the dimension diminishes.

For any ideal $\fa$ of $\TT$ the sequence $\fa^n$ becomes stationary,
i.e.\ $\fa^n = \fa^{n+1}$ for all $n$ ``big enough''.
Then we will use the notation $\fa^\infty$ for $\fa^n$.

\begin{proposition}\label{propartin}
Let $\TT$ be an Artinian ring.
\begin{enumerate}[(a)]
\item Every prime ideal of $\TT$ is maximal.
\item There are only finitely many maximal ideals in $\TT$.
\item Let $\fm$ be a maximal ideal of~$\TT$. It is the only maximal
ideal containing~$\fm^\infty$.
\item Let $\fm \neq \fn$ be two maximal ideals. 
For any $k \in \NN$ and $k=\infty$ the ideals $\fm^k$ and $\fn^k$
are coprime.
\item The Jacobson radical $\bigcap_{\fm \in \Spec(\TT)} \fm$ is equal
to the nilradical and consists of the nilpotent elements.
\item We have $\bigcap_{\fm \in \Spec(\TT)} \fm^\infty = (0)$.
\item (Chinese Remainder Theorem) The natural map
$$\TT \xrightarrow{a \mapsto (\dots, a + \fm^\infty, \dots)} 
     \prod_{\fm \in \Spec(\TT)} \TT/\fm^\infty$$
is an isomorphism. 
\item For every maximal ideal $\fm$, the ring $\TT/\fm^\infty$ is local with
maximal ideal~$\fm$ and is hence isomorphic to $\TT_\fm,$
the localisation of $\TT$ at~$\fm$.
\end{enumerate}
\end{proposition}

\begin{proof}
(a) Let $\fp$ be a prime ideal of~$\TT$. The quotient $\TT \twoheadrightarrow \TT/\fp$
is an Artinian integral domain, since ideal chains in $\TT/\fp$ lift to ideal
chains in~$\TT$. Let $0 \neq x \in \TT/\fp$. We have $(x)^n = (x)^{n+1} = (x)^\infty$
for some $n$ big enough. Hence, $x^n = y x^{n+1}$ with some $y \in \TT/\fp$ and
so $xy = 1$, as $\TT/\fp$ is an integral domain.

(b) Assume there are infinitely many maximal ideals, number a countable
subset of them by $\fm_1, \fm_2, \dots$. Form the descending ideal chain
$$ \fm_1 \supset \fm_1 \cap \fm_2 \supset \fm_1\cap\fm_2\cap\fm_3 \supset \dots.$$
This chain becomes stationary, so that for some $n$ we have
$$ \fm_1\cap\dots\cap\fm_n = \fm_1\cap\dots\cap\fm_n\cap\fm_{n+1}.$$
Consequently, $\fm_1\cap\dots\cap\fm_n \subset \fm_{n+1}$. We claim that there
is $i \in \{1,2,\dots,n\}$ with $\fm_i \subset \fm_{n+1}$. Due to the maximality
of $\fm_i$ we obtain the desired contradiction. 
To prove the claim we assume that $\fm_i \not\subseteq \fm_{n+1}$ for all~$i$.
Let $x_i \in \fm_i - \fm_{n+1}$ and $y = x_1\cdot x_2 \cdots x_n$.
Then $y \in \fm_1\cap\dots\cap\fm_n$, but $y \not\in \fm_{n+1}$ due to the primality
of~$\fm_{n+1}$, giving a contradiction.

(c) Let $\fm \in \Spec(\TT)$ be a maximal ideal. Assume
that $\fn$ is a different maximal ideal with $\fm^\infty \subset \fn$.
Choose $x \in \fm$. Some power $x^r \in \fm^\infty$ and, thus,
$x^r \in \fn$. As $\fn$ is prime, $x \in \fn$ follows, implying
$\fm \subseteq \fn$, contradicting the maximality of~$\fm$.

(d) Assume that $I := \fm^k + \fn^k \neq \TT$. Then $I$ is contained in some
maximal ideal $\fp$. Hence, $\fm^\infty$ and $\fn^\infty$ are contained in~$\fp$,
whence by (c), $\fm=\fn=\fp$; contradiction.

(e) It is a standard fact from Commutative Algebra that 
the nilradical (the ideal of nilpotent elements) is the intersection 
of the minimal prime ideals.

(f) For $k \in \NN$ and $k = \infty$, (d) implies
$$ \bigcap_{\fm \in \Spec(\TT)} \fm^k = \prod_{\fm \in \Spec(\TT)} \fm^k
= (\prod_{\fm \in \Spec(\TT)} \fm)^k = (\bigcap_{\fm \in \Spec(\TT)} \fm)^k.$$
By (e) we know that $\bigcap_{\fm \in \Spec(\TT)} \fm$ is the nilradical. It
can be generated by finitely many elements $a_1,\dots,a_n$ all of which are
nilpotent. So a high enough power of $\bigcap_{\fm \in \Spec(\TT)} \fm$ is zero.

(g) The injectivity follows from~(f). It suffices to show that the elements
$(0,\dots,0,1,0,\dots,0)$ are in the image of the map. Suppose the $1$ is at
the place belonging to~$\fm$. Due to coprimeness (d) for any maximal ideal $\fn\neq \fm$
we can find $a_\fn \in \fn^\infty$ and $a_\fm \in \fm^\infty$ such that $1= a_\fm + a_\fn$.
Let $x := \prod_{\fn \in \Spec(\TT), \fn \neq \fm} a_\fn$. 
We have $x \in \prod_{\fn \in \Spec(\TT), \fn \neq \fm} \fn^\infty$ and
$x = \prod_{\fn \in \Spec(\TT), \fn \neq \fm} (1-a_\fm) \equiv 1 \mod \fm$.
Hence, the map sends $x$ to $(0,\dots,0,1,0,\dots,0)$, proving the surjectivity.

(h) By (c), the only maximal ideal of~$\TT$ containing~$\fm^\infty$ is~$\fm$.
Consequently, $\TT/{\fm^\infty}$ is a local ring with maximal ideal the image of~$\fm$.
Let $s \in \TT - \fm$. As $s + \fm^\infty \not\in \fm/\fm^\infty$, the element
$s + \fm^\infty$ is a unit in $\TT/\fm^\infty$. Thus, the map
$$ \TT_\fm \xrightarrow{\frac{y}{s} \mapsto y s^{-1} + \fm^\infty} \TT/\fm^\infty$$
is well-defined. It is clearly surjective. Suppose $\frac{y}{s}$ maps to~$0$.
Since the image of~$s$ is a unit, $y \in \fm^\infty$ follows. 
The element $x$ constructed in~(g) is in $\prod_{\fn \in \Spec(\TT), \fn \neq \fm} \fn^\infty$, but not in~$\fm$.
By (f) and (d), $(0) = \prod_{\fm \in \Spec(\TT)} \fm^\infty$. Thus, $y \cdot x = 0$ and
also $\frac{y}{s} = \frac{yx}{sx} = 0$, proving the injectivity.
\end{proof}

A useful and simple way to rephrase a product decomposition as in~(g) is to
use idempotents. In concrete terms, the idempotents of $\TT$ (as in the proposition)
are precisely the elements of the form $(\dots,x_\fm,\dots)$
with $x_\fm \in \{0,1\} \subseteq \TT/\fm^\infty$.

\begin{definition}
Let $\TT$ be a ring. An {\em idempotent of~$\TT$} is an element~$e$ that satisfies
$e^2=e$.
Two idempotents $e$, $f$ are {\em orthogonal} if $ef=0$.
An idempotent $e$ is {\em primitive} if $e\TT$ is a local ring.
A set of idempotents $\{e_1,\dots, e_n\}$ is said to be {\em complete} if
$1 = \sum_{i=1}^n e_i$.
\end{definition}

In concrete terms for $\TT = \prod_{\fm \in \Spec(\TT)} \TT/\fm^\infty$, a complete
set of primitive pairwise orthogonal idempotents is given by
$$(1,0,\dots,0), (0,1,0,\dots,0), \dots, (0,\dots,0,1,0), (0,\dots,0,1).$$
In Exercise~\ref{exidempotent}, you are asked (among other things)
to prove that in the above case
$\fm^\infty$ is a principal ideal generated by an idempotent.

Below we will present an algorithm for computing a complete set of primitive
pairwise orthogonal idempotents for an Artinian ring.

We now come to a more general setting, namely working with a finite algebra~$\TT$
over a complete local ring instead of a field. We will lift the idempotents 
of the reduction of~$\TT$ (for the maximal ideal of the complete local ring)
to idempotents of~$\TT$ by Hensel's lemma. This gives us a proposition
very similar to Proposition~$\ref{propartin}$.

\begin{proposition}[Hensel's lemma]\label{hensel}
Let $R$ be a ring that is complete with respect to the ideal~$\fm$ and let
$f \in R[X]$ be a polynomial. If
$$ f(a) \equiv 0 \mod (f'(a))^2 \fm$$
for some $a \in R$, then there is $b \in R$ such that
$$ f(b) = 0 \textnormal{ and } b \equiv a \mod f'(a)\fm.$$
If $f'(a)$ is not a zero-divisor, then $b$ is unique with
these properties.
\end{proposition}

\begin{proof}
\cite{Eisenbud}, Theorem~7.3.
\end{proof}

Recall that the {\em height} of a prime ideal $\fp$ in a ring~$R$ is the supremum among
all $n \in \NN$ such that there are inclusions of prime ideals
$\fp_0 \subsetneq \fp_1 \subsetneq \dots \subsetneq \fp_{n-1} \subsetneq \fp$.
The {\em Krull dimension} of~$R$ is the supremum of the heights of the prime ideals of~$R$.

\begin{proposition}\label{commalg}
Let $\cO$ be an integral domain of characteristic zero which is a finitely
generated $\ZZ$-module. Write $\widehat{\cO}$ for the completion of $\cO$ at a 
maximal prime of~$\cO$ and denote by $\FF$ the residue field
and by $\Khat$ the fraction field of~$\widehat{\cO}$.
Let furthermore $\TT$ be a commutative $\cO$-algebra
which is finitely generated as an $\cO$-module.
For any ring homomorphism $\cO\to S$ write $\TT_S$ for $\TT \otimes_\cO S$.
Then the following statements hold.

\begin{enumerate}[(a)]
\item The Krull dimension of $\TT_{\widehat{\cO}}$ is less than or equal to~$1$,
i.e.\ between any prime ideal and any maximal ideal $\fp \subset \fm$ there is
no other prime ideal.
The maximal ideals of $\TT_{\widehat{\cO}}$ correspond bijectively under
taking pre-images to the maximal ideals of $\TT_\FF$. Primes $\fp$ of
height $0$ (i.e.\ those that do not contain any other prime ideal)
which are properly contained in a prime of height~$1$ (i.e.\ a maximal prime) of
$\TT_{\widehat{\cO}}$ are in bijection with primes of $\TT_\Khat$ under extension
(i.e.\ $\fp \TT_\Khat$), for which the notation $\fp^e$ will be used.

Under the correspondences, one has
$$\TT_{\FF,\fm} \cong \TT_{\widehat{\cO},\fm} \otimes_{\widehat{\cO}} \FF$$
and
$$\TT_{\widehat{\cO}, \fp} \cong \TT_{\Khat,\fp^e}.$$

\item The algebra $\TT_{\widehat{\cO}}$ decomposes as
$$ \TT_{\widehat{\cO}} \cong \prod_\fm \TT_{\widehat{\cO},\fm},$$
where the product runs over the maximal ideals $\fm$ of $\TT_{\widehat{\cO}}$.

\item The algebra $\TT_\FF$ decomposes as
$$ \TT_\FF \cong \prod_\fm \TT_{\FF,\fm},$$
where the product runs over the maximal ideals $\fm$ of $\TT_\FF$.

\item The algebra $\TT_\Khat$ decomposes as
$$ \TT_\Khat \cong \prod_\fp \TT_{\Khat,\fp^e} \cong  \prod_\fp \TT_{\widehat{\cO},\fp} ,$$
where the products run over the minimal prime ideals $\fp$ of
$\TT_{\widehat{\cO}}$ which are contained in a prime ideal of height~$1$.
\end{enumerate}
\end{proposition}

\begin{proof} 
We first need that $\widehat{\cO}$ has Krull dimension~$1$. This, however,
follows from the fact that $\cO$ has Krull dimension~$1$, as it is an integral extension of~$\ZZ$,
and the correspondence between the prime ideals of a ring and its completion.
As $\TT_{\widehat{\cO}}$ is a finitely generated $\widehat{\cO}$-module, 
$\TT_{\widehat{\cO}}/\fp$ with a prime $\fp$ is an integral domain which is 
a finitely generated $\widehat{\cO}/(\fp \cap \widehat{\cO})$-module.  
Hence, it is either a finite field (when the prime ideal $\fp \cap \widehat{\cO}$ is the
unique maximal ideal of~$\widehat{\cO}$) or a finite
extension of $\widehat{\cO}$ (when $\fp \cap \widehat{\cO}=0$ so that the structure map
$\widehat{\cO} \to \TT_{\widehat{\cO}}/\fp$ is injective). This proves that the height of $\fp$ is less
than or equal to~$1$.  The correspondences and the isomorphisms of Part~(a) are 
the subject of Exercise~\ref{excorrespondences}.

We have already seen Parts~(c) and~(d) in Lemma~\ref{propartin}. Part~(b)
follows from~(c) by applying Hensel's lemma (Proposition~\ref{hensel}) to the idempotents
of the decomposition of~(c). We follow \cite{Eisenbud}, Corollary~7.5, 
for the details. Since $\widehat{\cO}$ is complete with respect 
to some ideal~$\fp$, so is $\TT_{\widehat{\cO}}$ as it is a finitely generated $\widehat{\cO}$-module. Hence, we may use
Hensel's lemma in $\TT_{\widehat{\cO}}$.
Given an idempotent $\overline{e}$ of $\TT_\FF$, we will first show
that it lifts to a unique idempotent of~$\TT_{\widehat{\cO}}$.
Let $e$ be any lift of $\overline{e}$ and let $f(X) = X^2 - X$ be a polynomial 
annihilating~$\overline{e}$.
We have that $f'(e) = 2e-1$ is a unit since
$(2e-1)^2 = 4e^2-4e+1\equiv 1 \mod \fp$.
Hensel's lemma now gives us a unique root $e_1 \in \TT_{\widehat{\cO}}$ of~$f$,
i.e.\ an idempotent, lifting~$\overline{e}$.

We now lift every element of a set of pairwise orthogonal idempotents
of $\TT_\FF$. It now suffices to show that the lifted idempotents are also
pairwise orthogonal (their sum is $1$; otherwise we would get a contradiction
in the correspondences in~(a): there cannot be more idempotents in $\TT_{\widehat{\cO}}$
than in $\TT_\FF$). As their reductions are orthogonal, a product $e_ie_j$
of lifted idempotents is in~$\fp$. Hence, $e_ie_j=e_i^de_j^d \in \fp^d$ for
all~$d$, whence $e_i e_j = 0$, as desired.
\end{proof}

\subsubsection{Commutative algebra of Hecke algebras}

Let $k\ge 1$, $N \ge 1$ and $\chi: (\ZZ/N\ZZ)^\times \to \CC^\times$.
Moreover, let $p$ be a prime, $\cO := \ZZ[\chi]$, $\fP$ a maximal
prime of $\cO$ above~$p$, and let $\FF$
be the residue field of $\cO$ modulo~$\fP$.
We let $\widehat{\cO}$ denote the completion of $\cO$
at~$\fP$. Moreover, the field of fractions of $\widehat{\cO}$
will be denoted by~$\Khat$ and an algebraic closure by $\Khatbar$.
For $\TT_\cO(\Mk kN\chi\CC)$ we only write $\TT_\cO$
for short, and similarly over other rings.
We keep using the fact that $\TT_\cO$ is finitely generated
as an $\cO$-module.
We shall now apply Proposition~\ref{commalg} to $\TT_{\widehat{\cO}}$.

\begin{proposition}\label{prop:pure-one}
The Hecke algebras $\TT_{\cO}$ and $\TT_{\widehat{\cO}}$ are pure of Krull dimension~$1$,
i.e.\ every maximal prime contains some minimal prime ideal.
\end{proposition}

\begin{proof}
It suffices to prove that $\TT_{\widehat{\cO}}$ is pure of Krull dimension~$1$ because
completion of $\TT_{{\cO}}$ at a maximal ideal of $\calO$ does not change the Krull dimension.
First note that $\widehat{\cO}$ is pure of Krull dimension~$1$ as it
is an integral extension of~$\ZZ_p$ (and the Krull dimension is an invariant in integral extensions).
With the same reasoning, $\TT_{\widehat{\cO}}$ is of Krull dimension~$1$; we have to see that it is pure.
According to proposition~\ref{commalg}, $\TT_{\widehat{\cO}}$ is the direct product of finite local
$\widehat{\cO}$-algebras~$\TT_i$. As each $\TT_i$ embeds into a finite dimensional matrix algebra over~$\Khat$,
it admits a simultaneous eigenvector (after possibly a finite extension of $\Khat$) for the standard action
of the matrix algebra on the corresponding $\Khat$-vector space and the map $\varphi$ sending
an element of $\TT_i$ to its eigenvalue is non-trivial and its kernel is a prime ideal strictly contained
in the maximal ideal of~$\TT_i$.
To see this, notice that the eigenvalues are integral, i.e.\ lie in the valuation ring of a finite extension of~$\Khat$,
and can hence be reduced modulo the maximal ideal. The kernel of $\varphi$ followed by the reduction map is the required maximal ideal.
This proves that the height of the maximal ideal is~$1$.
\end{proof}

By Proposition~\ref{commalg}, minimal primes of $\TT_{\widehat{\cO}}$
correspond to the maximal primes of $\TT_\Khat$ and hence
to $\Gal(\Khatbar/\Khat)$-conjugacy classes of eigenforms
in $\Mk kN\chi\Khatbar$. By a brute force identification of
$\Khatbar = \Qbar_p$ with $\CC$ we may still think about
these eigenforms as the usual holomorphic ones (the Galois
conjugacy can then still be seen as conjugacy by a decomposition
group above~$p$ inside the absolute Galois group of
the field of fractions of~$\cO$).

Again by Proposition~\ref{commalg}, maximal prime ideals of $\TT_{\widehat{\cO}}$
correspond to the maximal prime ideals of $\TT_\FF$ and hence
to $\Gal(\Fbar/\FF)$-conjugacy classes of eigenforms
in $\Mk kN\chi\Fbar$.

The spectrum of $\TT_{\widehat{\cO}}$ allows one to phrase very
elegantly when conjugacy classes of eigenforms are congruent
modulo a prime above~$p$. Let us first explain what that means.
Normalised eigenforms~$f$ take their coefficients $a_n(f)$ in
rings of integers of number fields ($\TT_\cO / \fm$, when $\fm$
is the kernel of the $\cO$-algebra homomorphism $\TT_\cO \to \CC$,
given by $T_n \mapsto a_n(f)$), so they can be reduced modulo
primes above~$p$ (for which we will often just say ``reduced
modulo~$p$'').
The reduction modulo a prime above~$p$ of the $q$-expansion 
of a modular form $f$ in $\Mk kN\chi\CC$ is the
formal $q$-expansion of an eigenform in $\Mk kN\chi\Fbar$.

If two normalised eigenforms $f,g$ in $\Mk kN\chi\CC$ or
$\Mk kN\chi\Khatbar$ reduce to the same element in
$\Mk kN\chi\Fbar$, we say that they are {\em congruent
modulo~$p$}. 

Due to Exercise~\ref{exconj}, we may speak about {\em reductions
modulo~$p$} of $\Gal(\Khatbar/\Khat)$-conjugacy classes of normalised 
eigenforms to $\Gal(\Fbar/\FF)$-conjugacy classes.
We hence say that two $\Gal(\Khatbar/\Khat)$-con\-ju\-gacy
classes, say corresponding to normalised eigenforms $f,g$,
respectively, minimal ideals $\fp_1$
and $\fp_2$ of $\TT_{\widehat{\cO}}$, are {\em congruent
modulo~$p$}, if they reduce to the same
$\Gal(\Fbar/\FF)$-conjugacy class.

\begin{proposition}\label{conjugacy}
The $\Gal(\Khatbar/\Khat)$-conjugacy classes belonging to minimal
primes $\fp_1$ and $\fp_2$ of $\TT_{\widehat{\cO}}$ are congruent
modulo~$p$ if and only if they are contained in a
common maximal prime $\fm$ of $\TT_{\widehat{\cO}}$.
\end{proposition}

\begin{proof}
Exercise~\ref{exconjugacy}.
\end{proof}

We mention the fact that if $f$ is a newform belonging to the maximal ideal
$\fm$ of the Hecke algebra $\TT := \TT_\QQ(S_k(\Gamma_1(N),\CC))$, then $\TT_\fm$ is isomorphic
to $\QQ_f = \QQ(a_n | n \in \NN)$.
This follows from newform (Atkin-Lehner) theory (see \cite[\S5.6-5.8]{DS}), which implies that the Hecke
algebra on the newspace is diagonalisable, so that it is the direct product of the coefficient fields.

We include here the famous Deligne-Serre lifting lemma \cite[Lemme~6.11]{DeligneSerre}, which we can easily prove with the tools developed so far.

\begin{proposition}[Deligne-Serre lifting lemma]\label{prop:deligne-serre}
Any normalised eigenform $\fbar \in \Skone kN{\Fbar_p}$ is the reduction of a normalised eigenform $f \in \Skone kN\CC$.
\end{proposition}

\begin{proof}
Let $\TT_\ZZ = \TT_\ZZ(\Skone kN\CC)$.
By definition, $\fbar$ is a ring homomorphism $\TT_\ZZ \to \Fbar_p$ and its kernel is a maximal ideal $\fm$ of $\TT_\ZZ$.
According to Proposition~\ref{prop:pure-one}, the Hecke algebra  is pure of Krull dimension one, hence $\fm$
is of height~$1$, meaning that it strictly contains a minimal prime ideal~$\fp \subset \TT_\ZZ$.
Let $f$ be the composition of the maps in the first line of the diagram:
$$ \xymatrix@=1cm{
\TT_\ZZ \ar@{->>}[r]\ar@{->>}[dr]\ar@{->}[drr]^(.7){\fbar} & \TT_\ZZ/\fp \ar@{^(->}[r]\ar@{->>}[d] & \Zbar \ar@{->>}[d] \ar@{^(->}[r] &\CC\\
&  \TT_\ZZ/\fm \ar@{^(->}[r] & \Fbar_p}$$
where all surjections and all injections are the natural ones,
and the map $\Zbar \twoheadrightarrow \Fbar_p$ is chosen in order to make the diagram commutative.
Note that $\fbar$ is a ring homomorphism and thus a normalised eigenform in $\Skone kN\CC$. By the diagram, its reduction is~$\fbar$.
\end{proof}

\subsection{Algorithms and Implementations: Localisation Algorithms}

Let $K$ be a perfect field, $\Kbar$ an algebraic closure
and $A$ a finite dimensional commutative $K$-algebra. 
In the context of Hecke algebras we would like to
compute a local decomposition of~$A$ as in Proposition~\ref{commalg}.

\subsubsection{Primary spaces}

\begin{definition}
An $A$-module $V$ which is finite dimensional as $K$-vector space is called a {\em primary space} for $A$ if the minimal polynomial for all
$a \in A$ is a prime power in~$K[X]$.
\end{definition}

\begin{lemma}\label{LemDec}
\begin{enumerate}[(a)]
\item $A$ is local if and only if the minimal polynomial of~$a$ (in $K[X]$)
is a prime power for all $a \in A$. 
\item Let $V$ be an $A$-module which is finite dimensional as $K$-vector space and which is a primary space for~$A$.
Then the image of $A$ in $\End_K(V)$ is a local algebra. 
\item Let $V$ be an $A$-module which is finite dimensional as $K$-vector space and let $a_1, \dots, a_n$ be generators
of the algebra $A$. Suppose that for 
$i \in \{1,\dots,n\}$ the minimal polynomial $a_i$ on $V$ is a power of $(X - \lambda_i)$ in $K[X]$
for some $\lambda_i \in K$ (e.g.\ if $K = \Kbar$).
Then the image of $A$ in $\End_K(V)$ is a local algebra. 
\end{enumerate}
\end{lemma}

\begin{proof}
(a) Suppose first that $A$ is local and take $a \in A$. 
Let $\phi_a: K[X] \to A$ be the homomorphism of $K$-algebras defined
by sending $X$ to $a$. Let $(f)$ be the kernel with $f$ monic, 
so that by definition $f$ is the minimal polynomial of~$a$. 
Hence, $K[X]/(f) \hookrightarrow A$, whence $K[X]/(f)$ is local, as
it does not contain any non-trivial idempotent.
Thus, $f$ cannot have two different prime factors. 

Conversely, if $A$ were not local, we would have an idempotent
$e \not \in  \{0,1\}$. The minimal polynomial of $e$
is $X(X-1)$, which is not a prime power. 

(b) follows directly.
For (c) one can use the following. Suppose that $(a-\lambda)^r V = 0$
and $(b - \mu)^s V = 0$. Then $((a+b) - (\lambda + \mu))^{r+s} V = 0$,
as one sees by rewriting $((a+b) - (\lambda + \mu)) = (a-\lambda)+(b-\mu)$
and expanding out. From this it also follows that $(ab - \lambda \mu)^{2(r+s)} V = 0$
by rewriting $ab - \lambda \mu = (a-\lambda)(b-\mu) + \lambda(b-\mu) + \mu(a - \lambda)$.
\end{proof}

We warn the reader that algebras such that a set of generators acts primarily 
need not be local, unless they are defined over an algebraically closed
field, as we have seen in Part~(c) above. In Exercise~\ref{exnonlocal}
you are asked to find an example.

The next proposition, however, tells us that an algebra over a field
having a basis consisting of primary elements is local. I found the idea
for that proof in~\cite{Eberly}.

\begin{proposition}\label{primary}
Let $K$ be a field of characteristic~$0$ or a finite field.
Let $A$ be a finite dimensional commutative algebra over~$K$ and let
$a_1,\dots,a_n$ be a $K$-basis of~$A$ with the property that
the minimal polynomial of each~$a_i$ is a power of a prime
polynomial $p_i \in K[X]$.

Then $A$ is local.
\end{proposition}

\begin{proof}
We assume that $A$ is not local and take a decomposition
$\alpha: A \xrightarrow{\sim} \prod_{j=1}^r A_j$ with $r \ge 2$.
Let $K_j$ be the residue field of $A_j$ and consider the finite dimensional $K$-algebra $\overline{A} := \prod_{j=1}^r K_j$.
Write $\overline{a_1},\dots,\overline{a_n}$ for the images of the $a_i$ in~$\overline{A}$.
They form a $K$-basis. In order to have access to the components, also write $\overline{a_i} = (\overline{a_{i,1}},\dots,\overline{a_{i,r}})$.
Since the minimal polynomial of an element in a product is the lowest common multiple of the minimal polynmials
of the components, the assumption implies that, for each $i=1,\dots,r$, the minimal polynomial of
$a_{i,j}$ is independent of~$j$; call it $p_i \in K[X]$.
Let $N/K$ be the splitting field of the polynomials $p_1,\dots,p_r$. This means that $N$ is the normal closure of $K_j$ over~$K$
for any~$j$. As a particular case, note that $N=K_j$ for all $j$ if $K$ is a finite field since finite extensions of finite fields
are automatically normal.
Now consider the trace $\Tr_{N/K}$ and note that $\Tr_{N/K}(\overline{a_i})$ is a diagonal element in~$\overline{A}$ for all $i=1,\dots,r$
since the components $\overline{a_{i,j}}$ are roots of the same minimal polynomial.
Consequently, $\Tr_{N/K}(\overline{a})$ is a diagonal element for all $\overline{a} \in \overline{A}$ since the $\overline{a_i}$
form a $K$-basis of~$\overline{A}$.

In order to come to a contradiction, it now suffices to produce an element the trace of which is not diagonal.
By Exercise~\ref{extrace} there is $x \in K_1$ such that $\Tr_{N/K} (x) \neq 0$.
Then the element $(x,0,\dots,0) \in \overline{A}$ clearly provides an example of an element with non-diagonal trace.
\end{proof}

\begin{lemma}\label{algfield}
Let $A$ be a local finite dimensional commutative algebra over a perfect field~$K$.
Let $a_1,\dots, a_n$ be a set of $K$-algebra generators of~$A$
such that the minimal polynomial of each $a_i$ is a prime polynomial.
Then $A$ is a field.
\end{lemma}

\begin{proof}
As the $a_i$ are simultaneously diagonalisable over a separable closure (considering the algebra as a matrix algebra)
due to their minimal polynomials being squarefree (using here the perfectness of $K$), so are sums and products of the~$a_i$.
Hence, $0$ is the only nilpotent element in~$A$. As the maximal
ideal in an Artinian local algebra is the set of nilpotent elements,
the lemma follows.
\end{proof}

\begin{proposition}\label{genmax}
Let $A$ be a local finite dimensional commutative algebra over a perfect field~$K$.
Let $a_1,\dots, a_n$ be a set of $K$-algebra generators of~$A$.
Let $p_i^{e_i}$ be the minimal polynomial of~$a_i$ (see Lemma~\ref{LemDec}).

Then the maximal ideal~$\fm$ of~$A$ is generated by 
$\{p_1(a_1),\dots,p_n(a_n)\}$.
\end{proposition}

\begin{proof}
Let $\fa$ be the ideal generated by $\{p_1(a_1),\dots,p_n(a_n)\}$.
The quotient $A/\fa$ is generated by the images of the~$a_i$,
call them $\overline{a_i}$.
As $p_i(a_i) \in \fa$, it follows $p_i(\overline{a_i}) = 0$, whence
the minimal polynomial of $\overline{a_i}$ equals the prime polynomial~$p_i$.
By Lemma~\ref{algfield}, we know that $A/\fa$ is a field, whence $\fa$ is the maximal ideal.
\end{proof}

\subsubsection{Algorithm for computing common primary spaces}

It may help to think about finite dimensional commutative algebras over
a field as algebras of matrices. Then the localisation statements
of this section just mean choosing a basis such that one obtains block matrices.

By a {\em common primary space} for commuting matrices we mean a direct summand
of the underlying vector space on which the minimal polynomials of the given matrices
are prime powers. By Proposition~\ref{primary}, a common primary
space of a basis of a matrix algebra is a local factor of the algebra.

By a {\em generalised eigenspace} for commuting matrices we mean a vector
subspace of the underlying vector space on which the minimal polynomial of the given matrices
are irreducible. Allowing base changes to extension fields, the matrices
restricted to the generalised eigenspace are diagonalisable.

In this section we present a straight forward algorithm for
computing common primary spaces and common generalised eigenspaces.

\begin{algorithm}\label{algPrimary}
\underline{Input}: list {\tt ops} of commuting operators acting
on the $K$-vector space~$V$.

\underline{Output}: list of the common primary spaces inside~$V$
for all operators in {\tt ops}.

\begin{enumerate}[(1)]
\itemsep=0cm plus 0pt minus 0pt
\item List := [V];
\item for $T$ in {\tt ops} do
\item \ms newList := [];
\item \ms for $W$ in List do
\item \ms \ms compute the minimal polynomial $f \in K[X]$ of $T$
restricted to $W$.
\item \ms \ms factor $f$ over $K$ into its prime powers $f(X) = \prod_{i=1}^n p_i(X)^{e_i}$.
\item \ms \ms if $n$ equals $1$, then
\item \ms \ms \ms append $W$ to newList,
\item \ms \ms else for i := 1 to n do
\item \ms \ms \ms compute $\widetilde{W}$ as the kernel of
      $p_i(T|_W)^\alpha$ with $\alpha = {e_i}$ for common primary spaces
 or $\alpha = 1$ for common generalised eigenspaces.
\item \ms \ms \ms append $\widetilde{W}$ to newList.
\item \ms \ms end for; end if;
\item \ms end for;
\item \ms List := newList;
\item end for;
\item return List and stop. 
\end{enumerate}
\end{algorithm}

\subsubsection{Algorithm for computing idempotents}

Using Algorithm~\ref{algPrimary} it is possible to compute a complete
set of orthogonal primitive idempotents for~$A$. We now sketch a direct algorithm.

\begin{algorithm}\label{algIdempotents}
\underline{Input}: matrix~$M$.

\underline{Output}: complete set of orthogonal primitive idempotents
for the matrix algebra generated by $M$ and $1$.

\begin{enumerate}[(1)]
\itemsep=0cm plus 0pt minus 0pt
\item compute the minimal polynomial $f$ of~$M$.
\item factor it $f = (\prod_{i=1}^n p_i^{e_i}) X^e$ over $K$ with
$p_i$ distinct irreducible polynomials different from~$X$.
\item List := [];
\item for $i = 1$ to $n$ do
\item \ms $g := f / p_i^{e_i}$;
\item \ms $M_1 := g(M)$. If we think about $M_1$ in block form,
then there is only one non-empty block on the diagonal, the rest is zero.
In the next steps this block is replaced by the identity.
\item \ms compute the minimal polynomial $h$ of $M_1$.
\item \ms strip possible factors $X$ from~$h$ and normalise $h$ so that $h(0)=1$.
\item \ms append $1-h(M_1)$ to List. Note that $h(M_1)$ is the identity matrix
except at the block corresponding to $p_i$, which is zero. Thus $1-h(M_1)$
is the idempotent being zero everywhere and being the identity in the block
corresponding to~$p_i$.
\item end for;
\item if $e > 0$ then
\item \ms append $1-\sum_{e \in \textnormal{ List}} e$ to List.
\item end if;
\item return List and stop. 
\end{enumerate}
\end{algorithm}

The algorithm for computing a complete set of orthogonal primitive idempotents
for a commutative matrix algebra consists of multiplying together the
idempotents of every matrix in a basis. See Computer Exercise~\ref{cexidempotents}.

\subsection{Theoretical exercises}

\begin{exercise}
Use your knowledge on modular forms to prove that a modular
form $f = \sum_{n=0}^\infty a_n(f) q^n$ of weight $k \ge 1$ and 
level $N$ (and Dirichlet character~$\chi$)
is uniquely determined by $\sum_{n=1}^\infty a_n(f) q^n$.
\end{exercise}

\begin{exercise}\label{exgaloisconjugacy}
Prove Proposition~\ref{galoisconjugacy}.

Hint: use that the kernel of a ring homomorphism into an integral domain is a prime ideal; moreover, use that
all prime ideals in the Hecke algebra in the exercise are maximal; finally, use that field
homomorphisms can be extended to separable extensions (using here that $K$ is perfect).
\end{exercise}

\begin{exercise}\label{exidempotent}
Let $\TT$ be an Artinian ring.
\begin{enumerate}[(a)] 
\item Let $\fm$ be a maximal ideal of~$\TT$.
Prove that $\fm^\infty$ is a principal ideal generated by an
idempotent. Call it~$e_\fm$.
\item Prove that the idempotents $1 - e_\fm$ and $1 - e_\fn$ for different
maximal ideals $\fm$ and $\fn$ are orthogonal.
\item Prove that the set $\{1-e_\fm | \fm \in \Spec(\TT)\}$ forms a complete set
of pairwise orthogonal idempotents.
\end{enumerate}
Hint: see \cite[\S8]{AM}.
\end{exercise}

\begin{exercise}\label{excorrespondences}
Prove the correspondences and the isomorphisms from Part~(a)
of Proposition~\ref{commalg}.

Hint: you only need basic reasonings from commutative algebra.
\end{exercise}

\begin{exercise}\label{exconj}
Let $f,g \in \Mk kN\chi\Khatbar$ be normalised eigenforms that we assume to be
$\Gal(\Khatbar/\Kbar)$-conjugate. Prove that their reductions modulo~$p$
are $\Gal(\Fbar/\FF)$-conjugate.
\end{exercise}

\begin{exercise}\label{exconjugacy}
Prove Proposition~\ref{conjugacy}.

Hint: it suffices to write out the definitions.
\end{exercise}

\begin{exercise}\label{exnonlocal}
Find a non-local algebra~$A$ over a field~$K$ (of your choice) such that~$A$ 
is generated as a $K$-algebra by $a_1,\dots,a_n$ having the property
that the minimal polynomial of each $a_i$ is a power of an irreducible
polynomial in~$K[X]$.
\end{exercise}

\begin{exercise}\label{extrace}
Let $K$ be a field of characteristic~$0$ or a finite field.
Let $L$ be a finite extension of~$K$ with Galois closure $N$ over~$K$.
Show that there is an element $x \in L$ with $\Tr_{N/K} (x) \neq 0$.
\end{exercise}

\begin{exercise}\label{exlowertriangular}
Let $A$ be a commutative matrix algebra over a perfect field~$K$.
Suppose that the minimal polynomial of each element of a generating
set is the power of a prime polynomial (i.e.\ it is primary).

Show that there exist base change matrices such that the base
changed algebra consists only of lower triangular matrices. You may
and you may have to extend scalars to a finite extension of~$K$.
In Computer Exercise~\ref{cexlowertriangular} you are asked to find and
implement an algorithm computing such base change matrices.
\end{exercise}

\subsection{Computer exercises}

\begin{cexercise}
Change Algorithm~\ref{algmodsymsketch} 
(see Computer Exercise~\ref{cexalgmodsymsketch}) so that
it works for modular forms over a given ring~$R$.
\end{cexercise}

\begin{cexercise}\label{cexmaxid}
Let $A$ be a commutative matrix algebra over a perfect field~$K$.
\begin{enumerate}[(a)]
\item Write an algorithm to test whether $A$ is local.
\item Suppose $A$ is local. Write an algorithm to compute its
maximal ideal.
\end{enumerate}
\end{cexercise}

\begin{cexercise}
Let $A$ be a commutative algebra over a field~$K$.
The regular representation is defined as the image
of the injection
$$ A \to \End_K(A), \;\;\; a \mapsto (b \mapsto a\cdot b).$$
Write a function computing the regular representation.
\end{cexercise}

\begin{cexercise}
Implement Algorithm~\ref{algPrimary}. Also write a function that
returns the local factors as matrix algebras (possibly using
regular representations).
\end{cexercise}

\begin{cexercise}\label{cexidempotents}
\begin{enumerate}[(a)]
\item Implement Algorithm~\ref{algIdempotents}.
\item Let $S$ be a set of idempotents. 
Write a function selecting a subset of~$S$ consisting of pairwise
orthogonal idempotents such that the subset spans~$S$ (all idempotents
in~$S$ can be obtained as sums of elements in the subset).
\item Write a function computing a complete set of pairwise orthogonal
idempotents for a commutative matrix algebra~$A$ over a field by multiplying
together the idempotents of the matrices in a basis and selecting a
subset as in~(b).
\item Use Computer Exercise~\ref{cexmaxid} to compute the maximal ideals
of~$A$.
\end{enumerate}
\end{cexercise}

\begin{cexercise}
Let $A$ be a commutative matrix algebra over a perfect field~$K$.
Suppose that $A$ is a field (for instance obtained as the quotient
of a local $A$ by its maximal ideal computed in Computer Exercise~\ref{cexmaxid}).
Write a function returning an irreducible polynomial~$p$
such that $A$ is $K[X]/(p)$.

If possible, the algorithm should not use factorisations of polynomials.
It is a practical realisation of Kronecker's primitive element theorem.
\end{cexercise}

\begin{cexercise}\label{cexlowertriangular}
Let $A$ be a commutative matrix algebra over a perfect field~$K$.
Suppose that the minimal polynomial of each element of a generating
set is the power of a prime polynomial (i.e.\ it is primary).

Write a function computing base change matrices such that the base
changed algebra consists only of lower triangular matrices
(cf.\ Exercise~\ref{exlowertriangular}).
\end{cexercise}

\section{Homological algebra}

In this section we provide the tools from homological algebra that will be necessary
for the modular symbols algorithm (in its group cohomological version).
A good reference is~\cite{Weibel}.

We will be sloppy about categories. When we write category below,
we really mean abelian category, since we obviously need the existence of
kernels, images, quotients etc. For what we have in mind, we should really understand the word
category not in its precise mathematical sense but as a placeholder for 
$R-\modules$, or (co-)chain complexes of $R-\modules$ and other categories
from everyday life.

\subsection{Theory: Categories and Functors}

\begin{definition}
A {\em category} $\cC$ consists of the following data:
\begin{itemize}
\item a class $\obj(\cC)$ of {\em objects},
\item a set $\Hom_\cC(A,B)$ of {\em morphisms} for every ordered pair $(A,B)$
of objects,
\item an {\em identity morphism} $\id_A \in \Hom_\cC(A,A)$ for every object $A$, and
\item a {\em composition function}
$$ \Hom_\cC(A,B) \times \Hom_\cC(B,C) \to \Hom_\cC(A,C), \;\; (f,g) \mapsto g \circ f$$
for every ordered triple $(A,B,C)$ of objects
\end{itemize}
such that
\begin{itemize}
\item (Associativity) $(h \circ g) \circ f = h \circ (g \circ f)$ for all
$f \in \Hom_\cC(A,B)$, $g \in \Hom_\cC(B,C)$, $h \in \Hom_\cC (C,D)$ and
\item (Unit Axiom) $\id_B \circ f = f = f \circ \id_A$ for $f \in \Hom_\cC(A,B)$.
\end{itemize}
\end{definition}

\begin{example}
Examples of categories are
\begin{itemize}
\item Sets: objects are sets, morphisms are maps.
\item Let $R$ be a not necessarily commutative ring. Left $R$-modules ($R-\modules$): objects are $R$-modules, morphisms are $R$-module homomorphisms.
This is the category we are going to work with most of the time. Note that the category
of $\ZZ$-modules is the category of abelian groups.
\item Right $R$-modules ($\modules-R$): as above.
\end{itemize}
\end{example}

\begin{definition}
Let $\cC$ and $\calD$ be categories. A {\em covariant/contravariant functor} $F: \cC \to \calD$ is
\begin{itemize}
\item a rule $\obj(\cC) \to \obj(\calD), \; C \mapsto F(C)$ and
\item a rule 
$\begin{cases}
\textnormal{covariant:} & \Hom_\cC(C_1,C_2) \to \Hom_\calD(F(C_1),F(C_2)), \; f \mapsto F(f)\\
\textnormal{contravariant:} & \Hom_\cC(C_1,C_2) \to \Hom_\calD(F(C_2),F(C_1)), \; f \mapsto F(f)\\
\end{cases}$
\end{itemize}
such that
\begin{itemize}
\item $F(\id_C) = \id_{F(C)}$ and
\item $\begin{cases}
\textnormal{covariant:} & F(g \circ f) = F(g) \circ F(f)\\
\textnormal{contravariant:} & F(g \circ f) = F(f) \circ F(g)\\
\end{cases}$
\end{itemize}
\end{definition}

\begin{example}
\begin{itemize}
\item Let $M \in \obj(R-\modules)$. Define
$$\Hom_R(M,\cdot): R-\modules \to \ZZ-\modules, \;\; A \mapsto \Hom_R(M,A).$$
This is a covariant functor.
\item Let $M \in \obj(R-\modules)$. Define
$$\Hom_R(\cdot,M): R-\modules \to \ZZ-\modules, \;\; A \mapsto \Hom_R(A,M).$$
This is a contravariant functor.
\item Let $M \in \obj(R-\modules)$. Define
$$\cdot \otimes_R M: \modules-R \to \ZZ-\modules, \;\; A \mapsto A \otimes_R M.$$
This is a covariant functor.
\item Let $M \in \obj(\modules-R)$. Define
$$M \otimes_R \cdot: R-\modules \to \ZZ-\modules, \;\; A \mapsto M \otimes_R A.$$
This is a covariant functor.
\end{itemize}
\end{example}

\begin{definition}
Let $\cC$ and $\calD$ be categories and  $F,G: \cC \to \calD$ be both covariant or both contra\-variant functors.
A {\em natural transformation} $\alpha: F \Rightarrow G$ is a collection of morphisms
$\alpha= (\alpha_C)_{C \in \cC}: F(C) \to G(C)$ in~$\calD$ for $C \in \cC$ such that
for all morphisms $f:C_1 \to C_2$ in~$\cC$ the following diagram commutes:\\
\begin{tabular}{cc}
covariant: & contravariant:\\
$\xymatrix@=0.5cm{
F(C_1) \ar@{->}[r]^{F(f)} \ar@{->}[d]_{\alpha_{C_1}} & F(C_2) \ar@{->}[d]^{\alpha_{C_2}}\\
G(C_1) \ar@{->}[r]^{G(f)}  & G(C_2) }$ &
$\xymatrix@=0.5cm{
F(C_1) \ar@{<-}[r]^{F(f)} \ar@{->}[d]_{\alpha_{C_1}} & F(C_2) \ar@{->}[d]_{\alpha_{C_2}}\\
G(C_1) \ar@{<-}[r]^{G(f)}  & G(C_2). }$
\end{tabular}
\end{definition}

\begin{example}\label{nattrans}
Let $R$ be a not necessarily commutative ring and let $A,B \in \obj(R-\modules)$
together with a morphism $A \to B$.
Then there are natural transformations
$\Hom_R(B,\cdot) \Rightarrow \Hom_R(A,\cdot)$ and 
$\Hom_R(\cdot,A) \Rightarrow \Hom_R(\cdot,B)$ as well as
$\cdot \otimes_R A \Rightarrow \cdot \otimes_R B$ and
$A \otimes_R \cdot \Rightarrow B \otimes_R \cdot$.
\end{example}

\begin{proof}
Exercise~\ref{exnattrans}.
\end{proof}

\begin{definition}
\begin{itemize}
\item A covariant functor $F: \cC \to \calD$ is called {\em left-exact}, if for every exact sequence
$$ 0 \to A \to B \to C$$
the sequence
$$ 0 \to F(A) \to F(B) \to F(C)$$
is also exact.
\item A contravariant functor $F: \cC \to \calD$ is called {\em left-exact}, if for every exact sequence
$$ A \to B \to C \to 0$$
the sequence
$$ 0 \to F(C) \to F(B) \to F(A)$$
is also exact.
\item A covariant functor $F: \cC \to \calD$ is called {\em right-exact}, if for every exact sequence
$$ A \to B \to C \to 0$$
the sequence
$$ F(A) \to F(B) \to F(C) \to 0$$
is also exact.
\item A contravariant functor $F: \cC \to \calD$ is called {\em right-exact}, if for every exact sequence
$$ 0 \to A \to B \to C$$
the sequence
$$ F(C) \to F(B) \to F(A) \to 0$$
is also exact.
\item A covariant or contravariant functor is {\em exact} if it is both left-exact and right-exact.
\end{itemize}
\end{definition}

\begin{example}\label{HomTensor}
Both functors $\Hom_R(\cdot,M)$ and $\Hom_R(M,\cdot)$ for $M \in \obj(R-\modules)$
are left-exact.
Both functors $\cdot \otimes_R M$ for $M \in \obj(R-\modules)$
and $M \otimes_R \cdot$ for $M \in \obj(\modules-R)$
are right-exact.
\end{example}

\begin{proof}
Exercise~\ref{exHomTensor}.
\end{proof}

\begin{definition}
Let $R$ be a not necessarily commutative ring.
A left $R$-module $P$ is called {\em projective} if the functor
$\Hom_R(P,\cdot)$ is exact.
A left $R$-module $I$ is called {\em injective} if the functor
$\Hom_R(\cdot,I)$ is exact.
\end{definition}

\begin{lemma}\label{lemprojective}
Let $R$ be a not necessarily commutative ring and let $P$ be a left $R$-module.
Then $P$ is projective if and only if $P$ is a direct summand of some
free $R$-module.
In particular, free modules are projective.
\end{lemma}

\begin{proof}
Exercise~\ref{exprojective}.
\end{proof}

\subsection{Theory: Complexes and Cohomology}

\begin{definition}
A {\em (right) chain complex} $C_\bullet$ in the category $R-\modules$ is a collection of
objects $C_n \in \obj(R-\modules)$ for $n \ge m$ for some $m \in \ZZ$ together with homomorphisms
$C_{n+1} \xrightarrow{\partial_{n+1}} C_n$, i.e.\
$$ \cdots \to C_{n+1} \xrightarrow{\partial_{n+1}} C_{n} \xrightarrow{\partial_{n}}
C_{n-1} \to \cdots \to C_{m+2} \xrightarrow{\partial_{m+2}} C_{m+1} \xrightarrow{\partial_{m+1}} C_m
\xrightarrow{\partial_m} 0,$$
such that
$$ \partial_{n} \circ \partial_{n+1} = 0 $$
for all $n \ge m$.
The group of {\em $n$-cycles} of this chain complex is defined as
$$ \Z_n(C_\bullet) = \ker(\partial_n).$$
The group of {\em $n$-boundaries} of this chain complex is defined as
$$ \B_n(C_\bullet) = \Image(\partial_{n+1}).$$
The {\em $n$-th homology group} of this chain complex is defined as
$$ \h_n(C_\bullet) = \ker(\partial_n)/\Image(\partial_{n+1}).$$
The chain complex $C_\bullet$ is {\em exact} if $\h_n(C_\bullet) = 0$ for all~$n$.
If $C_\bullet$ is exact and $m = -1$, one often says that $C_\bullet$ is a
{\em resolution} of~$C_{-1}$.

A {\em morphism of right chain complexes} $\phi_\bullet: C_\bullet \to D_\bullet$ is a collection of
homomorphisms $\phi_n: C_n \to D_n$ for $n \in \NN_0$ such that all the diagrams
$$ \begin{CD}
C_{n+1} @>{\partial_{n+1}}>> C_n \\
@V{\phi_{n+1}}VV @V{\phi_n}VV \\
D_{n+1} @>{\partial_{n+1}}>> D_n
\end{CD} $$
are commutative.

If all $\phi_n$ are injective, we regard $C_\bullet$ as a sub-chain complex of $D_\bullet$.
If all $\phi_n$ are surjective, we regard $D_\bullet$ as a quotient complex of $C_\bullet$.
\end{definition}

\begin{definition}
A {\em (right) cochain complex} $C^\bullet$ in the category $R-\modules$ is a collection of
objects $C^n \in \obj(R-\modules)$ for $n \ge m$ for some $m \in \ZZ$ together with homomorphisms
$C^n \xrightarrow{\partial^{n+1}} C^{n+1}$, i.e.\
$$ 0 \xrightarrow{\partial^m} C^m \xrightarrow{\partial^{m+1}} C^{m+1}
 \xrightarrow{\partial^{m+2}} C^{m+2} \to \cdots \to
C^{n-1} \xrightarrow{\partial^{n}} C^{n} \xrightarrow{\partial^{n+1}} C^{n+1} \to \cdots,$$
such that
$$ \partial^{n+1} \circ \partial^{n} = 0 $$
for all $n \ge m$.
The group of {\em $n$-cocycles} of this cochain complex is defined as
$$ \Z^n(C_\bullet) = \ker(\partial^{n+1}).$$
The group of {\em $n$-coboundaries} of this cochain complex is defined as
$$ \B^n(C_\bullet) = \Image(\partial_{n}).$$
The {\em $n$-th cohomology group} of this cochain complex is defined as
$$ \h^n(C^\bullet) = \ker(\partial^{n+1})/\Image(\partial^{n}).$$
The cochain complex $C^\bullet$ is {\em exact} if $\h^n(C_\bullet) = 0$ for all~$n$.
If $C^\bullet$ is exact and $m = -1$, one often says that $C^\bullet$ is a
{\em resolution} of~$C^{-1}$.

A {\em morphism of right cochain complexes} $\phi^\bullet: C^\bullet \to D^\bullet$ is a collection of
homomorphisms $\phi^n: C^n \to D^n$ for $n \in \NN_0$ such that all the diagrams
$$ \begin{CD}
C^n @>{\partial^{n+1}}>> C^{n+1} \\
@V{\phi^{n}}VV @V{\phi^{n+1}}VV \\
D^n @>{\partial^{n+1}}>> D^{n+1}
\end{CD} $$
are commutative.

If all $\phi^n$ are injective, we regard $C^\bullet$ as a sub-chain complex of $D^\bullet$.
If all $\phi^n$ are surjective, we regard $D^\bullet$ as a quotient complex of $C^\bullet$.
\end{definition}

In Exercise~\ref{excomplexes} you are asked to define kernels, cokernels and images
of morphisms of cochain complexes and to show that morphisms of cochain complexes
induce natural maps on the cohomology groups. In fact, cochain complexes of $R$-modules
form an abelian category.

\subsection*{Example: standard resolution of a group}

Let $G$ be a group and $R$ a commutative ring.
Write $G^n$ for the $n$-fold direct product $G \times  \dots \times G$ and
equip $R[G^n]$ with the diagonal $R[G]$-action.
We describe the {\em standard resolution} $F(G)_\bullet$ of 
$R$ by free $R[G]$-modules:
$$ \xymatrix@=1cm{
  0  \ar@{<-}[r] \ar@{<-}@/^1pc/[rr]^{\partial_0}
& R \ar@{<-}[r]^(.4){\epsilon}
& F(G)_0 := R[G] \ar@{<-}[r]^{\partial_1}
& F(G)_1 := R[G^2] \ar@{<-}[r]^(.7){\partial_2}
& \cdots,}$$
where we put (the ``hat'' means that we leave out that element):
$$ \partial_n := \sum_{i=0}^{n} (-1)^i d_i  \;\;\text{ and } \;\;
d_i (g_0, \dots, g_n) := (g_0, \dots, \hat{g_i}, \dots, g_n).$$
The map $\epsilon$ is the usual augmentation map defined by sending
$g \in G$ to~$1 \in R$.
By `standard resolution' we refer to the straight maps. We have included
the bended arrow $\partial_0$, which is $0$ by definition, because it will be needed
in the definition of group cohomology (Definition~\ref{defi:gp-coh}).
In Exercise~\ref{exstandardresolution} you are asked to check that
the standard resolution is indeed a resolution, i.e.\ that the above
complex is exact.

\subsection*{Example: bar resolution of a group}

We continue to treat the standard resolution $R$ by $R[G]$-modules, but we will write
it differently. \cite{Weibel} calls the following the
{\em unnormalised bar resolution} of~$G$. We shall simply say {\em bar resolution}.
If we let $h_r := g_{r-1}^{-1} g_r$, then we get the identity
$$(g_0, g_1, g_2, \dots, g_n) = g_0 . (1, h_1, h_1 h_2, \dots, h_1 h_2 \dots h_n ) =: 
    g_0 . [h_1|h_2| \dots h_n].$$
The symbols $[h_1|h_2| \dots |h_n]$ 
with arbitrary $h_i \in G$ hence form an $R[G]$-basis of~$F(G)_n$, and
one has $F(G)_n = R[G] \otimes_R (\text{free $R$-module on } [h_1|h_2| \dots |h_n])$.
One computes the action of $\partial_n$ on this basis and gets
$\partial_n = \sum_{i=0}^n (-1)^i d_i$ where
$$
  d_i([h_1| \dots |h_n]) = 
\begin{cases}
h_1 \text{[} h_2| \dots |h_n \text{]}                       & i = 0 \\
\text{[} h_1| \dots |h_i h_{i+1}| \dots |h_n \text{]}   & 0 < i < n \\
\text{[}h_1| \dots |h_{n-1}\text{]}                   & i = n. 
\end{cases} $$
We will from now on, if confusion is unlikely, simply write $(h_1,\dots,h_n)$
instead of $[h_1|\dots|h_n]$.

\subsection*{Example: resolution of a cyclic group}\label{seccyclic}

Let $G=\langle T \rangle$ be an infinite cyclic group (i.e.\ a group isomorphic
to $(\ZZ,+)$). Here is a very simple resolution of~$R$ by free $R[G]$-modules:
\begin{equation}\label{eq:res-cyclic-infinite}
0 \to R[G] \xrightarrow{T-1} R[G] \xrightarrow{\epsilon} R \to 0.
\end{equation}
Let now $G=\langle \sigma \rangle$ be a finite cyclic group of order~$n$ and let $N_\sigma := \sum_{i=0}^{n-1} \sigma^i$.
Here is a resolution of~$R$ by free $R[G]$-modules:
\begin{equation}\label{eq:res-cyclic-finite}
 \cdots \to R[G] \xrightarrow{N_\sigma} R[G] \xrightarrow{1-\sigma} R[G]
\xrightarrow{N_\sigma} R[G]  \xrightarrow{1-\sigma} R[G] \xrightarrow{\epsilon} R \to 0.
\end{equation}
In Exercise~\ref{excyclic} you are asked to verify the exactness of these two sequences.

\subsection*{Group cohomology}

A standard reference for group cohomology is~\cite{Brown}.

\begin{definition}\label{defi:gp-coh}
Let $R$ be a ring, $G$ a group and $M$ a left $R[G]$-module.
Recall that $F(G)_\bullet$ denotes the standard resolution of~$R$ by
free $R[G]$-modules.
\begin{enumerate}[(a)]
\item Let $M$ be a left $R[G]$-module.
When we apply the functor $\Hom_{R[G]}(\cdot,M)$ to the standard
resolution $F(G)_\bullet$ cut off at~$0$ (i.e.\ $F(G)_1 \xrightarrow{\partial_1} F(G)_0
\xrightarrow{\partial_0} 0$), we get the cochain complex $\Hom_{R[G]}(F(G)_\bullet,M)$: 
$$ \cdots \to \Hom_{R[G]}(F(G)_{n-1},M) \xrightarrow{\partial^{n}} \Hom_{R[G]}(F(G)_{n},M) 
\xrightarrow{\partial^{n+1}} \Hom_{R[G]}(F(G)_{n+1},M) \to \cdots .$$
Define the $n$-th cohomology group of~$G$ with values in the $G$-module~$M$ as
$$ \h^n(G,M) := \h^n(\Hom_{R[G]}(F(G)_\bullet,M)).$$

\item Let $M$ be a right $R[G]$-module.
When we apply the functor $M \otimes_{R[G]} \cdot $ to the standard
resolution $F(G)_\bullet$ cut off at~$0$ we get the chain complex 
$M \otimes_{R[G]} F(G)_\bullet$ : 
$$ \cdots \to M \otimes_{R[G]} F(G)_{n+1} \xrightarrow{\partial_{n+1}} M \otimes_{R[G]} F(G)_{n}
\xrightarrow{\partial_{n}} M \otimes_{R[G]} F(G)_{n-1} \to \cdots .$$

Define the $n$-th homology group of~$G$ with values in the $G$-module~$M$ as
$$ \h_n(G,M) := \h_n(M \otimes_{R[G]} F(G)_\bullet).$$
\end{enumerate}
\end{definition}

In this lecture we shall only use group cohomology.
As a motivation for looking at group cohomology in this lecture, we
can already point out that
$$ \h^1(\Gamma_1(N),V_{k-2}(R)) \cong \cM_k(\Gamma_1(N),R),$$
provided that $6$ is invertible in~$R$ (see Theorem~\ref{compthm}).
The reader is invited to compute explicit descriptions of $\h^0$, $\h_0$ and $\h^1$
in Exercise~\ref{exsmallh}.

\subsection{Theory: Cohomological Techniques}

The cohomology of groups fits into a general machinery, namely that of derived
functor cohomology. Derived functors are universal cohomological $\delta$-functors
and many properties of them can be derived in a purely formal way from
the universality. What this means will be explained in this section. We omit all
proofs.

\begin{definition}
Let $\cC$ and $\calD$ be (abelian) categories (for instance, $\cC$ the right cochain complexes
of $R-\modules$ and $\calD = R-\modules$).
A {\em positive covariant cohomological $\delta$-functor}
between $\cC$ and $\calD$ is a collection
of functors $\h^n: \cC \to \calD$ for $n \ge 0$ together with {\em connecting morphisms}
$$\delta^n: \h^{n} (C) \to \h^{n+1} (A)$$ 
which are defined for every short exact sequence
$0 \to A \to B \to C \to 0$ in~$\cC$
such that the following hold:

\begin{enumerate}[(a)]
\item (Positivity) $\h^n$ is the zero functor if $n < 0$.
\item
For every short exact sequence 
$0 \to A \to B \to C \to 0$ in~$\cC$ there is the {\em long exact sequence} in~$\calD$:
$$ \cdots \to \h^{n-1}(C) \xrightarrow{\delta^{n-1}} \h^n (A) \to \h^n(B) \to \h^n(C)
\xrightarrow{\delta^{n}} \h^{n+1} (A) \to \cdots,$$
where the maps $\h^n (A) \to \h^n(B) \to \h^n(C)$ are those that are induced from
the homomorphisms in the exact sequence $0 \to A \to B \to C \to 0$.
\item
For every commutative diagram in~$\cC$
$$ \begin{CD}
0 @>>> A @>>> B @>>> C @>>> 0 \\
&& @VfVV @VgVV @VhVV \\
0 @>>> A' @>>> B' @>>> C' @>>> 0 
\end{CD} $$
with exact rows the following diagram in $\calD$ commutes, too:
$$ \begin{CD}
\h^{n-1} (C) @>{\delta^{n-1}}>> \h^{n} (A) @>>> \h^{n} (B) @>>> \h^{n} (C) @>{\delta^{n}}>> \h^{n+1} (A) \\
@V{\h^{n-1} (h)}VV @V{\h^{n} (f)}VV @V{\h^{n} (g)}VV @V{\h^{n} (h)}VV @V{\h^{n+1} (f)}VV\\
\h^{n-1} (C') @>{\delta^{n-1}}>> \h^{n} (A') @>>> \h^{n} (B') @>>> \h^{n} (C') @>{\delta^{n}}>> \h^{n+1} (A')
\end{CD} $$
\end{enumerate}
\end{definition}

\begin{theorem}\label{thm:deltafun}
Let $R$ be a ring (not necessarily commutative).
Let $\cC$ stand for the category of cochain complexes of left $R$-modules.
Then the cohomology functors
$$ \h^n: \cC \to R-\modules, \;\;\; C^\bullet \mapsto \h^n(C^\bullet)$$
form a cohomological $\delta$-functor.
\end{theorem}

\begin{proof}
This theorem is proved by some 'diagram chasing' starting from the
snake lemma. See Chapter~1 of~\cite{Weibel} for details.
\end{proof}

It is not difficult to conclude that group cohomology also forms
a cohomological $\delta$-functor.

\begin{proposition}\label{gpcohfunctor}
Let $R$ be a commutative ring and $G$ a group.
\begin{enumerate}[(a)]
\item The functor from $R[G]-\modules$ to cochain complexes of $R[G]-\modules$
which associates to a left $R[G]$-module~$M$ the cochain complex
$\Hom_{R[G]} (F(G)_\bullet,M)$ with $F(G)_\bullet$ the bar resolution of $R$
by free $R[G]$-modules is exact, i.e.\ it takes an exact sequence
$0 \to A \to B \to C \to 0$ of $R[G]$-modules to the exact sequence
$$0 \to \Hom_{R[G]}(F(G)_\bullet,A) \to \Hom_{R[G]}(F(G)_\bullet,B) 
    \to \Hom_{R[G]}(F(G)_\bullet,C) \to 0$$
of cochain complexes.
\item The functors 
$$ \h^n(G,\cdot): R[G]-\modules \to R-\modules, \;\;\; M \mapsto \h^n(G,M)$$
form a positive cohomological $\delta$-functor.
\end{enumerate}
\end{proposition}

\begin{proof}
Exercise~\ref{exgpcohfunctor}.
\end{proof}

We will now come to universal $\delta$-functors. Important examples of such
(among them group cohomology)
are obtained from injective resolutions. Although the following discussion
is valid in any abelian category (with enough injectives), we restrict to
$R-\modules$ for a not necessarily commutative ring~$R$.

\begin{definition}
Let $R$ be a not necessarily commutative ring and let $M \in \obj(R-\modules)$.

A {\em projective resolution} of~$M$ is a resolution
$$ \cdots \to P_2 \xrightarrow{\partial_2} P_1 \xrightarrow{\partial_1} P_0 \to M \to 0,$$
i.e.\ an exact chain complex, in which all the $P_n$ for $n \ge 0$ are projective
$R$-modules.

An {\em injective resolution} of~$M$ is a resolution
$$ 0 \to M \to I^0 \xrightarrow{\partial^1} I^1 \xrightarrow{\partial^2} I^2 \to \cdots,$$
i.e.\ an exact cochain complex, in which all the $I^n$ for $n \ge 0$ are injective
$R$-modules.
\end{definition}

We state the following lemma as a fact. It is easy for projective resolutions
and requires work for injective ones (see e.g.\ \cite{Eisenbud}).

\begin{lemma}
Injective and projective resolutions exist in the category of $R$-modules,
where $R$ is any ring (not necessarily commutative).
\end{lemma}

Note that applying a left exact covariant functor $\cF$ to an injective resolution 
$$0 \to M \to I^0 \to I^1 \to I^2 \to \cdots$$
of~$M$ gives rise to a cochain complex
$$ 0 \to \cF(M) \to \cF(I^0) \to \cF(I^1) \to \cF(I^2) \to \cdots,$$
of which only the part $0 \to \cF(M) \to \cF(I^0) \to \cF(I^1)$ need be exact.
This means that the $\h^0$ of the (cut off at~$0$) cochain complex
$\cF(I^0) \to \cF(I^1) \to \cF(I^2) \to \cdots$ is equal to $\cF(M)$.

\begin{definition}
Let $R$ be a not necessarily commutative ring. 
\begin{enumerate}[(a)]
\item Let $\cF$ be a left exact covariant functor
on the category of $R$-modules (mapping for instance to $\ZZ-\modules$).

The {\em right derived functors} $R^n\cF(\cdot)$ of~$\cF$ are the functors
on the category of $R-\modules$ defined as follows.
For $M \in \obj(R-\modules)$ choose an injective resolution
$0 \to M \to I^0 \to I^1 \to \cdots$ and let
$$ R^n\cF(M) := \h^n \big(\cF(I^0) \to \cF(I^1) \to \cF(I^2) \to \cdots\big).$$

\item Let $\cG$ be a left exact contravariant functor
on the category of $R$-modules.

The {\em right derived functors} $R^n\cG(\cdot)$ of~$\cG$ are the functors
on the category of $R-\modules$ defined as follows.
For $M \in \obj(R-\modules)$ choose a projective resolution
$\cdots \to P_1 \to P_0 \to M \to 0$ and let
$$ R^n\cG(M) := \h^n \big(\cG(P_0) \to \cG(P_1) \to \cG(P_2) \to \cdots\big).$$
\end{enumerate}
\end{definition}

We state the following lemma without a proof. It is a simple consequence of
the injectivity respectively projectivity of the modules in the resolution.

\begin{lemma}
The right derived functors do not depend on the choice of the resolution
and they form a cohomological $\delta$-functor.
\end{lemma}

Of course, one can also define left derived functors of right exact functors.
An important example is the $\Tor$-functor which is obtained by deriving the
tensor product functor in a way dual to~$\Ext$ (see below).
As already mentioned, the importance of right and left derived functors comes
from their universality.

\begin{definition}\label{defi:deltafunctor}
\begin{enumerate}[(a)]
\item
Let $(\h^n)_n$ and $(T^n)_n$ be cohomological $\delta$-functors.
A {\em morphism of cohomological $\delta$-functors} is a collection
of natural transformations $\eta^n: \h^n \Rightarrow T^n$ that commute
with the connecting homomorphisms~$\delta$, i.e.\
for every short exact sequence $0 \to A \to B \to C \to 0$ and every~$n$
the diagram
$$ \begin{CD}
\h^n(C) @>{\delta}>> \h^{n+1}(A) \\
 @V{\eta_C^n}VV @V{\eta_A^{n+1}}VV \\
T^n(C) @>{\delta}>> T^{n+1}(A) 
\end{CD} $$
commutes.

\item The cohomological $\delta$-functor $(\h^n)_n$ is {\em universal}
if for every cohomological $\delta$-functor $(T^n)_n$
and every natural transformation $\eta^0: \h^0(\cdot) \Rightarrow T^0(\cdot)$
there is a unique natural transformation $\eta^n: \h^n(\cdot) \Rightarrow T^n(\cdot)$
for all~$n\ge 1$ such that the $\eta^n$ form a morphism of 
cohomological $\delta$-functors between $(\h^n)_n$ and $(T^n)_n$.
\end{enumerate}
\end{definition}

For the proof of the following central result we refer to \cite{Weibel}, Chapter~2.

\begin{theorem}
Let $R$ be a not necessarily commutative ring and
let $\cF$ be a left exact covariant or contravariant functor
on the category of $R$-modules (mapping for instance to $\ZZ-\modules$).

The {\em right derived functors} $(R^n\cF(\cdot))_n$ of~$\cF$ 
form a \underline{universal} cohomological $\delta$-functor.
\end{theorem}

\begin{example}
\begin{enumerate}[(a)]
\item Let $R$ be a commutative ring and $G$ a group.
The functor 
$$(\cdot)^G: R[G]-\modules \to R-\modules, \;\;\; M \mapsto M^G$$
is left exact and covariant, hence we can form its right derived functors
$R^n(\cdot)^G$. Since we have the special case $(R^0(\cdot)^G)(M) = M^G$,
universality gives a morphism of cohomological $\delta$-functors
$R^n(\cdot)^G \Rightarrow H^n(G,\cdot)$. We shall see that this is an isomorphism.
\item Let $R$ be a not necessarily commutative ring. We have seen that
the functors $\Hom_R(\cdot,M)$ and $\Hom_R(M,\cdot)$ are left exact.
We write
$$ \Ext_R^n (\cdot,M) := R^n \Hom_R(\cdot,M) \;\;\; \text{ and }\;\;\; 
\Ext_R^n (M,\cdot) := R^n \Hom_R(M,\cdot).$$
See Theorem~\ref{extbalanced} below.
\item Many cohomology theories in (algebraic) geometry are also
of a right derived functor nature. For instance, let $X$ be a topological
space and consider the category of sheaves of abelian groups on~$X$.
The global sections functor $\cF \mapsto \cF(X) = \h^0(X,\cF)$ is left exact
and its right derived functors $R^n(\h^0(X,\cdot))$ can be formed.
They are usually denoted by $\h^n(X,\cdot)$ and they define 'sheaf cohomology'
on~$X$. Etale cohomology is an elaboration of this based on a generalisation
of topological spaces.
\end{enumerate}
\end{example}

\subsection*{Universal properties of group cohomology}

\begin{theorem}\label{extbalanced}
Let $R$ be a not necessarily commutative ring.
The $\Ext$-functor is {\em balanced}. This means that for any two
$R$-modules $M,N$ there are isomorphisms
$$ (\Ext^n_R (\cdot,N))(M) \cong (\Ext^n_R (M,\cdot))(N) =: \Ext_R^n(M,N).$$
\end{theorem}

\begin{proof}
\cite{Weibel}, Theorem~2.7.6.
\end{proof}

\begin{corollary}\label{gpcoh}
Let $R$ be a commutative ring and $G$ a group. For every $R[G]$-module~$M$
there are isomorphisms
$$ H^n(G,M) \cong \Ext_{R[G]}^n(R,M) \cong (R^n(\cdot)^G)(M)$$
and the functors $(H^n(G,\cdot))_n$ form a universal cohomological $\delta$-functor.
Moreover, apart from the standard resolution of $R$ by free $R[G]$-modules,
any resolution of $R$ by projective $R[G]$-modules may be used to
compute $H^n(G,M)$.
\end{corollary}

\begin{proof}
We may compute
$\Ext_{R[G]}^n(\cdot,M) (R)$ by any resolution of~$R$ by projective $R[G]$-modules.
Our standard resolution is such a resolution, since any free module is projective.
Hence, $H^n(G,M) \cong \Ext_{R[G]}^n(\cdot,M)(R)$.
The key is now that $\Ext$ is balanced (Theorem~\ref{extbalanced}), since it
gives $H^n(G,M) \cong \Ext_{R[G]}^n(R,\cdot)(M) \cong R^n(\cdot)^G(M) \cong \Ext^n_{R[G]}(R,M)$.
As the $\Ext$-functor is universal (being a right derived functor), also
$H^n(G,\cdot)$ is universal. For the last statement we recall that
right derived functors do not depend on the chosen projective respectively
injective resolution.
\end{proof}

You are invited to look at Exercise~\ref{exfree} now.

\subsection{Theory: Generalities on Group Cohomology}

We now apply the universality of the $\delta$-functor
of group cohomology. Let $\phi: H \to G$ be a group
homomorphism and $A$ an $R[G]$-module. 
Via~$\phi$ we may
consider $A$ also as an $R[H]$-module and
$\res^0: \h^0(G,\cdot) \to \h^0(H,\cdot)$ is a natural
transformation. By the universality of $\h^\bullet(G,\cdot)$ we get
natural transformations
$$\res^n: \h^n(G,\cdot) \to \h^n(H,\cdot).$$
These maps are called {\em restrictions}. See Exercise~\ref{excochains}
for a description in terms of cochains. Very often
$\phi$ is just the embedding map of a subgroup.

Assume now that $H$ is a normal subgroup of~$G$ and $A$
is an $R[G]$-module. Then we can consider
$\phi: G \to G/H$ and the restriction above gives
natural transformations
$\res^n: \h^n(G/H,(\cdot)^H) \to \h^n(G,(\cdot)^H)$.
We define the {\em inflation maps} to be
$$\infl^n: \h^n(G/H,A^H) \xrightarrow{\res^n} \h^n(G,A^H)
\longrightarrow \h^n(G,A)$$
where the last arrow is induced from the natural inclusion $A^H \hookrightarrow A$.

Under the same assumptions, conjugation by~$g \in G$ preserves~$H$
and we have the isomorphism $H^0(H,A) = A^H \xrightarrow{a \mapsto ga} A^H = H^0(H,A)$. 
Hence by universality we obtain natural maps $H^n(H,A) \to H^n(H,A)$ for every
$g \in G$. One even gets an $R[G]$-action on $\h^n(H,A)$.
As $h \in H$ is clearly the identity on $\h^0(H,A)$, the above action 
is in fact also an $R[G/H]$-action.

Let now $H \le G$ be a subgroup of finite index.
Then the norm $N_{G/H} := \sum_{\{g_i\}} \in R[G]$ with $\{g_i\}$
a system of representatives of $G/H$ gives a natural
transformation
$\cores^0: \h^0(H,\cdot) \to \h^0(G,\cdot)$
where $\cdot$ is an $R[G]$-module. By universality
we obtain
$$\cores^n: \h^n(H,\cdot) \to \h^n(G,\cdot),$$
the {\em corestriction (transfer)} maps.

The inflation map, the $R[G/H]$-action and the corestriction
can be explicitly described in terms of cochains of the bar resolution
(see Exercise~\ref{excochains}).

It is clear that $\cores^0 \circ \res^0$ is multiplication
by the index $(G:H)$. By universality, also $\cores^n \circ \res^n$ is multiplication
by the index $(G:H)$.
Hence we have proved the first part of the following proposition.

\begin{proposition}\label{corres}
\begin{enumerate}[(a)]
\item Let $H < G$ be a subgroup of finite index $(G:H)$.
For all $i$ and all $R[G]$-modules $M$ one has the equality
$$  \cores_H^G \circ \res_H^G = (G:H) $$
on all $\h^i(G,M)$.
\item Let $G$ be a finite group of order~$n$ and $R$
a ring in which $n$ is invertible. Then
$\h^i(G,M) = 0$ for all $i\ge 1$ and all $R[G]$-modules~$M$.
\end{enumerate}
\end{proposition}

\begin{proof}
Part~(b) is an easy consequence with $H=1$, since
$$ \h^i(G,M) \xrightarrow{\res_H^G} \h^i(1,M) \xrightarrow{\cores_H^G} \h^i(G,M) $$
is the zero map (as $\h^i(1,M)=0$ for $i\ge 1$), but it also is multiplication by~$n$.
\end{proof}

The following exact sequence turns out to be very important for
our purposes.

\begin{theorem}[Hochschild-Serre]\label{thm:hochschild-serre}
Let $H \le G$ be a normal subgroup and $A$ an $R[G]$-module.
There is the exact sequence:
$$0 \to \h^1(G/H,A^H) \xrightarrow{\infl} \h^1(G,A) \xrightarrow{\res}
\h^1(G,A)^{G/H} \to \h^2(G/H,A^H) \xrightarrow{\infl} \h^2(G,A).$$
\end{theorem}

\begin{proof}
We only sketch the proof for those who know spectral sequences.
It is, however, possible to verify the exactness on cochains
explicitly (after having defined the missing map appropriately).
Grothendieck's theorem on spectral sequences (\cite{Weibel}, 6.8.2)
associates to the composition of functors
$$ (A \mapsto A^H \mapsto (A^H)^{G/H}) = (A \mapsto A^G)$$
the spectral sequence
$$ E^{p,q}_2: H^p(G/H,H^q(H,A)) \Rightarrow H^{p+q}(G,A).$$
The statement of the theorem is then just the $5$-term sequence that
one can associate with every spectral sequence of this type.
\end{proof}

\subsection*{Coinduced modules and Shapiro's Lemma}

Let $H < G$ be a subgroup and $A$ be a left $R[H]$-module.
The $R[G]$-module
$$\Coind_H^G (A) := \Hom_{R[H]} (R[G], A)$$ 
is called the {\em coinduction} or the {\em coinduced module} 
from $H$ to $G$ of $A$.
We make $\Coind_H^G (A)$ into a left $R[G]$-module by
$$ (g.\phi)(g') = \phi(g'g) \;\; \forall\, g,g' \in G,\, \phi \in \Hom_{R[H]} (R[G], A).$$

\begin{proposition}[Shapiro's Lemma]\label{shapiro}
For all $n \ge 0$, the map
$$ \Sh: \h^n (G, \Coind_H^G (A)) \to \h^n (H, A)$$
given on cochains is given by
$$ c \mapsto ((h_1,\dots,h_n) \to (c(h_1,\dots,h_n))(1_G))$$
is an isomorphism.
\end{proposition}

\begin{proof}
Exercise~\ref{exshapiro}.
\end{proof}

\subsection*{Mackey's formula and stabilisers}

If $H \le G$ are groups and $V$ is an $R[G]$-module, we denote by
$\Res_H^G(V)$ the module~$V$ considered as an $R[H]$-module if we want to stress that the module is
obtained by restriction. In later sections, we will often silently restrict modules to subgroups.

\begin{proposition}\label{propmackey}
Let $R$ be a ring, $G$ be a group and $H,K$ subgroups of~$G$. 
Let furthermore $V$ be an $R[H]$-module. {\em Mackey's formula} is the isomorphism
$$ \Res_K^G \Coind_H^G V \cong \prod_{g \in H\backslash G / K}
\Coind_{K\cap g^{-1}Hg}^K {}^g (\Res^H_{H\cap gKg^{-1}} V).$$
Here ${}^g (\Res^H_{H\cap gKg^{-1}} V)$ denotes the
$R[K \cap g^{-1}Hg]$-module obtained from $V$ via the 
conjugated action $g^{-1}hg ._g v := h. v$ for $v \in V$
and $h \in H$ such that $g^{-1}hg \in K$.
\end{proposition}

\begin{proof}
We consider the commutative diagram
$$ \xymatrix@=0.5cm{
\Res_K^G \Hom_{R[H]}(R[G],V) \ar@{->}[r]  \ar@{->}[dr] &
\prod_{g \in H\backslash G / K}
\Hom_{R[K\cap g^{-1}Hg]}(R[K], {}^g (\Res^H_{H\cap gKg^{-1}} V)) \ar@{->}^\sim[d] \\
&\prod_{g \in H\backslash G / K}
 \Hom_{R[H\cap gKg^{-1}]}(R[gKg^{-1}], \Res^H_{H\cap gKg^{-1}} V)).}$$
The vertical arrow is just given by conjugation and is clearly
an isomorphism.
The diagonal map is the product of the natural restrictions.
From the bijection
$$ \big(H \cap gKg^{-1}\big) \backslash gKg^{-1} 
\xrightarrow{gkg^{-1} \mapsto Hgk} H \backslash HgK$$
it is clear that also the diagonal map is an isomorphism,
proving the proposition.
\end{proof}

From Shapiro's Lemma \ref{shapiro} we directly get the following.

\begin{corollary}\label{cormackey}
In the situation of Proposition~\ref{propmackey} one has
\begin{align*}
\h^i(K,\Coind_H^G V) 
& \cong \prod_{g \in H\backslash G / K}
\h^i(K \cap g^{-1}Hg,  {}^g (\Res^H_{H\cap gKg^{-1}} V) \\
& \cong \prod_{g \in H\backslash G / K}
\h^i(H \cap gKg^{-1},  \Res^H_{H \cap gKg^{-1}} V) 
\end{align*}
for all $i \in \NN$.
\end{corollary}

\subsection{Theoretical exercises}

\begin{exercise}\label{exnattrans}
Check the statements made in Example~\ref{nattrans}.
\end{exercise}

\begin{exercise}\label{exHomTensor}
Verify the statements of Example~\ref{HomTensor}.
\end{exercise}

\begin{exercise}\label{exprojective}
Prove Lemma~\ref{lemprojective}.

Hint: take a free $R$-module $F$ which surjects onto~$P$, {\em i.e.} $\pi: F \twoheadrightarrow P$,
and use the definition of $P$ being projective to show that the surjection admits a split $s: P \to F$, meaning that
$\pi \circ s$ is the identity on~$P$. This is then equivalent to the assertion.
\end{exercise}

\begin{exercise}\label{excomplexes}
Let $\phi^\bullet: C^\bullet \to D^\bullet$ be a morphism of cochain complexes.
\begin{enumerate}[(a)]
\item Show that $\ker(\phi^\bullet)$ is a cochain complex and is a subcomplex of $C^\bullet$ in a 
natural way.
\item Show that $\Image(\phi^\bullet)$ is a cochain complex and is a subcomplex of $D^\bullet$ in a 
natural way.
\item Show that $\coker(\phi^\bullet)$ is a cochain complex and is a quotient of $D^\bullet$ in a 
natural way.
\item Show that $\phi^\bullet$ induces homomorphisms 
$\h^n(C^\bullet) \xrightarrow{\h^n(\phi^\bullet)} \h^n(D^\bullet)$
for all $n \in \NN$.
\end{enumerate}
\end{exercise}

\begin{exercise}\label{exstandardresolution}
Check the exactness of the standard resolution of a group~$G$.
\end{exercise}

\begin{exercise}\label{excyclic}
Check the exactness of the resolutions \eqref{eq:res-cyclic-infinite} and \eqref{eq:res-cyclic-finite} for an infinite and a finite cyclic group,
respectively.
\end{exercise}

\begin{exercise}\label{exsmallh}
Let $R$, $G$, $M$ be as in the definition of group (co-)homology.
\begin{enumerate}[(a)]
\item Prove $\h^0(G,M) \cong M^G$, the $G$-invariants of~$M$.
\item Prove $\h_0(G,M) \cong M_G$, the $G$-coinvariants of~$M$.
\item Prove the explicit descriptions:
\begin{align*}
\Z^1(G,M) &= \{ f: G \to M \text{ map } |\;  f(gh) = g.f(h) + f(g) \; \forall g,h \in G \},\\
\B^1(G,M) &= \{ f: G \to M \text{ map } |\;  \exists m \in M : f(g) = (1-g)m \; \forall g \in G\},\\
\h^1(G,M) &= Z^1(G,M) / B^1(G,M).
\end{align*}
In particular, if the action of $G$ on~$M$ is trivial, the boundaries $B^1(G,M)$ are zero,
and one has:
$$ \h^1(G,M) = \Hom_{\textnormal{group}}(G,M).$$
\end{enumerate}
\end{exercise}

\begin{exercise}\label{exgpcohfunctor}
Prove Proposition~\ref{gpcohfunctor}.

Hint: for (a), use that free modules are projective.
(b) follows from (a) together with Theorem~\ref{thm:deltafun} or, alternatively, by direct calculation.
See also \cite[III.6.1]{Brown}.
\end{exercise}

\begin{exercise}\label{exfree}
Let $R$ be a commutative ring.
\begin{enumerate}[(a)]
\item
Let $G= \langle T \rangle$ be a free cyclic group and $M$ any $R[G]$-module. Prove
$$\h^0(G, M) = M^G, \;\;\;\h^1(G, M) = M / (1-T)M \;\;\;\text{ and }\;\;\; \h^i(G,M)= 0$$
for all $i \ge 2$.
\item For a finite cyclic group~$G = \langle \sigma \rangle$ of order~$n$ and any $R[G]$-module~$M$ prove that
\begin{align*}
 \h^0(G,M)  &\cong M^G,            &\h^1(G, M) &\cong \{m \in M \;|\; N_\sigma m = 0\} / (1-\sigma)M, \\
 \h^2(G, M) &\cong M^G/N_\sigma M, &\h^i(G,M)  &\cong \h^{i+2}(G,M) \textnormal{ for all $i \ge 1$.}
\end{align*}
\end{enumerate}
\end{exercise}

\begin{exercise}\label{excochains}
Let $R$ be a commutative ring.
\begin{enumerate}[(a)]
\item Let $\phi: H \to G$ be a group homomorphism and $A$ an $R[G]$-module.
Prove that the restriction maps $\res^n: H^n(G,A) \to H^n(H,A)$
are given in terms of cochains of the bar resolution by composing
the cochains by~$\phi$. 
\item Let $H$ be a normal subgroup of~$G$.
Describe the inflation maps in terms of cochains of the bar resolution.
\item Let $H$ be a normal subgroup of~$G$ and $A$ an $R[G]$-module.
Describe the $R[G/H]$-action on $H^n(H,A)$ in terms of cochains of the bar resolution.
\item Let now $H \le G$ be a subgroup of finite index.
Describe the corestriction maps in terms of cochains of the bar resolution.
\end{enumerate}
\end{exercise}

\begin{exercise}\label{exshapiro}
Prove Shapiro's lemma, i.e.\ Proposition~\ref{shapiro}.

Hint: see \cite[(6.3.2)]{Weibel} for an abstract proof; see also \cite[III.6.2]{Brown} for the underlying map.
\end{exercise}

\section{Cohomology of $\PSL_2(\ZZ)$}

In this section, we shall calculate the cohomology of the group $\PSL_2(\ZZ)$ and important properties thereof.
This will be at the basis of our treatment of Manin symbols in the following section.
The key in this is the description of $\PSL_2(\ZZ)$ as a free product of two cyclic groups.

\subsection{Theory: The standard fundamental domain for $\PSL_2(\ZZ)$}

We define the matrices of $\SL_2(\ZZ)$
$$\sigma := \mat 0{-1}10, \;\;\; \tau := \mat {-1}1{-1}0,\;\;\;
T = \mat 1101 = \tau \sigma.$$
By the definition of the action of $\SL_2(\ZZ)$ on $\HH$ in equation~\ref{eq:moebius}, we have for all $z \in \HH$:
$$\sigma. z = \frac{-1}{z}, \;\;\; \tau. z := 1-\frac{1}{z} ,\;\;\;
T.z = z+1.$$
These matrices have the following conceptual meaning:
$$ \langle \pm \sigma \rangle = \Stab_{\SL_2(\ZZ)} (i), \; 
\langle \pm \tau \rangle = \Stab_{\SL_2(\ZZ)} (\zeta_6)\; \text{ and } \;
\langle \pm T \rangle = \Stab_{\SL_2(\ZZ)} (\infty)$$
with $\zeta_6 = e^{2\pi i/6}$.
From now on we will often represent
classes of matrices in $\PSL_2(\ZZ)$ by matrices in $\SL_2(\ZZ)$.
The orders of $\sigma$ and $\tau$ in $\PSL_2(\ZZ)$ are $2$ and~$3$,
respectively. These statements are checked by calculation.
Exercise~\ref{exfiniteorder} is recommended at this point.

Even though in this section our interest concerns the full group $\SL_2(\ZZ)$,
we give the definition of fundamental domain for general subgroups of $\SL_2(\ZZ)$ of finite index.

\begin{definition}\label{defi:fd}
Let $\Gamma \le \SL_2(\ZZ)$ be a subgroup of finite index.
A {\em fundamental domain} for the action of~$\Gamma$ on~$\HH$
is a subset $\cF \subset \HH$ such that the following hold:
\begin{enumerate}[(i)]
\item $\cF$ is open.
\item For every $z \in \HH$, there is $\gamma \in \Gamma$ such that $\gamma.z \in \overline{\cF}$.
\item If $\gamma.z \in \cF$ for $z \in \cF$ and $\gamma \in \Gamma$, then one has $\gamma = \pm \mat 1001$.
\end{enumerate}
\end{definition}

In other words, a fundamental domain is an open set, which is small enough not to contain any two points that
are equivalent under the operation by~$\Gamma$, and which is big enough that every point in the upper half plane
is equivalent to some point in the closure of the fundamental domain.

\begin{proposition}\label{prop:fd}
The set
$$ \cF := \{ z \in \HH \;| \; |z| > 1 \textnormal{ and } -\frac{1}{2} < \Real(z) < \frac{1}{2} \}$$
is a fundamental domain for the action of $\SL_2(\ZZ)$ on~$\HH$.
\end{proposition}

It is clear that $\cF$ is open. For~(ii), we use the following lemma.

\begin{lemma}\label{lem:fd}
Let $z \in \HH$. The orbit $\SL_2(\ZZ).z$ contains a point $\gamma.z$
with maximal imaginary part (geometrically also called `height'), {\em i.e.}
$$ \Imag (\gamma.z) \ge \Imag(g.z)  \;\; \forall g \in \SL_2(\ZZ).$$
A point $z \in \HH$ is of maximal height
if $|c z + d| \ge 1$ for all coprime $c,d \in \ZZ$.
\end{lemma}

\begin{proof}
We have the simple formula $\Imag (\gamma.z) = \frac{\Imag(z)}{|cz+d|^2}$.
It implies
$$ \Imag (z) \le \Imag(\gamma.z) \Leftrightarrow |cz + d| \le 1.$$
For fixed $z = x+iy$ with $x,y \in \RR$, consider the inequality
$$1 \ge |cz + d|^2 = (cx+d)^2 + c^2y^2.$$
This expression admits only finitely many solutions $c,d \in \ZZ$.
Among these finitely many, we may choose a coprime pair $(c,d)$ with minimal $|cz + d|$.
Due to the coprimeness, there are $a,b \in \ZZ$ such that the matrix $M := \mat abcd$ belongs to $\SL_2(\ZZ)$.
It is now clear that $M.z$ has maximal height.
\end{proof}

We next use a simple trick to show~(ii) in Definition~\ref{defi:fd} for $\calF$.
Let $z \in \HH$. By Lemma~\ref{lem:fd}, we choose $\gamma \in \SL_2(\ZZ)$ such that $\gamma.z$ has maximal height.
We now `transport' $\gamma.z$ via an appropriate translation $T^n$ in such a way that
$-1/2 \le \Real(T^n \gamma . z) < 1/2$. The height is obviously left invariant.
Now we have $|T^n \gamma . z| \ge 1$ because otherwise the height of
$T^n \gamma . z$ would not be maximal. For, if $|T^n \gamma . z + 0| < 1$ then applying $\sigma$ (corresponding to reflection on the unit circle)
would make the height strictly bigger.
More precisely, we have the following result.

\begin{lemma}\label{lem:fd2}
Every point of maximal height in $\HH$ can be translated into the closure of the fundamental domain~$\overline{\cF}$.
Conversely, $\overline{\cF}$ only contains points of maximal height.
\end{lemma}

\begin{proof}
The first part was proved in the preceding discussion. The second one follows from the calculation
\begin{multline}\label{eq:fd2}
|cz+d|^2 = (cx+d)^2+c^2y^2 = c^2 |z|^2 + 2cdx + d^2 \\
\ge c^2|z|^2  - |cd| + d^2 \ge c^2  - |cd| + d^2 \ge (|c|-|d|)^2 + |cd| \ge 1
\end{multline}
for all coprime integers $c,d$ and $z=x+iy \in \HH$ with $x,y \in \RR$.
\end{proof}

\begin{proof}[End of the proof of Proposition~\ref{prop:fd}.]
Let $z \in \cF$ and $\gamma := \mat abcd  \in \SL_2(\ZZ)$ such that $\gamma.z \in \calF$.
By Lemma~\ref{lem:fd2}, $z$ and $\gamma.z$ both have maximal height, whence $|cz+d| = 1$.
Hence the inequalities in equation~\ref{eq:fd2} are equalities, implying $c=0$.
Thus, $\gamma = \pm T^n$ for some $n \in \ZZ$. But only $n=0$ is compatible with the
assumption $\gamma.z \in \calF$. This proves (iii) in Definition~\ref{defi:fd} for $\calF$.
\end{proof}

\begin{proposition}\label{prop:gensl2z}
The group $\SL_2(\ZZ)$ is generated by the matrices $\sigma$ and~$\tau$.
\end{proposition}

\begin{proof}
Let $\Gamma := \langle \sigma,\tau \rangle$ be the subgroup of $\SL_2(\ZZ)$ generated by $\sigma$ and~$T$.

We prove that for any $z \in \HH$ there is $\gamma \in \Gamma$ such that $\gamma.z \in \overline{\calF}$.
For that, note that the orbit $\Gamma.z$ contains a point $\gamma.z$ for $\gamma \in \Gamma$ of maximal height as it is a subset of $\SL_2(\ZZ).z$,
for which we have seen that statement.
As $\Gamma$ contains $T=\tau \sigma$, we can translate $\gamma.z$ so as to have real part in between $-\frac{1}{2}$ and $\frac{1}{2}$.
As $\Gamma$ also contains $\sigma$, the absolute value of the new point has to be at least~$1$ because other $\sigma$ would make the height bigger.

In order to conclude, choose any point $z \in \cF$ and let $M \in \SL_2(\ZZ)$.
We consider the point $M.z$ and `transport' it back into $\calF$ via a matrix $\gamma \in \Gamma$.
We thus have $(\gamma M).z \in \calF$. As $\calF$ is a fundamental domain for $\SL_2(\ZZ)$, it follows $\gamma M = \pm 1$,
showing $M \in \Gamma$.
\end{proof}

An alternative algorithmic proof is provided in Algorithm~\ref{algpsl} below.

\subsection{Theory: $\PSL_2(\ZZ)$ as a free product}

We now apply the knowledge about the (existence of the) fundamental domain for $\PSL_2(\ZZ)$
to derive that $\PSL_2(\ZZ)$ is a free product.

\begin{definition}
Let $G$ and $H$ be two groups. The {\em free product $G * H$}
of $G$ and $H$ is the group having as elements
all the possible {\em words}, i.e.\ sequences of symbols,
$a_1 a_2 \dots a_n$ with $a_i \in G-\{1\}$ or $a_i \in H-\{1\}$ such that
elements from $G$ and $H$ alternate (i.e.\ if $a_i \in G$,
then $a_{i+1} \in H$ and vice versa) together with the empty word, which
we denote by~$1$.
The integer~$n$ is called the {\em length} of the group element (word)
$w=a_1 a_2 \dots a_n$ and denoted by~$l(w)$. We put $l(1) = 0$ for the
empty word.

The group operation in $G*H$ is concatenation of words followed by `reduction' (in order to obtain
a new word obeying to the rules). The reduction ruls are: for all words $v,w$, all $g_1,g_2 \in G$ and all $h_1,h_2 \in H$:
\begin{itemize}
\item $v1w=vw$,
\item $vg_1g_2w=v(g_1g_2)w$ (i.e.\ the multiplication of $g_1$ and $g_2$ in $G$ is carried out),
\item $vh_1h_2w=v(h_1h_2)w$ (i.e.\ the multiplication of $h_1$ and $h_2$ in $H$ is carried out).
\end{itemize}
\end{definition}

In Exercise~\ref{exfreegroup} you are asked to verify that $G*H$
is indeed a group and to prove a universal property.
Alternatively, if $G$ is given by the set of generators $\calG_G$ together with relations $\calR_G$ and similarly for the group~$H$,
then the free product $G * H$ can be described as the group generated by $\calG_G \cup \calG_H$ with relations $\calR_G \cup \calR_H$.

\begin{theorem}\label{thm:freepr}
Let $\calP$  be the free product $\langle \sigma \rangle * \langle \tau \rangle$ of
the cyclic groups $\langle \sigma \rangle$ of order~$2$ and $\langle \tau \rangle$ of order~$3$.

Then $\calP$ is isomorphic to~$\PSL_2(\ZZ)$.
In particular, as an abstract group, $\PSL_2(\ZZ)$ can be represented by
generators and relations as $\langle \sigma, \tau \, | \, \sigma^2=\tau^3=1 \rangle$.
\end{theorem}

In the proof, we will need the following statement, which we separate because it is entirely computational.

\begin{lemma}\label{lem:freepr}
Let $\gamma \in \calP$ be $1$ or any word starting in~$\sigma$ on the left, {\em i.e.} $\sigma \tau^{e_1} \sigma \tau^{e_2} \dots$.
Then $\Imag(\tau^2 \gamma.i) < 1$.
\end{lemma}

\begin{proof}
For $\gamma=1$, the statement is clear. Suppose $\gamma = \sigma \tau^{e_1} \sigma \tau^{e_2} \sigma \dots \tau^{e_{r-1}} \sigma \tau^{e_r}$
with $r \ge 0$, $e_i \in \{1,2\}$ for $i=1,\dots,r$. We prove more generally
$$ \Imag(\tau^2\gamma.i)=\Imag(\tau^2(\gamma\sigma).i) > \Imag(\tau^2 (\gamma \sigma) \tau^e.i) = \Imag(\tau^2 (\gamma \sigma) \tau^e \sigma.i)$$
for any $e=1,2$. This means that extending the word to the right by $\sigma \tau^e$, the imaginary part goes strictly down for both $e=1,2$.

We first do some matrix calculations.
Let us say that an integer matrix $\mat abcd$ satisfies (*) if $(c+d)^2> \max(c^2,d^2)$.
The matrix $\tau^2\sigma = \mat {-1}0{-1}{-1}$ clearly satisfies~(*).
Let us assume that $\gamma =\mat abcd$ satisfies~(*). We show that
$\gamma \tau \sigma = \mat ** c {c+d}$ and $\gamma \tau^2 \sigma = \mat ** {-c-d} {-d}$ also satisfy~(*).
The first one follows once we know $(2c+d)^2 > \max(c^2,(c+d)^2)$. This can be seen like this:
$$ (2c+d)^2 = (c^2+2cd) + 2c^2  + (c+d)^2 > 2c^2+(c+d)^2 \ge \max(c^2,(c+d)^2),$$
where we used that (*) implies $(c+d)^2>d^2$ and, thus, $c^2+2cd>0$. The second inequality is obtained by
exchanging the roles of $c$ and~$d$.

We thus see that $\tau^2\gamma = \mat abcd$ satisfies~(*) for all words $\gamma$ starting and ending in~$\sigma$.
Finally, we have for all such $\gamma$:
\begin{align*}
\Imag(\tau^2\gamma i) &=  \frac{1}{|ci+d|^2} &= \frac{1}{c^2+d^2},\\
\Imag(\tau^2 \gamma  \tau i) &=  \Imag(\tau^2\gamma (i+1))
= \frac{1}{|c(i+1)+d|^2} &= \frac{1}{(c+d)^2+c^2},\\
\Imag(\tau^2 \gamma  \tau^2 i) &=  \Imag(\tau^2\gamma \frac{1+i}{2})
= \frac{1/2}{|c(i/2+1/2)+d|^2} &= \frac{2}{(c+2d)^2+c^2}.
\end{align*}
Now (*) implies the desired inequalities of the imaginary parts.
\end{proof}

\begin{proof}[Proof of Theorem~\ref{thm:freepr}]
As $\SL_2(\ZZ)$ is generated by $\sigma$ and~$\tau$ due to Proposition~\ref{prop:gensl2z},
the universal property of the free product
gives us a surjection of groups $\calP \twoheadrightarrow \PSL_2(\ZZ)$.

Let $B$ be the geodesic path from $\zeta_6$ to~$i$, i.e.\ the arc 
between $\zeta_6$ and~$i$ in positive orientation (counter clockwise)
on the circle of radius~$1$ around the origin, lying entirely on the closure $\overline{\calF}$ of the
standard fundamental domain from Proposition~\ref{prop:fd}.
Define the map
$$ \PSL_2(\ZZ) \xrightarrow{\phi} \{\textnormal{Paths in }\HH\}$$
which sends $\gamma \in \PSL_2(\ZZ)$ to $\gamma.B$, i.e.\ the image of~$B$
under~$\gamma$.
The proof of the theorem is finished by showing that the composite
$$ \calP \twoheadrightarrow \PSL_2(\ZZ) \xrightarrow{\phi} \{\textnormal{Paths in }\HH\}$$
is injective, as then the first map must be an isomorphism.

This composition is injective because its image is a tree, that is, a graph without circles.
By drawing it, one convinces oneself very quickly hereof.
We, however, give a formal argument, which can also be nicely visualised on the geometric realisation of the graph
as going down further and further in every step.

In order to prepare for the proof, let us first suppose that $\gamma_1.B$ and $\gamma_2.B$ for some $\gamma_1,\gamma_2 \in \PSL_2(\ZZ)$
meet in a point which is not the endpoint of either of the two paths. Then $\gamma.B$ intersects $B$ in some interior point for
$\gamma := \gamma_1^{-1}\gamma_2$. This intersection point lies on the boundary of the fundamental dommain $\calF$.
Consequently, by (iii) in Definition~\ref{defi:fd}, $\gamma = \pm 1$ and $\gamma_1.B = \gamma_2.B$.
This implies that if $\Imag(\gamma_1.i) \neq \Imag(\gamma_2.i)$ where $i=\sqrt{-1}$,
then $\gamma_1.B$ and $\gamma_2.B$ do not meet in any interior point and are thus distinct paths.

It is obvious that $B, \sigma.B, \tau.B$ are distinct paths. They share the property that their point
that is conjugate to~$i$ has imaginary part~$1$ (in fact, the points cojugate to~$i$ in the paths are $i$, $i$, $i+1$,
respectively).

By Lemma~\ref{lem:freepr}, for $\gamma$ equal to~$1$ or any word in $\calP$ starting with~$\sigma$ on the left,
we obtain that $\tau^2\gamma.B$ is distinct from $B, \sigma.B, \tau.B$ because it lies `lower'.
In particular, $\tau^2\gamma.B \neq B$.
As $\tau^2\gamma.B \neq \tau.B$, we also find $\tau \gamma.B \neq B$.
Finally, if $\gamma.B = B$ and $\gamma = \sigma\tau^e \gamma'$ with $e\in \{1,2\}$ and $\gamma'$ starting in $\sigma$ or $\gamma'=1$,
then $\tau^e\gamma'.B = \sigma.B$, which has already been excluded.
We have thus found that for any non-trivial word $\gamma\in \calP$, the conjugate $\gamma.B$ is distinct from~$B$.
This proves the desired injectivity.
\end{proof}

\subsection{Theory: Mayer-Vietoris for $\PSL_2(\ZZ)$}

Motivated by the description $\PSL_2(\ZZ) = C_2 * C_3$, we
now consider the cohomology of a group~$G$
which is the free product of two finite groups $G_1$ and $G_2$, i.e.\ $G = G_1 * G_2$.

\begin{proposition}\label{mvbieri}
Let $R$ be a commutative ring. The sequence
$$0 \to R[G] \xrightarrow{\alpha} R[G/G_1] \oplus R[G/G_2] \xrightarrow{\epsilon} R \to 0$$
with $\alpha(g) = (gG_1,-gG_2)$ and $\epsilon(gG_1,0)=1=\epsilon(0,gG_2)$
is exact.
\end{proposition}

\begin{proof}
This proof is an even more elementary version of an elementary proof
that I found in \cite{Bieri}.
Clearly, $\epsilon$ is surjective and also $\epsilon \circ \alpha = 0$.

Next we compute exactness at the centre.
We first claim that for every element $g \in G$ we have
$$ g -1 = \sum_j \alpha_j g_j (h_j-1) \in R[G/G_1]$$
for certain $\alpha_j \in R$ and certain $g_j \in G$, $h_j \in G_2$
and analogously with the roles of $G_1$ and $G_2$ exchanged.
To see this, we write $g= a_1 a_2 \dots a_n$ with $a_i$ alternatingly in $G_1$ and $G_2$ (we do not need the
uniqueness of this expression). If $n=1$, there is nothing to do.
If $n > 1$, we have 
$$a_1 a_2 \dots a_n - 1 = a_1 a_2 \dots a_{n-1} (a_n - 1) + (a_1 a_2 \dots a_{n-1} - 1)$$
and we obtain the claim by induction.
Consequently, we have for all $\lambda = \sum_i r_i g_i G_1$ and
all $\mu = \sum_k \tilde{r}_k \tilde{g}_k G_2$
with $r_i, \tilde{r}_k \in R$ and $g_i, \tilde{g}_k \in G$
$$\lambda - \sum_i r_i 1_G G_1= \sum_j \alpha_j g_j (h_j-1) \in R[G/G_1]$$
and
$$\mu - \sum_k \tilde{r}_k 1_G G_2
= \sum_l \tilde{\alpha}_l \tilde{g}_l (\tilde{h}_l-1) \in R[G/G_2]$$
for certain $\alpha_j, \tilde{\alpha}_l \in R$,
certain $g_j, \tilde{g}_l \in G$ and certain $h_j \in G_2$, $\tilde{h}_l \in G_1$.
Suppose now that with $\lambda$ and $\mu$ as above we have
$$\epsilon(\lambda, \mu)= \sum_i r_i +  \sum_k \tilde{r}_k= 0.$$
Then we directly get
$$ \alpha(\sum_j \alpha_j g_j (h_j-1)
        - \sum_l \tilde{\alpha}_l \tilde{g}_l (\tilde{h}_l-1)
        + \sum_i r_i 1_G\big)= (\lambda, \mu)$$
and hence the exactness at the centre.

It remains to prove that $\alpha$ is injective.
Now we use the freeness of the product.
Let $\lambda = \sum_w a_w w \in R[G]$ be an element in the kernel of~$\alpha$.
Hence, $\sum_w a_w wG_1 = 0$ and $\sum_w a_w wG_2=0$.
Let us assume that $\lambda \neq 0$. It is clear that $\lambda$
cannot just be a multiple of~$1 \in G$, as otherwise it would not
be in the kernel of~$\alpha$.
Now pick the $g \in G$ with $a_g \neq 0$ having maximal length $l(g)$
(among all the $l(w)$ with $a_w \neq 0$).
It follows that $l(g)> 0$.
Assume without loss of generality that the 
representation of~$g$ ends in a non-zero element of~$G_1$. 
Further, since $a_g \neq 0$ and $0 = \sum_w a_w wG_2$, there
must be an $h \in G$ with $g \neq h$, 
$g G_2 = h G_2$ and $a_h \neq 0$. As $g$ does not end in~$G_2$,
we must have $h = gy$ for some $0 \neq y \in G_2$.
Thus, $l(h) > l(g)$, contradicting the maximality and proving
the proposition.
\end{proof}

Recall that we usually denote the restriction of a module to a subgroup by the same symbol.
For example, in the next proposition we will write $\h^1(G_1, M)$ instead of $\h^1(G_1,\Res^G_{G_1}(M))$.

\begin{proposition}[Mayer-Vietoris]\label{mv}
Let $G = G_1 * G_2$ be a free product.
Let $M$ be a left $R[G]$-module. Then the Mayer-Vietoris
sequence gives the exact sequences
$$0 \to M^{G} \to M^{G_1} \oplus M^{G_2} \to M \to 
\h^1(G,M) \xrightarrow{\res} \h^1(G_1, M) \oplus \h^1(G_2, M) \to 0.$$
and for all $i \ge 2$ an isomorphism
$$\h^i(G,M) \cong \h^i(G_1, M) \oplus \h^i(G_2, M).$$
\end{proposition}

\begin{proof} 
We see that all terms in the exact sequence of
Proposition~\ref{mvbieri} are free $R$-modules. We now apply the functor
$\Hom_R(\cdot, M)$ to this exact sequence and obtain the exact
sequence of $R[G]$-modules
$$ 0 \to M \to \underbrace{\Hom_{R[G_1]}(R[G],M)}_{\cong \Coind_{G_1}^G(M)} \oplus \underbrace{\Hom_{R[G_2]}(R[G],M)}_{\cong \Coind_{G_2}^G(M)}\to
\underbrace{\Hom_R (R[G],M)}_{\cong \Coind_{1}^G(M)} \to 0.$$
The central terms, as well as the term on the
right, can be identified with coinduced modules.  Hence, the
statements on cohomology follow by taking the long exact sequence of
cohomology and invoking Shapiro's Lemma~\ref{shapiro}.
\end{proof}

We now apply the Mayer-Vietoris sequence (Prop.~\ref{mv}) to $\PSL_2(\ZZ)$
and get that for any ring~$R$ and any left
$R[\PSL_2(\ZZ)]$-module~$M$ the sequence
\begin{multline}\label{mayervietoris}
0 \to M^{\PSL_2(\ZZ)} \to M^{\langle \sigma \rangle} \oplus M^{\langle \tau \rangle} 
\to M \\ \xrightarrow{m \mapsto f_m}
\h^1(\PSL_2(\ZZ),M) 
\xrightarrow{\res} \h^1(\langle \sigma \rangle, M) \oplus \h^1(\langle \tau \rangle, M) \to 0
\end{multline}
is exact and for all $i \ge 2$ one has isomorphisms
\begin{equation}\label{mveins}
\h^i(\PSL_2(\ZZ),M) \cong \h^i(\langle \sigma \rangle, M) \oplus \h^i(\langle \tau \rangle, M).
\end{equation}
The $1$-cocycle $f_m$ can be explicitly described as the cocycle
given by $f_m(\sigma) = (1-\sigma)m$ and $f_m(\tau) = 0$ (see
Exercise~\ref{exmv}).

\begin{lemma}\label{mackeypsl}
Let $\Gamma \le \PSL_2(\ZZ)$ be a subgroup of finite index
and let $x \in \HH \cup \PP^1(\QQ)$ be any point.
Recall that $\PSL_2(\ZZ)_x$ denotes the stabiliser of~$x$ for
the $\PSL_2(\ZZ)$-action. 
\begin{enumerate}[(a)]
\item The map
$$ \Gamma \backslash \PSL_2(\ZZ) / \PSL_2(\ZZ)_x  
  \xrightarrow{g \mapsto gx} \Gamma \backslash \PSL_2(\ZZ)x$$
is a bijection.
\item For $g \in \PSL_2(\ZZ)$ the stabiliser of~$gx$
for the $\Gamma$-action is
$$ \Gamma_{gx} = \Gamma \cap g \PSL_2(\ZZ)_x g^{-1}.$$
\item For all $i \in \NN$, and all $R[\Gamma]$-modules, Mackey's formula (Prop.~\ref{propmackey})
gives an isomorphism
$$\h^i(\PSL_2(\ZZ)_x, \Coind_\Gamma^{\PSL_2(\ZZ)} V) \cong 
\prod_{y \in \Gamma \backslash \PSL_2(\ZZ)x} \h^i(\Gamma_y, V).$$
\end{enumerate}
\end{lemma}

\begin{proof}
(a) and (b) are clear and (c) follows directly from Mackey's formula.
\end{proof}

\begin{corollary}\label{corhzweinull}
Let $R$ be a ring and $\Gamma \le \PSL_2(\ZZ)$ be a subgroup of
finite index such that all the orders of all stabiliser groups
$\Gamma_x$ for $x \in \HH$ are invertible in~$R$.  Then for all
$R[\Gamma]$-modules~$V$ one has
$$\h^1(\Gamma,V) \cong M/(M^{\langle \sigma \rangle} + M^{\langle \tau \rangle})$$
with $M =  \Coind_\Gamma^{\PSL_2(\ZZ)}(V)$ and
$$\h^i(\Gamma,V) = 0$$
for all  $i \ge 2$.
\end{corollary}

\begin{proof}
By Lemma~\ref{mackeypsl}~(b), all non-trivial stabiliser groups for
the action of $\Gamma$ on $\HH$ are of the form $g \langle \sigma
\rangle g^{-1} \cap \Gamma$ or $g \langle \tau \rangle g^{-1} \cap
\Gamma$ for some $g \in \PSL_2(\ZZ)$.
Due to the invertibility assumption we get from Prop.~\ref{corres}
that the groups on the right in the equation in Lemma~\ref{mackeypsl}~(c) are zero.
Hence, by Shapiro's lemma (Prop.~\ref{shapiro}) we have
$$ \h^i(\Gamma,V) \cong \h^i(\PSL_2(\ZZ),M)$$
for all $i \ge 0$, so that by Equations (\ref{mayervietoris}) and~(\ref{mveins})
we obtain the proposition.
\end{proof} 

By Exercise~\ref{exfiniteorder}, the assumptions of the proposition are
for instance always satisfied if $R$ is a field of characteristic not
$2$ or $3$. Look at Exercise~\ref{exfotwo} to see for which~$N$ 
the assumptions hold for $\Gamma_1(N)$ and $\Gamma_0(N)$ over an
arbitrary ring (e.g.\ the integers).

\subsection{Theory: Parabolic group cohomology}

Before going on, we include a description of the cusps as $\PSL_2(\ZZ)$-orbits that is very useful for the sequel.

\begin{lemma}
The cusps $\PP^1(\QQ)$ lie in a single $\PSL_2(\ZZ)$-orbit.
The stabiliser group of $\infty$ for the $\PSL_2(\ZZ)$-action is $\langle T \rangle$ and the map
$$ \PSL_2(\ZZ)/\langle T \rangle \xrightarrow{g \langle T \rangle \mapsto g \infty}\PP^1(\QQ)$$
is a $\PSL_2(\ZZ)$-equivariant bijection.
\end{lemma}

\begin{proof}
The claim on the stabiliser follows from a simple direct computation. This makes the map well-defined and injective.
The surjectivity is equivalent to the claim that the cusps lie in a single $\PSL_2(\ZZ)$-orbit and
simply follows from the fact that any pair of coprime integers $(a,c)$ appears as the first column of a matrix in $\SL_2(\ZZ)$.
\end{proof}

Let $R$ be a ring, $\Gamma \le \PSL_2(\ZZ)$ a subgroup of finite index.
One defines the {\em  parabolic cohomology group for the left $R[\Gamma]$-module~$V$}
as the kernel of the restriction map in
\begin{equation}\label{pardef}
0 \to \Hpar^1(\Gamma, V) \to \h^1(\Gamma,V) \xrightarrow{\res} 
\prod_{g \in \Gamma \backslash \PSL_2(\ZZ) / \langle T \rangle}
\h^1(\Gamma \cap \langle g T g^{-1} \rangle, V).
\end{equation}

\begin{proposition}\label{leraypar}
Let $R$ be a ring and $\Gamma \le \PSL_2(\ZZ)$ be a subgroup of finite index such
that all the orders of all stabiliser groups $\Gamma_x$ for $x \in \HH$ are invertible in~$R$.
Let $V$ be a left $R[\Gamma]$-module. Write for short $G = \PSL_2(\ZZ)$
and $M = \Hom_{R[\Gamma]}(R[G],V)$.
Then the following diagram is commutative, its vertical maps are isomorphisms
and its rows are exact:
\begin{small}
$$ \xymatrix@R=.8cm@C=.4cm{
0  \ar@{->}[r]
& \Hpar^1(\Gamma, V)  \ar@{->}[r]
& \h^1(\Gamma,V)  \ar@{->}[r]^(.3){\res} 
& \underset{g \in \Gamma \backslash \PSL_2(\ZZ) / \langle T \rangle}{\prod}
\h^1(\Gamma \cap \langle g T g^{-1} \rangle, V) \ar@{->}[r]
& V_\Gamma  \ar@{->}[r]
& 0\\
0  \ar@{->}[r]
& \Hpar^1(G, M)  \ar@{->}[r]\ar@{->}[u]^(.5){\textnormal{Shapiro}}
& \h^1(G,M)  \ar@{->}[r]^(.3){\res} \ar@{->}[u]^(.5){\textnormal{Shapiro}}
& \h^1(\langle T \rangle, M) \ar@{->}[r]\ar@{->}[u]^(.5){\textnormal{Mackey}}
& V_\Gamma  \ar@{->}[r]\ar@{=}[u]
& 0\\
0  \ar@{->}[r]
& \Hpar^1(G, M)  \ar@{->}[r]\ar@{=}[u]
& M/(M^{\langle \sigma \rangle} + M^{\langle \tau \rangle}) \ar@{->}[r]^(.5){m \mapsto (1-\sigma)m} \ar@{->}[u]^{m \mapsto f_m}
& M/(1-T)M \ar@{->}[r]\ar@{<-}[u]^{c \mapsto c(T)}
& M_G  \ar@{->}[r]\ar@{->}[u]^{\phi}
& 0
}$$%
\end{small}
The map $\phi: M_G \to V_\Gamma$ is given as $f \mapsto \sum_{g \in \Gamma \backslash G} f(g)$.
\end{proposition}

\begin{proof}
The commutativity of the diagram is checked in Exercise~\ref{exparcompat}.
By Exercise~\ref{exfree} we have $\h^1(\langle T \rangle, M) \cong M/(1-T)M$.
Due to the assumptions we may apply Corollary~\ref{corhzweinull}.
The cokernel of
$M/(M^{\langle \sigma \rangle} + M^{\langle \tau \rangle}) 
\xrightarrow{m \mapsto (1-\sigma)m} M/(1-T)M$
is immediately seen to be~$M/((1-\sigma)M + (1-T)M)$, which is
equal to~$M_G$, as $T$ and $\sigma$ generate~$\PSL_2(\ZZ)$.
Hence, the lower row is an exact sequence.

We now check that the map $\phi$ is well-defined. For this we verify
that the image of $f(g)$ in $V_\Gamma$ only depends on the coset
$\Gamma \backslash G$:
$$ f(g) - f(\gamma g) = f(g) - \gamma f(g) = (1-\gamma) f(g) = 0 \in V_\Gamma.$$
Hence, for any $h \in G$ we get 
$$\phi((1-h).f) = \sum_{g \in \Gamma \backslash \PSL_2(\ZZ)} (f(g) - f(gh)) = 0,$$
as $gh$ runs over all cosets. Thus, $\phi$ is well-defined.
To show that $\phi$ is an isomorphism, we give an inverse $\psi$ to~$\phi$ by
$$ \psi: V_\Gamma \to \Hom_{R[\Gamma]}(R[G],V)_G, \;\;\; v \mapsto
e_v \textnormal{ with } e_v(g) = \begin{cases}
g v, & \textnormal{ for } g \in \Gamma\\
0, & \textnormal{ for } g \not\in\Gamma.
\end{cases}$$
It is clear that $\phi \circ \psi$ is the identity. The map~$\phi$ is
an isomorphism, as $\psi$ is surjective. In order to see this, fix a system of representatives
$\{1=g_1,g_2,\dots,g_n\}$ for $\Gamma \backslash \PSL_2(\ZZ)$. We first have
$ f = \sum_{i=1}^n g_i^{-1}.e_{f(g_i)} $
because for all $h \in G$ we find
$$f(h) = g_j^{-1}.e_{f(g_j)}(h) =  e_{f(g_j)}(hg_j^{-1}) = hg_j^{-1}.f(g_j)=.f(hg_j^{-1}g_j)=f(h),$$
where $1 \le j \le n$ is the unique index such that $h \in \Gamma g_j$.
Thus
$$f= \sum_{i=1}^n e_{f(g_i)} - \sum_{i=2}^n (1-g_i^{-1}).e_{f(g_i)}\in \Image(\psi),$$
as needed.

More conceptually, one can first identify
the coinduced module $\Coind_\Gamma^{\PSL_2(\ZZ)}(V)$ with the
induced one $\Ind_\Gamma^{\PSL_2(\ZZ)}(V) = R[G] \otimes_{R[\Gamma]} V$.
We claim that the $G$-coinvariants are isomorphic
to $R \otimes_{R[\Gamma]} V \cong V_\Gamma$.
As $R$-modules we have $R[G] = I_G \oplus R1_G$ since
$r \mapsto r 1_G$ defines a splitting of the augmentation map.
Here $I_G$ is the augmentation ideal defined in Exercise~\ref{exgp}.
Consequently, 
$R[G] \otimes_{R[\Gamma]} V \cong (I_G \otimes_{R[\Gamma]} V) \oplus R \otimes_{R[\Gamma]} V$.
The claim follows, since
$I_G (R[G] \otimes_{R[\Gamma]} V) \cong I_G \otimes_{R[\Gamma]} V$.

Since all the terms in the upper and the middle row are isomorphic
to the respective terms in the lower row, all rows are exact.
\end{proof}

\subsection{Theory: Dimension computations}

This seems to be a good place to compute the dimension of $\h^1(\Gamma,V_{k-2}(K))$
and $\Hpar^1(\Gamma,V_{k-2}(K))$ over a field~$K$ under certain conditions.
The results will be important for
the proof of the Eichler-Shimura theorem.

\begin{lemma}\label{lemvknull}
Let $R$ be a ring and let $n \ge 1$ be an integer, $t = \mat 1 N 0 1$ and
$t' = \mat 1 0 N 1$.
\begin{enumerate}[(a)]
\item If $n! N$ is not a zero divisor in~$R$, then
for the $t$-invariants we have
$$V_n(R)^{\langle t \rangle} = \langle X^n \rangle$$
and for the $t'$-invariants
$$V_n(R)^{\langle t' \rangle} = \langle Y^n \rangle.$$
\item If $n! N$ is invertible in~$R$, then the coinvariants are
given by
$$V_n(R)_{\langle t \rangle} = V_n(R)/\langle Y^n, XY^{n-1}, \dots, X^{n-1}Y \rangle$$
respectively
$$V_n(R)_{\langle t' \rangle} = V_n(R)/\langle X^n, X^{n-1}Y, \dots, XY^{n-1} \rangle.$$
\item If $n! N$ is not a zero divisor in~$R$, 
then the $R$-module of $\Gamma(N)$-invariants $V_n(R)^{\Gamma(N)}$ is zero. 
In particular, if $R$ is a field of characteristic~$0$ and $\Gamma$ is
any congruence subgroup, then $V_n(R)^\Gamma$ is zero.
\item If $n! N$ is invertible in~$R$,
then the $R$-module of $\Gamma(N)$-coinvariants $V_n(R)_{\Gamma(N)}$ is zero. 
In particular, if $R$ is a field of characteristic~$0$ and $\Gamma$ is
any congruence subgroup, then $V_n(R)_\Gamma$ is zero.
\end{enumerate}
\end{lemma}

\begin{proof}
(a) The action of~$t$ is $t.(X^{n-i} Y^i) = X^{n-i} (NX + Y)^i$ and consequently
$$(t-1). (X^{n-i} Y^i) 
= (\sum_{j=0}^i \vect ij N^{i-j} X^{i-j}Y^j)X^{n-i} - X^{n-i}Y^i
= \sum_{j=0}^{i-1} r_{i,j} X^{n-j}Y^j$$
with $r_{i,j} =N^{i-j} \vect{i}{j}$, which is not a zero divisor, 
respectively invertible, by assumption.
For $x = \sum_{i=0}^n a_i X^{n-i}Y^i$ we have
\begin{multline*} (t-1).x 
= \sum_{i=0}^n a_i \sum_{j=0}^{i-1}r_{i,j} X^{n-j}Y^j
= \sum_{j=0}^{n-1} X^{n-j}Y^j (\sum_{i=j+1}^n a_i r_{i,j})\\
= XY^{n-1} a_nr_{n,n-1} + X^2Y^{n-2} (a_n r_{n,n-2} + a_{n-1}r_{n-1,n-2}) + \dots.
\end{multline*}
If $(t-1).x = 0$, we conclude for $j = n-1$ that $a_n = 0$. 
Next, for $j = n-2$ it follows that $a_{n-1} = 0$,
and so on, until $a_1 = 0$. This proves the statement on
the $t$-invariants. The one on the $t'$-invariants follows from symmetry. 

(b) The claims on the coinvariants are proved in a very similar and
straightforward way. 

(c) and (d) As $\Gamma(N)$ contains the matrices $t$ and $t'$, this follows
from Parts~(a) and~(b).
\end{proof}

\begin{proposition}\label{dimheins}
Let $K$ be a field of characteristic~$0$ and $\Gamma \le \SL_2(\ZZ)$ be
a congruence subgroup of finite index~$\mu$ such that $\Gamma_y = \{1\}$ for all~$y \in \HH$
(e.g.\ $\Gamma = \Gamma_1(N)$ with $N \ge 4$). We can and do consider $\Gamma$ as a subgroup of $\PSL_2(\ZZ)$.

Then
$$ \dim_K \h^1(\Gamma,V_{k-2}(K)) = (k-1) \frac{\mu}{6} + \delta_{k,2}$$
and
$$ \dim_K \Hpar^1(\Gamma,V_{k-2}(K)) = (k-1) \frac{\mu}{6} - \nu_\infty + 2 \delta_{k,2},$$
where $\nu_\infty$ is the number of cusps of~$\Gamma$, i.e.\ the cardinality of $\Gamma \backslash \PP^1(\QQ)$,
and $\delta_{k,2} = \begin{cases}1 & \textnormal{if } k=2\\0 & \textnormal{otherwise.}\end{cases}$
\end{proposition}

\begin{proof}
Let $M = \Coind_\Gamma^{\PSL_2(\ZZ)} (V_{k-2}(K))$. This module has dimension
$(k-1)\mu$. From the Mayer-Vietoris exact sequence
$$0 \to M^{\PSL_2(\ZZ)} \to M^{\langle \sigma \rangle} \oplus M^{\langle \tau \rangle}
 \to M \to \h^1(\PSL_2(\ZZ),M) \to 0,$$
we obtain
$$ \dim \h^1(\Gamma,V_{k-2}(K)) = \dim M + \dim M^{\PSL_2(\ZZ)} 
- \dim \h^0 (\langle \sigma \rangle,M) - \dim \h^0 (\langle \tau \rangle,M).$$
Recall the left $\PSL_2(\ZZ)$-action on $\Hom_{K[\Gamma]}(K[\PSL_2(\ZZ)],V_{k-2}(K))$,
which is given by $(g.\phi)(h) = \phi(hg)$. It follows directly that
every function in the $K$-vector space $\Hom_{K[\Gamma]}(K[\PSL_2(\ZZ)],V_{k-2}(K))^{\PSL_2(\ZZ)}$
is constant and equal to its value at~$1$.
The $\Gamma$-invariance, however, imposes additionally
that this constant lies in $V_{k-2}(K)^\Gamma$.
Hence, by Lemma~\ref{lemvknull}, $\dim M^{\PSL_2(\ZZ)} = \delta_{k,2}$.
The term $\h^0 (\langle \sigma \rangle,M)$ is handled by Mackey's formula:
$$\dim \h^0 (\langle \sigma \rangle,M) 
= \sum_{x \in \Gamma \backslash \PSL_2(\ZZ).i} \dim V_{k-2}(K)^{\Gamma_x}
= (k-1) \#(\Gamma \backslash \PSL_2(\ZZ).i) = (k-1) \frac{\mu}{2},$$
since all $\Gamma_x$ are trivial by assumption and there are hence precisely
$\mu/2$ points in $Y_\Gamma$ lying over~$i$ in $Y_{\SL_2(\ZZ)}$.
By the same argument we get
$$\dim \h^0 (\langle \tau \rangle,M) = \frac{\mu}{3}.$$
Putting these together gives the first formula:
$$ \dim_K \h^1(\Gamma,V_{k-2}(K)) = (k-1) (\mu - \frac{\mu}{2} - \frac{\mu}{3}) + \delta_{k,2} =
(k-1)\frac{\mu}{6} + \delta_{k,2}.$$

The second formula can be read off from the diagram in Proposition~\ref{leraypar}.
It gives directly
\begin{multline*}
 \dim \Hpar^1(\Gamma,V_{k-2}(K)) =
 \dim \h^1(\Gamma,V_{k-2}(K)) + \dim V_{k-2}(K)_\Gamma \\ 
- \sum_{g \in \Gamma \backslash \PSL_2(\ZZ) / \langle T \rangle}
\dim \h^1(\Gamma \cap \langle g T g^{-1} \rangle, V_{k-2}(K)).
\end{multline*}
All the groups $\Gamma \cap \langle g T g^{-1} \rangle$ are of the form
$\langle T^n \rangle$ for some $n \ge 1$. Since they are cyclic, we have
$$\dim \h^1(\Gamma \cap \langle g T g^{-1} \rangle, V_{k-2}(K)) = 
\dim V_{k-2}(K)_{\langle T^n \rangle} = 1$$
by Lemma~\ref{lemvknull}. As the set $\Gamma \backslash \PSL_2(\ZZ) / \langle T \rangle$
is the set of cusps of~$\Gamma$, we conclude
$$ \sum_{g \in \Gamma \backslash \PSL_2(\ZZ) / \langle T \rangle}
\dim \h^1(\Gamma \cap \langle g T g^{-1} \rangle, V_{k-2}(K)) = \nu_\infty.$$
Moreover, also by Lemma~\ref{lemvknull}, $\dim V_{k-2}(K)_\Gamma = \delta_{k,2}$.
Putting everything together yields the formula
$$ \dim \Hpar^1(\Gamma,V_{k-2}(K)) = (k-1)\frac{\mu}{6} + 2 \delta_{k,2} - \nu_\infty,$$
as claimed.
\end{proof}

\begin{remark}\label{remdimcoh}
One can derive a formula for the dimension
even if $\Gamma$ is not torsion-free. One only needs to compute the dimensions 
$V_{k-2}(K)^{\langle \sigma \rangle}$ and
$V_{k-2}(K)^{\langle \tau \rangle}$ and to modify the above proof slightly.
\end{remark}

\subsection{Theoretical exercises}

\begin{exercise}\label{exfreegroup}
\begin{enumerate}[(a)]
\item
Verify that $G*H$ is a group.
\item
Prove the universal property represented by the commutative diagram
$$ \xymatrix@=0.5cm{
& P \\
G\ar@{^(->}^{\eta_G}[ur]\ar@{^(->}_{\iota_G}[dr] && H \ar@{^(->}_{\eta_H}[ul]\ar@{^(->}^{\iota_H}[dl]\\
& G * H.\ar@{=>}^\phi[uu] & }$$
More precisely, let $\iota_G: G \to G*H$ and $\iota_H:H \to G*H$ be the natural inclusions.
Let $P$ be any group together with group injections
$\eta_G : G \to P$ and $\eta_H:H \to P$, then there is a unique
group homomorphism $\phi: G*H \to P$ such that
$\eta_G = \phi \circ \iota_G$ and $\eta_H = \phi \circ \iota_H$.
\end{enumerate}
\end{exercise}

\begin{exercise}\label{exfiniteorder}
\begin{enumerate}[(a)]
\item Let $M \in \SL_n(\ZZ)$ be an element of finite order~$m$.
Determine the primes that may divide~$m$.
[Hint: Look at the characteristic polynomial of $M$.]
\item Determine all conjugacy classes of elements of finite
order in $\PSL_2(\ZZ)$.
\end{enumerate}
\end{exercise}

\begin{exercise}\label{exfotwo}
\begin{enumerate}[(a)]
\item Determine the $N \ge 1$ for which $\Gamma_1(N)$
has no element of finite order apart from the identity.
[Hint: You should get $N \ge 4$.]
\item Determine the $N \ge 1$ for which $\Gamma_0(N)$
has no element of order~$4$. Also determine the cases in which
there is no element of order~$6$.
\end{enumerate}
\end{exercise}

\begin{exercise}\label{exmv}
\begin{enumerate}[(a)]
\item Prove that the explicit description of~$f_m$ in the Mayer-Vietoris sequence
(Equation~\ref{mayervietoris}) satisfies the properties required for the $0$-th connecting homomorphism
in Definition~\ref{defi:deltafunctor}.

Hint: Prove that if $f_m$ is a boundary, then $m \in M^{\langle \sigma \rangle} + M^{\langle \tau \rangle}$.
Moreover, prove that a $1$-cocycle in $\h^1(\PSL_2(\ZZ),M)$ which becomes a coboundary when restricted to either 
$\langle \sigma \rangle$ or $\langle \tau \rangle$ can be changed by a coboundary to be of the form~$f_m$ for
some $m \in M$.

\item Let $0 \to A \to B \to C \to 0$ be an exact sequence of $G$-modules
for some group~$G$.
Let $c \in C^G$ and write it as a class $b+A \in B/A \cong C$. As it is $G$-invariant,
we have $0=(1-g)c=(1-g)(b+A)$, whence $(1-g)b \in A$ for all $g \in G$.
Define the $1$-cocycle $\delta^0(c)$ as the map $G \to A$ sending $g$ to $(1-g)b\in A$.

Prove that $\delta^0$ satisfies the properties required for the $0$-th connecting homomorphism
in Definition~\ref{defi:deltafunctor}.

Note that the connecting homomorphisms are not unique (one can, e.g.\ replace them by their negatives).

\item As an alternative approach to~(a), you may apply~(b) to the exact sequence from which the Mayer-Vietoris
sequence is derived as the associated long cohomology sequence in Proposition~\ref{mv}.

\end{enumerate}
\end{exercise}

\begin{exercise}\label{exparcompat}
Verify the commutativity of the diagram in Proposition~\ref{leraypar}.
\end{exercise}

\subsection{Computer exercises}

\begin{cexercise}\label{cexpeins}
Let $N \ge 1$. Compute a list of the elements of $\PP^1(\ZZ/N\ZZ)$.
Compute a list of the cusps of $\Gamma_0(N)$ and $\Gamma_1(N)$
(cf.\ \cite{Stein}, p.~60). 
I recommend to use the decomposition of $\PP^1(\ZZ/N\ZZ)$ into $\PP^1(\ZZ/p^n\ZZ)$.
\end{cexercise}

\begin{cexercise}\label{cexdirichlet}
Let $K$ be some field. Let $\chi: (\ZZ/N\ZZ)^\times \to K^\times$ be
a Dirichlet character of modulus~$N$. For given $N$ and $K$, compute the
group of all Dirichlet characters. Every Dirichlet character should be
implemented as a map $\phi: \ZZ \to K^\times$ such that $\phi(a) = 0$ 
for all $a \in \ZZ$ with $(a,N) \neq 1$ and $\phi(a) = \chi(a \mod N)$
otherwise.
\end{cexercise}

\section{Modular symbols and Manin symbols}

\subsection{Theory: Manin symbols}

This section is an extended version of a specialisation of parts of my article~\cite{MS} to the group $\PSL_2(\ZZ)$.
Manin symbols provide an alternative description of modular symbols. 
See Definition~\ref{defi:manin} below.
We shall use this description for the comparison with
group cohomology and for implementing the modular symbols formalism.
We stay in the general setting over a ring~$R$.

\begin{proposition}\label{hparnull}
The sequence of $R$-modules
$$0 \to R[\PSL_2(\ZZ)]N_\sigma + R[\PSL_2(\ZZ)]N_\tau 
 \to R[\PSL_2(\ZZ)] \xrightarrow{g \,\mapsto\, g(1-\sigma)\infty} R[\PP^1(\QQ)] 
\xrightarrow{g\infty \,\mapsto\, 1} R \to 0$$
is exact. Here we are considering $R[\PSL_2(\ZZ)]$ as a right $R[\PSL_2(\ZZ)]$-module.
\end{proposition}

\begin{proof}
Let $H$ be a finite subgroup of a group~$G$ and let
$H \backslash G = \{g_i \;|\; i \in I\}$ stand for a fixed system of representatives of the cosets.
We write $R[H \backslash G]$ for the free $R$-module on the set of representatives.
The map
$$ \Hom_R(R[H], R[H \backslash G]) \to R[G],\;\;\; f \mapsto \sum_{h \in H} h. f(h)$$
is an isomorphism.
Indeed, suppose that for $f \in \Hom_R(R[H], R[H \backslash G])$ we have
$$ 0 = \sum_{h \in H} h.(f(h)) = \sum_{h \in H} h.(\sum_{i \in I} a_{h,i} g_i)
=  \sum_{h \in H} (\sum_{i \in I} a_{h,i} hg_i), $$
then $a_{h,i}=0$ for all $h\in H$ and all $i\in I$ (since the elements $hg_i$ are all distinct),
whence $f=0$. For the surjectivity, note that all elements in $R[G]$ can be written as
(finite) sums of the form $\sum_{h \in H} \sum_{i \in I} a_{h,i} hg_i$ because any element in $G$
is of the form $hg_i$ for a unique $h\in H$ and a unique $i \in I$.

This yields via Shapiro's lemma that
$$H^i(\langle \sigma \rangle, R[\PSL_2(\ZZ)]) = 
H^i(\langle 1 \rangle, R[\langle \sigma \rangle \backslash \PSL_2(\ZZ)]) = 0$$
for all $i \ge 1$, and similarly for $\langle \tau \rangle$.
The resolution for a finite cyclic group \eqref{eq:res-cyclic-finite} gives
\begin{align*}
 R[\PSL_2(\ZZ)]N_\sigma &= \ker_{R[\PSL_2(\ZZ)]} (1-\sigma) 
                             = R[\PSL_2(\ZZ)]^{\langle \sigma \rangle},\\
 R[\PSL_2(\ZZ)]N_\tau &= \ker_{R[\PSL_2(\ZZ)]} (1-\tau)
                             = R[\PSL_2(\ZZ)]^{\langle \tau \rangle}, \\
 R[\PSL_2(\ZZ)](1-\sigma) &= \ker_{R[\PSL_2(\ZZ)]} N_\sigma \;\;\; \text{ and } \\
 R[\PSL_2(\ZZ)](1-\tau) &= \ker_{R[\PSL_2(\ZZ)]} N_\tau.
\end{align*}
By Proposition~\ref{mvbieri}, we have the exact sequence
$$ 0 \to R[\PSL_2(\ZZ)] \to R[\PSL_2(\ZZ)]_{\langle \sigma \rangle} \oplus R[\PSL_2(\ZZ)]_{\langle \tau \rangle}
 \to R \to 0.$$
The injectivity of the first map in the exact sequence (which we recall
is a consequence of $\PSL_2(\ZZ) = \langle \sigma \rangle * \langle \tau \rangle$) leads to
$$ R[\PSL_2(\ZZ)](1-\sigma) \cap R[\PSL_2(\ZZ)](1-\tau) = 0.$$

Sending $g$ to $g\infty$ yields a bijection between
$R[\PSL_2(\ZZ)]/R[\PSL_2(\ZZ)](1-T)$ and $R[\PP^1(\QQ)]$.
In order to prove the exactness at $R[\PSL_2(\ZZ)]$,
we show that the equality
$ x(1-\sigma) = y (1-T)$ for $x,y \in R[\PSL_2(\ZZ)]$ yields that
$x$ belongs to $R[\PSL_2(\ZZ)]^{\langle \sigma \rangle} + R[\PSL_2(\ZZ)]^{\langle \tau \rangle}$.

Note that $x(1-\sigma) = y(1-T) = y(1-\tau) -yT(1-\sigma)$ because of the equality $\tau = T\sigma$.
This yields $x(1-\sigma) +yT(1-\sigma) = y(1-\tau)$. This expression, however,
is equal to zero. Hence, there exists a $z \in R[\PSL_2(\ZZ)]$ satisfying $y = zN_\tau$.
We have $N_\tau T = N_\tau \sigma$ because of $T = \tau \sigma$.
Consequently, we get
$$ y(1-T) = z N_\tau(1-T) = z N_\tau (1-\sigma) = y(1-\sigma).$$
The equality $x(1-\sigma) = y(1-\sigma)$ implies that 
$x-y$ belongs to $R[\PSL_2(\ZZ)]^{\langle \sigma \rangle}$. 
Since $y \in R[\PSL_2(\ZZ)]^{\langle \tau \rangle}$,
we see get that $x = (x-y) + y$ lies in $R[\PSL_2(\ZZ)]^{\langle \sigma \rangle} + R[\PSL_2(\ZZ)]^{\langle \tau \rangle}$,
as required.

It remains to prove the exactness at $R[\PP^1(\QQ)]$.
The kernel of $R[\PSL_2(\ZZ)] \xrightarrow{g \mapsto 1} R$ is the augmentation ideal,
which is generated by all elements of the $1-g$ for $g \in \PSL_2(\ZZ)$. Noticing further that we can write
$$ 1-\alpha\beta = \alpha.(1-\beta)+(1-\alpha)$$
for $\alpha,\beta \in \PSL_2(\ZZ)$, the fact that $\sigma$ and $T=\tau\sigma$ generate~$\PSL_2(\ZZ)$ implies
that the kernel of $R[\PSL_2(\ZZ)] \xrightarrow{g \mapsto 1} R$ equals
$$R[\PSL_2(\ZZ)](1-\sigma) + R[\PSL_2(\ZZ)](1-T)$$ 
inside $R[\PSL_2(\ZZ)]$
It suffices to take the quotient by $R[\PSL_2(\ZZ)](1-T)$ to obtain the desired exactness.
\end{proof}

\begin{lemma}\label{mrlem}
The sequence of $R$-modules
$$ 0 \to \cM_R \xrightarrow{\{\alpha,\beta\} \mapsto \beta - \alpha} R[\PP^1(\QQ)] 
\xrightarrow{\alpha \mapsto 1} R \to 0$$
is exact.
\end{lemma}

\begin{proof}
Note that, using the relations defining~$\cM_R$, any element in $\cM_R$ can be written
$\sum_{\alpha \neq \infty} r_\alpha \{\infty, \alpha\}$ with $r_\alpha \in R$.
This element is mapped to $\sum_{\alpha \neq \infty} r_\alpha \alpha - (\sum_{\alpha \neq \infty} r_\alpha) \infty$.
If this expression equals zero, all coefficients $r_\alpha$ have to be zero.
This shows the injectivity of the first map.

Let $\sum_\alpha r_\alpha \alpha \in R[\PP^1(\QQ)]$ be an element in the kernel
of the second map. Then $\sum_\alpha r_\alpha = 0$, so that we can write
$$\sum_\alpha r_\alpha \alpha = \sum_{\alpha \neq \infty} r_\alpha \alpha - 
(\sum_{\alpha \neq \infty} r_\alpha) \infty$$
to obtain an element in the image of the first map.
\end{proof}

\begin{proposition}\label{propker}
The homomorphism of $R$-modules
$$ R[\PSL_2(\ZZ)] \xrightarrow{\phi} \cM_R,\;\;\;
g \mapsto \{g.0,g.\infty\}$$
is surjective with kernel
$R[\PSL_2(\ZZ)]N_\sigma + R[\PSL_2(\ZZ)]N_\tau$. 
\end{proposition}

\begin{proof}
This follows from Proposition~\ref{hparnull} and Lemma~\ref{mrlem}.
\end{proof}

We have now provided all the input required to prove the description of modular symbols in terms of
Manin symbols. For this we need the notion of an induced module. In homology
it plays the role that the coinduced module plays in cohomology.

\begin{definition}
Let $R$ be a ring, $G$ a group, $H \le G$ a subgroup and $V$ a left
$R[H]$-module. The {\em induced module} of $V$ from $H$ to~$G$ is defined as
$$ \Ind_H^G (V) := R[G] \otimes_{R[H]} V,$$
where we view $R[G]$ as a right $R[H]$-module via the natural action.
The induced module is a left $R[G]$-module via the natural left
action of $G$ on~$R[G]$.
\end{definition}

In case of $H$ having a finite index in~$G$
(as in our standard example $\Gamma_1(N) \le \PSL_2(\ZZ)$), the induced module is
isomorphic to the coinduced one:

\begin{lemma}\label{indcoind}
Let $R$ be a ring, $G$ a group, $H \le G$ a subgroup of finite index and $V$ a left
$R[H]$-module. 
\begin{enumerate}[(a)]
\item $\Ind_H^G(V)$ and $\Coind_H^G(V)$ are isomorphic as left $R[G]$-modules.
\item Equip $(R[G] \otimes_R V)$ with the diagonal left $H$-action
$h.(g \otimes v) = h g \otimes h.v$ and the right
$G$-action $(g\otimes v).\tilde{g} = g\tilde{g} \otimes v$.
Consider the induced module $\Ind_H^G (V)$ as a
right $R[G]$-module by inverting the left action in the definition.
Then
$$ \Ind_H^G(V) \to (R[G] \otimes_R V)_H,\;\;\;
   g\otimes v \mapsto g^{-1} \otimes v$$
is an isomorphism of right $R[G]$-modules.
\end{enumerate}
\end{lemma}

\begin{proof}
Exercise~\ref{exindcoind}.
\end{proof}

\begin{definition}\label{defi:manin}
Let $\Gamma \subseteq \PSL_2(\ZZ)$ be a finite index subgroup, $V$ a left $R[\Gamma]$-module and consider
$M = \Ind_\Gamma^{\PSL_2(\ZZ)} (V)$, which we identify with the
right $R[\PSL_2(\ZZ)]$-module
$(R[\PSL_2(\ZZ)] \otimes_R V)_\Gamma$ as in Lemma~\ref{indcoind}~(b).

Elements in $M / (M N_\sigma + M N_\tau)$ are called {\em Manin symbols} over~$R$ (for the subgroup $\Gamma \subseteq \PSL_2(\ZZ)$ and the
left $R[\Gamma]$-module~$V$).
\end{definition}

\begin{theorem}\label{ManinSymbols}
In the setting of Definition~\ref{defi:manin}, the following statements hold:
\begin{enumerate}[(a)]
\item The homomorphism $\phi$ from Proposition~\ref{propker} induces the
exact sequence of $R$-modules
$$ 0 \to M N_\sigma + M N_\tau \to M \to \cM_R(\Gamma,V) \to 0,$$
and the homomorphism $M \to \cM_R(\Gamma,V)$ is given by 
$g \otimes v \mapsto \{g.0, g.\infty\} \otimes v$.

In other words, this map induces an isomorphism between Manin symbols over~$R$ 
(for the subgroup $\Gamma \subseteq \PSL_2(\ZZ)$ and the left $R[\Gamma]$-module~$V$)
and the modular symbols module $\cM_R(\Gamma,V)$.
\item The homomorphism $R[\PSL_2(\ZZ)] \to R[\PP^1(\QQ)]$ sending $g$ to $g.\infty$
induces the exact sequence of $R$-modules
$$ 0 \to M(1-T) \to M \to \cB_R(\Gamma,V) \to 0.$$
\item The identifications of (a) and~(b) imply the isomorphism
$$ \cCM_R(\Gamma,V) \cong \ker\big(
M/ (M N_\sigma + M N_\tau ) \xrightarrow{m \mapsto m(1-\sigma)} M/M(1-T) \big).$$
\end{enumerate}
\end{theorem}

\begin{proof}
(a) Proposition~\ref{propker} gives the exact sequence
$$ 0 \to R[\PSL_2(\ZZ)]N_\sigma + R[\PSL_2(\ZZ)]N_\tau \to R[\PSL_2(\ZZ)] \to \cM_R \to 0,$$
which we tensor with $V$ over $R$, yielding the exact sequence of left $R[\Gamma]$-modules
$$ 0 \to (R[\PSL_2(\ZZ)] \otimes_R V) N_\sigma +  (R[\PSL_2(\ZZ)] \otimes_R V) N_\tau 
\to (R[\PSL_2(\ZZ)] \otimes_R V) \to \cM_R(V) \to 0.$$
Passing to left $\Gamma$-coinvariants yields~(a) because $M N_\sigma$ and $M N_\tau$ are the images
of $(R[\PSL_2(\ZZ)] \otimes_R V) N_\sigma$ and $(R[\PSL_2(\ZZ)] \otimes_R V) N_\tau$ inside $M$,
respectively.
(b) is clear from the definition and (c) has already been observed
in the proof of Proposition~\ref{hparnull}.
\end{proof}

In the literature on Manin symbols one usually finds a more explicit
version of the induced module. This is the contents of the following
proposition. It establishes the link with the main theorem
on Manin symbols in \cite{Stein}, namely Theorem~8.4.

Since in the following proposition left and right actions
are involved, we sometimes indicate left (co-)invariants by
using left subscripts (resp.\ superscripts) and right
(co-)invariants by right ones.

\begin{proposition}\label{indprop}
Let $\chi: (\ZZ/N\ZZ)^\times \to R^\times$ be a character such that
$\chi(-1) = (-1)^k$.
Consider the $R$-module 
$$X := R[\Gamma_1(N) \backslash \SL_2(\ZZ)] \otimes_R V_{k-2}(R) \otimes_R R^\chi$$
equipped with the right $\SL_2(\ZZ)$-action
$(\Gamma_1(N) h \otimes V \otimes r)g = (\Gamma_1(N) hg \otimes g^{-1}v \otimes r)$
and with the left $\Gamma_1(N) \backslash \Gamma_0(N)$-action 
$g (\Gamma_1(N) h \otimes v \otimes r) = (\Gamma_1(N) g h \otimes v \otimes \chi(g)r)$. 

Then 
$$ X \cong \Ind_{\Gamma_1(N)}^{\SL_2(\ZZ)}(V_k^\chi(R))$$
as a right $R[\SL_2(\ZZ)]$-module and a left $R[\Gamma_1(N) \backslash \Gamma_0(N)]$-module.
Moreover,
$${}_{\Gamma_1(N) \backslash \Gamma_0(N)} X \cong \Ind_{\Gamma_0(N)}^{\SL_2(\ZZ)}(V_k^\chi(R)).$$ 
If $N \ge 3$, then the latter module is isomorphic
to $\Ind_{\Gamma_0(N)/\{\pm 1\}}^{\PSL_2(\ZZ)}(V_k^\chi(R))$. 
\end{proposition}

\begin{proof}
Mapping 
$g \otimes v \otimes r$ to $g \otimes g^{-1}v \otimes r$
defines an isomorphism of right $R[\SL_2(\ZZ)]$-modules and of left 
$R[\Gamma_1(N) \backslash \Gamma_0(N)]$-modules
$$ {}_{\Gamma_1(N)} (R[\SL_2(\ZZ)] \otimes_R V_{k-2}(R) \otimes_R R^\chi) \to X.$$
As we have seen above, the left hand side module is naturally isomorphic
to the induced module $\Ind_{\Gamma_1(N)}^{\SL_2(\ZZ)}(V_k^\chi(R))$ (equipped with its right
$R[\SL_2(\ZZ)]$-action described before). This establishes the first statement. 
The second one follows from
${}_{\Gamma_1(N) \backslash \Gamma_0(N)} \big( {}_{\Gamma_1(N)} M \big ) = {}_{\Gamma_0(N)} M$ for
any $\Gamma_0(N)$-module~$M$. 
The third statement is due to the fact that
${}_{\langle -1 \rangle} (R[\SL_2(\ZZ)] \otimes_R V_{k-2}^\chi(R))$
is naturally isomorphic to
$R[\PSL_2(\ZZ)] \otimes_R V_{k-2}^\chi(R)$, since $-1$ acts
trivially on the second factor, as the assumption assures that
$-1 \in \Gamma_0(N)$ but $-1 \not\in \Gamma_1(N)$. 
\end{proof}

For one more description of the induced module 
$\Ind_{\Gamma_0(N)/\{\pm 1\}}^{\PSL_2(\ZZ)}(V_k^\chi(R))$
see Exercise~\ref{exmaninp}.
It is this description that uses up the least memory in an implementation.
Now all the prerequisites have been provided for implementing
Manin symbols (say for $\Gamma_0(N)$ and a character). This is the
task of Computer Exercise~\ref{cexmanin}.

\subsection{Theory: Manin symbols and group cohomology}

Let $\Gamma \le \PSL_2(\ZZ)$ be a subgroup of finite index,
and $V$ a left $R[\Gamma]$-module for a ring~$R$.

\begin{theorem}\label{compthm}
Suppose that the orders of all stabiliser subgroups of~$\Gamma$ 
for the action on~$\HH$ are invertible in~$R$.
Then we have isomorphisms:
$$ \h^1(\Gamma,V) \cong \cM_R(\Gamma,V)$$
and
$$ \Hpar^1(\Gamma,V) \cong \cCM_R(\Gamma,V).$$
\end{theorem}

\begin{proof}
As before, set $M=\Ind_\Gamma^{\PSL_2(\ZZ)} (V)$ and recall that this module is isomorphic to $\Coind_\Gamma^{\PSL_2(\ZZ)} (V)$.
To see the first statement, in view of Theorem~\ref{ManinSymbols}
and the corollary of the Mayer-Vietoris exact sequence (Corollary~\ref{corhzweinull}), it suffices to
show $M^{\langle \sigma \rangle} = MN_\sigma$ and $M^{\langle \tau \rangle} = MN_\tau$.
By the resolution of~$R$ for a cyclic group in~\eqref{eq:res-cyclic-finite}, the quotient $M^{\langle \sigma \rangle} / MN_\sigma$
is equal to $H^2(\langle \sigma \rangle,M)$, but this one is zero by the application of Mackey's formula done in Lemma~\ref{mackeypsl}~(c).
The same argument works with $\tau$ instead of~$\sigma$.

The passage to the parabolic/cuspidal subspaces is immediate because the boundary map with source $M$ has the same explicit description in both
cases (see Theorem~\ref{ManinSymbols}~(c) and Proposition~\ref{leraypar}).
\end{proof}

\subsection{Algorithms and Implementations: Conversion between Manin and mo\-dular symbols}

We now use the Euclidean Algorithm to represent any element
$g \in \PSL_2(\ZZ)$ in terms of $\sigma$ and~$T$.

\begin{algorithm}\label{algpsl}
\underline{Input}: A matrix~$M = \mat abcd$ with integer entries and determinant~$1$.

\underline{Output}: A list of matrices $[A_1,A_2,\dots,A_n]$
where all $A_i \in \{T^n | n \in \ZZ\} \cup \{\sigma\}$ and $\sigma$ and $T^n$ alternate.

\begin{enumerate}[(1)]
\itemsep=0cm plus 0pt minus 0pt

\item create an empty list {\tt output}.
\item if  $|c| > |a|$ then
\item \ms  append $\sigma$ to {\tt output}.
\item \ms $M := \sigma M$.
\item end if;
\item while $c \neq 0$ do
\item \ms $q := a \textnormal{ div } c$.
\item \ms  append $T^q$ to {\tt output}.
\item \ms  append $\sigma$ to {\tt output}.
\item \ms  $M := \sigma T^{-q} M$.
\item end while;
\item if $M \not\in \{\mat 1001,\mat {-1}00{-1}\}$ then $\;\;\;$
[At this point $M \in \{ \mat 1*01, \mat {-1}*0{-1}\}$.]
\item \ms  append $M$ to {\tt output}.
\item end if;
\item return {\tt output}.
\end{enumerate}
\end{algorithm}

This algorithm gives a constructive proof of the fact (Proposition~\ref{prop:gensl2z}) that $\PSL_2(\ZZ)$
is generated by $\sigma$ and~$T$, and hence also by $\sigma$ and~$\tau$.
Note, however, that the algorithm does not necessarily give the shortest
such representation. See Exercise~\ref{exctdfrac} for a relation to
continued fractions.

We can use the algorithm to make a conversion between modular symbols and
Manin symbols, as follows. Suppose we are given the modular symbols
$\{\alpha,\infty\}$ (this is no loss of generality, as we can represent
$\{\alpha,\beta\} = \{\alpha,\infty\} - \{\beta,\infty\}$).
Suppose $\alpha$ is given as $g \infty$ with some $g \in \SL_2(\ZZ)$
(i.e.\ representing the cusp as a fraction $\frac{a}{c}$ with $(a,c)=1$,
then we can find $b,d$ by the Euclidean Algorithm such that $g=\mat abcd \in \SL_2(\ZZ)$
satisfies the requirements).
We now use Algorithm~\ref{algpsl} to represent $g$ as $\sigma T^{a_1} \sigma T^{a_2} \sigma
\dots T^{a_n} \sigma$ (for example). Then we have
$$ \{\alpha,\infty\} = \sigma T^{a_1} \sigma T^{a_2} \sigma \dots T^{a_n} \{0,\infty\}
+ \sigma T^{a_1} \sigma T^{a_2} \sigma \dots T^{a_{n-1}} \{0,\infty\} +  \dots +
\sigma T^{a_1} \{0,\infty\} + \{0,\infty\}.$$
If $g$ does not end in~$\sigma$ but $T^{a_n}$, then we must drop $T^{a_n}$ from the
above formula (since $T$ stabilises~$\infty$).
If $g$ starts in~$T^{a_1}$ (instead of~$\sigma$), then we must drop
the last summand.

Since we are in weight~$2$ (i.e.\ trivial module~$V$), the space of Manin symbols is a quotient
of $R[\PSL_2(\ZZ)]/\Gamma$ (see Definition~\ref{defi:manin}).
The Manin symbol corresponding to the above example chosen for the modular symbol $\{\alpha,\infty\}$ is then simply represented by
the formal sum
\begin{equation}\label{eq:maninex}
\sigma T^{a_1} \sigma T^{a_2} \sigma \dots T^{a_n} 
+ \sigma T^{a_1} \sigma T^{a_2} \sigma \dots T^{a_{n-1}}  + \dots +
\sigma T^{a_1}  + 1. 
\end{equation}
If the module $V$ is not trivial, a modular symbol would typically look like $\{\alpha,\infty\} \otimes v$ for $v \in V$
and the corresponding Manin symbol would be the formal sum in~\eqref{eq:maninex} tensored with~$v$.

In Computer Exercise~\ref{cexconv} you are asked to implement a conversion
between Manin and modular symbols.

\subsection{Theoretical exercises}

\begin{exercise}\label{exindcoind}
Prove Lemma~\ref{indcoind}.
\end{exercise}

\begin{exercise}\label{exmaninp}
Assume the set-up of Proposition~\ref{indprop}.
Describe a right $\PSL_2(\ZZ)$-action on
$$ Y := R[\PP^1(\ZZ/N\ZZ)] \otimes_R V_{k-2}(R) \otimes_R R^\chi$$
and an isomorphism
$$ {}_{\Gamma_1(N)\backslash \Gamma_0(N)} X \to Y$$
of right $\PSL_2(\ZZ)$-modules.
\end{exercise}

\begin{exercise}\label{exctdfrac}
Provide a relationship between Algorithm~\ref{algpsl} and continued fractions.
\end{exercise}

\subsection{Computer exercises}

\begin{cexercise}\label{cexmanin}
Use the description of Exercise~\ref{exmaninp} and your results from
Computer Exercises \ref{cexpeins} and~\ref{cexdirichlet} to implement Manin symbols
for $\Gamma_0(N)$ and a character over a field.
As a first approach you may use the trivial character only.
\end{cexercise}

\begin{cexercise}\label{cexconv}
\begin{enumerate}[(a)]
\item Write an algorithm to represent any element of the group $\PSL_2(\ZZ)$ in terms
of $\sigma$ and~$T$.
\item Write an algorithm that represents any modular symbol $\{\alpha,\beta\}$
as a Manin symbol (inside the vector space created in Computer Exercise~\ref{cexmanin}).
\end{enumerate}
\end{cexercise}

\section{Eichler-Shimura}

This section is devoted to proving the theorem by Eichler and Shimura that is at the
basis of the modular symbols algorithm and its group cohomological variant.
The standard reference for the Eichler-Shimura theorem is~\cite[\S 8.2]{Shimura}.
In the entire section, let $k \ge 2$ be an integer.

\subsection{Theory: Petersson scalar product}

Recall the standard fundamental domain for $\SL_2(\ZZ)$
$$ \cF = \{ z=x+iy \in \HH \,| \, |z| > 1, |x| <  \frac{1}{2}\}$$
from Proposition~\ref{prop:fd}.
Every subgroup $\Gamma \le \SL_2(\ZZ)$ of finite index has a fundamental
domain, for example, $\bigcup_{\gamma \in \Gammabar \backslash \PSL_2(\ZZ)} \gamma \cF$
for any choice of system of representatives of the cosets $\Gammabar \backslash \PSL_2(\ZZ)$,
where we put $\Gammabar = \Gamma/(\langle \pm 1 \rangle \cap \Gamma)$.

\begin{lemma}\label{lmdiff}
\begin{enumerate}[(a)]
\item Let $\gamma  \in \GL_2(\RR)^+$ be a real matrix with positive determinant.
Let $f \in \Mkg k\Gamma\CC$ and $g \in \Skg k\Gamma\CC$. We have with
$z \in \HH$
$$ f(\gamma z) \overline{g(\gamma z)} (\gamma z - \gamma \zbar)^k = 
\det(\gamma)^{2-k} f|_\gamma (z) \overline{g|_\gamma (z)} (z - \zbar)^k$$
for all $\gamma \in \SL_2(\RR)$. The function
$G(z) := f(z) \overline{g(z)} (z-\zbar)^k$ is bounded on~$\HH$.
\item We have $d\gamma z = \frac{\det(\gamma)}{(cz+d)^2} dz$ for all $\gamma \in \GL_2(\RR)^+$.
\item The differential form $\frac{dz \wedge d\zbar}{(z-\zbar)^2}$ is
$\GL_2(\RR)^+$-invariant. In terms of $z=x+iy$ we have
$\frac{dz \wedge d\zbar}{(z-\zbar)^2} = \frac{i}{2} \frac{dx \wedge dy}{y^2}$.
\item Let $\Gamma \le \SL_2(\ZZ)$ be a subgroup
with finite index $\mu = (\PSL_2(\ZZ):\Gammabar)$.
The volume of any fundamental domain $\cF_\Gamma$ for~$\Gamma$ with
respect to the differential form $\frac{2 dz \wedge d\zbar}{i(z-\zbar)^2}$, i.e.\
$$ \vol (\cF_\Gamma) = \int_{\cF_\Gamma} \frac{2dz \wedge d\zbar}{i(z-\zbar)^2},$$
is equal to $\mu \frac{\pi}{3}$.
\end{enumerate}
\end{lemma}

\begin{proof}
(a) The first statement is computed as follows:
\begin{align*}
& f(\gamma z) \overline{g(\gamma z)} (\gamma z - \gamma \zbar)^k \\
 =& \det(\gamma)^{2(1-k)}(f|_\gamma (z) (cz+d)^k) \overline{(g|_\gamma(z) (cz +d)^k)} 
(\frac{az+b}{cz+d} - \frac{a\zbar+b}{c\zbar + d})^k\\
=&\det(\gamma)^{2-2k}f|_\gamma (z) \overline{g|_\gamma(z)} ((az+b)(c\zbar+d)-(a\zbar+b)(cz+d))^k\\
=& \det(\gamma)^{2-k} f|_\gamma (z) \overline{g|_\gamma (z)} (z - \zbar)^k,
\end{align*}
where we write $\gamma = \mat abcd$.
By the preceding computation, the function $G(z)$ is invariant under
$\gamma \in \Gamma$. Hence, it suffices to check that
$|G(z)|$ is bounded on the closure of any fundamental domain $\cF_\Gamma$ for~$\Gamma$.
For this, it is enough to verify for every $\gamma$ in a system of
representatives of $\Gamma \backslash \SL_2(\ZZ)$ that any of the functions $G(\gamma z)$
is bounded on the closure of the standard fundamental domain~$\cF$.
By the preceding computation, we also have 
$G(\gamma z) = f|_\gamma (z) \overline{g|_\gamma (z)} (z - \zbar)^k$ for
$\gamma \in \SL_2(\ZZ)$.
Note that $f(z) g(z)$ is a cusp form in $\Skg {2k}\Gamma\CC$, in particular,
for every $\gamma \in \SL_2(\ZZ)$ the function 
$f|_\gamma (z) g |_\gamma (z)$ has a Fourier expansion in~$\infty$ of
the form $\sum_{n=1}^\infty a_n e^{2 \pi i z n}$. This series converges
absolutely and uniformly on compact subsets of~$\HH$, in particular, for any $C > 1$
$$ K_\gamma := \sum_{n=1}^\infty |a_n e^{2 \pi i (x+iC) n}|
      = \sum_{n=1}^\infty |a_n| e^{-2 \pi C n} $$
is a positive real number, depending on $\gamma$ (in a system of representatives
$\Gamma \backslash \SL_2(\ZZ)$).
We have with $ z = x + iy$ and $y \ge C$
\begin{align*}
 |G(\gamma z)| \le (2 y)^k \sum_{n=1}^\infty |a_n| e^{-2 \pi y n}
&= (2 y)^k e^{-2 \pi y} \sum_{n=1}^\infty |a_n| e^{-2 \pi y (n-1)} \\
& \le (2 y)^k e^{-2 \pi y} \sum_{n=1}^\infty |a_n| e^{-2 \pi C (n-1)} \\
& \le (2 y)^k e^{-2 \pi y} K_\gamma e^{2 \pi C}.
\end{align*}
This tends to~$0$ if $y$ tends to~$\infty$. Consequently, the function~$G(\gamma z)$
is bounded on the closure of the standard fundamental domain, as desired.

(b) Again writing $\gamma = \mat abcd$ we have
$$\frac{d\gamma z}{dz}= \frac{d\frac{az+b}{cz+d}}{dz} = \frac{1}{(cz+d)^2} (a(cz+d)-(az+b)c) 
= \frac{\det(\gamma)}{(cz+d)^2},$$
which gives the claim.

(c) This is again a simple computation:
\begin{align*}
(\gamma z - \gamma \zbar)^{-2} d\gamma z \wedge d \gamma \zbar 
&= \det(\gamma)^2(\frac{az+b}{cz+d} - \frac{a\zbar+b}{c\zbar + d})^{-2} (cz+d)^{-2}(c\zbar+d)^{-2} dz\wedge d\zbar\\
&= (z-\zbar)^{-2} dz \wedge d\zbar,
\end{align*}
using~(b). The last statement is
$$ \frac{dz \wedge d\zbar}{(z-\zbar)^2} = \frac{(dx +idy)\wedge(dx - i dy)}{(2iy)^2} 
= \frac{-2i dx \wedge dy}{-4y^2} = \frac{i dx \wedge dy}{2y^2}.$$

(d) Due to the $\Gamma$-invariance, it suffices to show
$$ \int_\cF \frac{dz \wedge d\zbar}{(z - \zbar)^2} = \frac{i\pi}{6}.$$
Let $\omega = - \frac{dz}{z-\zbar}$. The total derivative of~$\omega$ is
$$d \omega = ((z-\zbar)^{-2} dz - (z - \zbar)^{-2} d \zbar) \wedge dz = \frac{dz \wedge d \zbar}{(z-\zbar)^2}.$$
Hence, Stokes' theorem yields
$$ \int_\cF \frac{dz \wedge d\zbar}{(z - \zbar)^2} = - \int_{\partial \cF} \frac{dz}{z-\zbar},$$
where $\partial \cF$ is the positively oriented border of~$\cF$, which we describe
concretely as the path $A$ from $\infty$ to~$\zeta_3$ on the vertical line, followed by the path
$C$ from $\zeta_3$ to~$\zeta_6$ on the unit circle and finally followed by $-TA$.
Hence with $z =x+iy$ we have
$$ \int_\cF \frac{dz \wedge d\zbar}{(z - \zbar)^2} 
= -\frac{1}{2i} \big(\int_A \frac{dz}{y} - \int_{TA} \frac{dz}{y} + \int_C \frac{dz}{y} \big)
= -\frac{1}{2i} \int_C \frac{dz}{y},$$
since $dz = dTz$. Using the obvious parametrisation of~$C$ we obtain
\begin{multline*}
 -\frac{1}{2i} \int_C \frac{dz}{y} 
= -\frac{1}{2i} \int_{2\pi/3}^{2\pi/6} \frac{1}{\Imag(e^{i\phi})} \frac{d e^{i\phi}}{d\phi}d\phi
= -\frac{1}{2} \int_{2\pi/3}^{2\pi/6} \frac{e^{i\phi}}{\Imag(e^{i\phi})} d\phi\\
= -\frac{1}{2} \int_{2\pi/3}^{2\pi/6} (\frac{\cos(\phi)}{\sin(\phi)}+ i) d\phi
= -\frac{i}{2} (\frac{2\pi}{6} - \frac{2\pi}{3}) = \frac{i\pi}{6},
\end{multline*}
since $\sin$ is symmetric around $\pi/2$ and $\cos$ is antisymmetric, so that
the integral over $\frac{\cos(\phi)}{\sin(\phi)}$ cancels.
\end{proof}

\begin{definition}
Let $\Gamma \le \SL_2(\ZZ)$ be a subgroup of finite index and 
$\mu := (\PSL_2(\ZZ):\Gammabar)$ be the index of
$\Gammabar = \Gamma/(\langle \pm 1 \rangle \cap \Gamma)$ in
$\PSL_2(\ZZ)$.
We define the {\em Petersson pairing} as
\begin{align*}
 \Mkg k\Gamma\CC \times \Skg k\Gamma\CC \to & \CC\\
(f,g) \mapsto &\frac{-1}{(2i)^{k-1}\mu} \int_{\cF_\Gamma} f(z) \overline{g(z)} (z-\zbar)^k 
\frac{dz \wedge d\zbar}{(z-\zbar)^2} \\
= &\frac{1}{\mu} \int_{\cF_\Gamma} f(z) \overline{g(z)} y^{k-2}
dx \wedge dy =: (f,g),
\end{align*}
where $\cF_\Gamma$ is any fundamental domain for~$\Gamma$.
\end{definition}

\begin{proposition}
\begin{enumerate}[(a)]
\item The integral in the Petersson pairing converges. It does not
depend on the choice of the fundamental domain $\cF_\Gamma$.
\item The Petersson pairing is a sesqui-linear pairing (linear in the first
and anti-linear in the second variable).
\item The restriction of the Petersson pairing to $\Skg k\Gamma\CC$ is a positive
definite scalar product (the {\em Petersson scalar product}).
\item If $f,g$ are modular (cusp) forms for the group~$\Gamma$ and $\Gamma' \le \Gamma$
is a subgroup of finite index, then the Petersson pairing of $f$ and~$g$
with respect to~$\Gamma$ gives the same value as the one with respect to~$\Gamma'$.
\end{enumerate}
\end{proposition}

\begin{proof}
(a) By Lemma~\ref{lmdiff} the integral converges, since the function
$$G(z) := f(z) \overline{g(z)} (z-\zbar)^k$$
is bounded on~$\cF_\Gamma$ and the
volume of $\cF_\Gamma$ for the measure in question is finite.
The integral does not depend on the choice of the fundamental domain
by the invariance of~$G(z)$ under~$\Gamma$.

(b) is clear.

(c) We have
$$(f,f) =\frac{1}{\mu}\int_{\cF_\Gamma} |f(z)|^2 y^{k-2}dx \wedge dy,$$
which is clearly non-negative. It is~$0$ if and only if $f$ is the zero function,
showing that the product is positive definite.

(d) If $\cF_\Gamma$ is a fundamental domain for~$\Gamma$, then
$\bigcup_{\gamma \in \Gamma' \backslash \Gamma} \gamma \cF_\Gamma$ is a fundamental
domain for~$\Gamma'$ (for any choice of representatives of $\Gamma' \backslash \Gamma$).
But on every $\gamma \cF_\Gamma$ the integral takes the same value.
\end{proof}

\begin{proposition}\label{petexplicit}
Let $f,g \in \Skg k\Gamma\CC$. We have
$$ (f,g) = \frac{-1}{(2i)^{k-1}\mu}\sum_{\gamma \in \Gammabar\backslash \PSL_2(\ZZ)} 
\int_{\zeta_3}^i \int_\infty^0 f|_\gamma (z) \overline{g|_\gamma(z)} (z-\zbar)^{k-2} dz d\zbar.$$
\end{proposition}

\begin{proof}
Let us write for short $G_\gamma (z,\zbar) = f|_\gamma (z) \overline{g|_\gamma(z)} (z-\zbar)^k$
for $\gamma \in \SL_2(\ZZ)$.
Then
$$ -(2i)^{k-1}\mu (f,g) 
= \int_{\bigcup_\gamma \gamma\cF} G(z,\zbar) \frac{dz \wedge d\zbar}{(z-\zbar)^2}
= \sum_\gamma \int_\cF G_\gamma(z,\zbar) \frac{dz \wedge d\zbar}{(z-\zbar)^2}$$
by Lemma~\ref{lmdiff}, where the union resp.\ sum runs over a fixed system
of coset representatives of $\Gammabar \backslash \PSL_2(\ZZ)$; by our
observations, everything is independent of this choice.
Consider the differential form
$$ \omega_\gamma := 
\big( \int_\infty^z f|_\gamma(u)(u-\zbar)^{k-2}du \big) \overline{g|_\gamma(z)} d\zbar.$$
Note that the integral converges since $f$ is a cusp form.
The total derivative of $\omega_\gamma$ is
$d\omega_\gamma = G_\gamma (z,\zbar) \frac{dz \wedge d\zbar}{(z-\zbar)^2}$.
Consequently, Stokes' theorem gives
$$ \sum_\gamma \int_\cF G_\gamma(z,\zbar) \frac{dz \wedge d\zbar}{(z-\zbar)^2} =
\sum_\gamma \int_{\partial \cF} \big( \int_\infty^z f|_\gamma(u)(u-\zbar)^{k-2}du \big) \overline{g|_\gamma(z)} d\zbar,$$
where as above $\partial \cF$ is the positively oriented border of the standard
fundamental domain~$\cF$, which we describe as the path~$A$ along the vertical line
from $\infty$ to $\zeta_3$, followed by the path~$B$ from $\zeta_3$ to~$i$ along
the unit circle, followed by $-\sigma B$ and by~$-TA$.

We now make a small calculation. Let for this $C$ be any (piecewise
continuously differentiable) path in~$\HH$ and $M \in \SL_2(\ZZ)$:
\begin{align*}
&\int_{M C} \int_\infty^z f|_\gamma(u) \overline{g|_\gamma(z)} (u-\zbar)^{k-2}du d\zbar\\
=& \int_C \int_\infty^{Mz} f|_\gamma(u) \overline{g|_\gamma(Mz)} (u-M\zbar)^{k-2} du \frac{dM\zbar}{d\zbar} d\zbar \\
=& \int_C \int_{M^{-1}\infty}^z f|_{\gamma M}(u) \overline{g|_{\gamma M}(z)} (u-\zbar)^{k-2} du d\zbar\\
=& \int_C \int_\infty^z f|_{\gamma M}(u) \overline{g|_{\gamma M}(z)} (u-\zbar)^{k-2} du d\zbar
- \int_C \int_\infty^{M^{-1}\infty} f|_{\gamma M}(u) \overline{g|_{\gamma M}(z)} (u-\zbar)^{k-2} du d\zbar.
\end{align*}
This gives
\begin{multline*}
\int_{C - MC} \int_{\infty}^z f|_\gamma(u) \overline{g|_\gamma(z)} (u-\zbar)^{k-2}du d\zbar = \\
\int_C \int_\infty^z (G_\gamma(u,\zbar)-G_{\gamma M}(u,\zbar))du d\zbar
+ \int_C \int_\infty^{M^{-1}\infty} G_{\gamma M}(u,\zbar) du d\zbar.
\end{multline*}
Continuing with the main calculation, we have
\begin{align*}
& -(2i)^{k-1}\mu(f,g) \\
=& \sum_\gamma \big[ \int_A \int_\infty^z (G_\gamma (u,\zbar) - G_{\gamma T} (u,\zbar)) du d\zbar
+ \int_A \int_\infty^{T^{-1}\infty} G_{\gamma T} (u,\zbar) du d\zbar \big] \\
+& \sum_\gamma \big[ \int_B \int_\infty^z (G_\gamma (u,\zbar) - G_{\gamma \sigma} (u,\zbar)) du d\zbar
+ \int_B \int_\infty^{\sigma^{-1}\infty} G_{\gamma \sigma} (u,\zbar) du d\zbar \big]\\
=& \sum_\gamma \int_B \int_\infty^0 G_{\gamma \sigma}(u,\zbar) du d\zbar,
\end{align*}
using $T^{-1} \infty = \infty$, $\sigma^{-1} \infty = 0$ and the fact that the $\gamma T$
and $\gamma \sigma$ are just permutations of the cosets.
\end{proof}

\subsection{Theory: The Eichler-Shimura map}

Let $\Gamma \le \SL_2(\ZZ)$ be a subgroup of finite index.
We fix some $z_0 \in \HH$.
For $f \in \Mkg k\Gamma\CC$ with $k \ge 2$ and $\gamma,\delta$ in $\ZZ^{2 \times2}$ with positive determinant, let
$$ I_f(\gamma z_0,\delta z_0) := \int_{\gamma z_0}^{\delta z_0} f(z) (Xz+Y)^{k-2} dz \in V_{k-2}(\CC).$$

The integral is to be taken coefficient wise.
Note that it is independent of the chosen path since we are integrating holomorphic functions.

\begin{lemma}\label{esmaplem}
For any $z_0 \in \HH$ and any matrices $\gamma,\delta \in \ZZ^{2 \times 2}$ with
positive determinant we have
$$ I_f(z_0,\gamma \delta z_0) = I_f(z_0,\gamma z_0) + I_f(\gamma z_0,\gamma \delta z_0)$$
and
$$ I_f(\gamma z_0,\gamma \delta z_0) = \det(\gamma)^{2-k} (\gamma.\big(I_{f|_\gamma} (z_0,\delta z_0)\big) )
 = (\det(\gamma)^{-1} \gamma).\big(I_{f|_\gamma} (z_0,\delta z_0)\big).  $$
\end{lemma}

\begin{proof}
The first statement is clear. Write $\gamma = \mat abcd$.
Recall that by Lemma~\ref{lmdiff}~(b), we have $d\gamma z = \frac{\det(\gamma)}{(cz+d)^2} dz.$
We compute further
\begin{align*}
I_f(\gamma z_0,\gamma\delta z_0) & = \int_{\gamma z_0}^{\gamma \delta z_0} f(z) (Xz+Y)^{k-2} dz\\
& = \int_{z_0}^{\delta z_0} f(\gamma z) (X\gamma z+Y)^{k-2} \frac{d\gamma z}{dz} dz\\
&= \det(\gamma)^{2-k}\int_{z_0}^{\delta z_0} f|_\gamma (z) (cz+d)^{k-2} (X\frac{az+b}{cz+d}+Y)^{k-2} dz\\
&= \det(\gamma)^{2-k}\int_{z_0}^{\delta z_0} f|_\gamma (z) (X(az+b)+Y(cz+d))^{k-2} dz\\
&= \det(\gamma)^{2-k}\int_{z_0}^{\delta z_0} f|_\gamma (z) ((Xa+Yc)z + (Xb+Yd))^{k-2} dz\\
&= \det(\gamma)^{2-k}\int_{z_0}^{\delta z_0} f|_\gamma (z) (\gamma.(Xz+Y))^{k-2} dz\\
&= \det(\gamma)^{2-k}\cdot \gamma .\big(\int_{z_0}^{\delta z_0} f|_\gamma (z) (Xz+Y)^{k-2} dz\big)\\
&= \det(\gamma)^{2-k}\cdot \gamma .\big(I_{f|_\gamma } (z_0,\delta z_0)\big).
\end{align*}
We recall that for a polynomial $P(X,Y)$ we have the action
$$(g.P)(X,Y) = P((X,Y) \mat abcd) = P(Xa+Yc,Xb+Yd).$$
\end{proof}

\begin{definition}
The space of {\em antiholomorphic cusp forms} $\overline{\Skg k\Gamma\CC}$
consists of the functions $z \mapsto \fbar(z) := \overline{f(z)}$
with $f \in \Skg k\Gamma\CC$.
\end{definition}

We can consider an antiholomorphic cusp form as a power series in~$\zbar$. For instance,
if $f(z)=\sum_{n=1}^\infty a_n e^{2 \pi i n z}$, then $\overline{f(z)} = \sum_{n=1}^\infty \overline{a_n} e^{2 \pi i n (-\zbar)}
= \tilde{f}(-\zbar)$, where $\tilde{f}(z)=\sum_{n=1}^\infty \overline{a_n} e^{2 \pi i n z}$.
Note that
\begin{equation}\label{eq:int-bar}
 \int_\alpha \overline{F(z)} d\zbar 
= \int_0^1 \overline{F(\alpha(t))} \frac{d\overline{\alpha}}{dt} dt
= \int_0^1 \overline{F(\alpha(t))} \overline{\frac{d\alpha}{dt}} dt
= \overline{\int_0^1 F(\alpha(t)) \frac{d\alpha}{dt} dt}
= \overline{\int_\alpha F(z) dz}
\end{equation}
for any piecewise analytic path $\alpha:[0,1] \to \CC$ and any integrable
complex valued function~$F$.
This means for $f \in \Skg k\Gamma\CC$:
$$ \overline{I_f(\gamma z_0, \delta z_0)} = 
\int_{\gamma z_0}^{\delta z_0} \overline{f(z)} (X\overline{z}+Y)^{k-2} d\overline{z} \in V_{k-2}(\CC).$$

\begin{proposition}\label{esmap}
Let $k \ge 2$ and $\Gamma \le \SL_2(\ZZ)$ be a subgroup of finite
index and fix $z_0,z_1 \in \HH$.
\begin{enumerate}[(a)]
\item The {\em Eichler-Shimura map}
\begin{align*}
\Mkg k\Gamma\CC \oplus \overline{\Skg k\Gamma\CC} & \to \h^1(\Gamma,V_{k-2}(\CC)),\\
(f,\gbar) & \mapsto (\gamma \mapsto I_f(z_0,\gamma z_0) + \overline{I_g(z_1,\gamma z_1)})
\end{align*}
is a well-defined homomorphism of $\CC$-vector spaces.
It does not depend on the choice of $z_0$ and~$z_1$.
\item The {\em induced Eichler-Shimura map}
\begin{align*}
\Mkg k\Gamma\CC \oplus \overline{\Skg k\Gamma\CC} & \to \h^1(\SL_2(\ZZ),\Hom_{\CC[\Gamma]}(\CC[\SL_2(\ZZ)],V_{k-2}(\CC))),\\
(f,\gbar) & \mapsto (a \mapsto (b \mapsto I_f(b z_0,ba z_0) + \overline{I_g(b z_1,ba z_1)}))
\end{align*}
is a well-defined homomorphism of $\CC$-vector spaces.
It does not depend on the choice of $z_0$ and~$z_1$.
Via the map from Shapiro's lemma, this homomorphism coincides with
the one from~(a).
\end{enumerate}
\end{proposition}

\begin{proof}
(a) For checking that the map is well-defined,
it suffices to compute that $\gamma \mapsto I_f(z_0,\gamma z_0)$ is
a $1$-cocycle:
$$ I_f(z_0,\gamma \delta z_0) 
= I_f(z_0,\gamma z_0) + I_f(\gamma z_0 , \gamma \delta z_0)
= I_f(z_0,\gamma z_0) + \gamma.I_f( z_0 , \delta z_0),$$
using Lemma~\ref{esmaplem} and $f|_\gamma = f$ since $\gamma \in \Gamma$.

The independence of the base point is seen as follows.
Let $\tilde{z}_0$ be any base point.
$$ I_f(\tilde{z}_0,\gamma \tilde{z}_0) = I_f (\tilde{z}_0, z_0) + I_f(z_0, \gamma z_0) + I_f(\gamma z_0, \gamma \tilde{z}_0)
= I_f(z_0,\gamma z_0) + (1-\gamma) I_f(\tilde{z}_0,z_0).$$
The difference of the cocycles $(\gamma \mapsto I_f(\tilde{z}_0,\gamma \tilde{z}_0))$
and $(\gamma \mapsto I_f(z_0,\gamma z_0))$ is hence the coboundary
$(\gamma \mapsto (1-\gamma)I_f(\tilde{z}_0,z_0))$.

(b) We first check that the map $(b \mapsto I_f(b z_0,ba z_0) + \overline{I_g(b z_1,ba z_1)})$
is indeed in the coinduced module
$\Hom_{\CC[\Gamma]}(\CC[\SL_2(\ZZ)],V_{k-2}(\CC))$.
For that let $\gamma \in \Gamma$. We have
$$ I_f(\gamma b z_0,\gamma ba z_0) = \gamma.(I_f(b z_0,ba z_0))$$
by Lemma~\ref{esmaplem}, as desired.
The map $\phi(a) := (b \mapsto I_f(b z_0,ba z_0) + \overline{I_g(b z_1,ba z_1)})$
is a cocycle:
\begin{align*}
\phi(a_1 a_2)(b) &= I_f(b z_0,ba_1 a_2 z_0) + \overline{I_g(b z_1,ba_1a_2 z_1)} \\
                 &=I_f(b z_0,ba_1 z_0) + I_f(ba_1 z_0,ba_1 a_2 z_0) 
+ \overline{I_g(b z_1,ba_1 z_1)} + \overline{I_g(ba_1 z_1,ba_1 a_2 z_1)} \\
& = \phi(a_1)(b) + \phi(a_2) (b a_1) = \phi(a_1)(b) + (a_1.(\phi(a_2)))(b), 
\end{align*}
by the definition of the left action of $\SL_2(\ZZ)$ on the coinduced module.
Note that the map in Shapiro's lemma in our situation is given by
$$ \phi \mapsto (\gamma \mapsto \phi(\gamma)(1) = I_f(z_0,\gamma z_0)) + \overline{I_g(z_1,\gamma z_1))} ,$$
which shows that the maps from (a) and (b) coincide.
The independence from the base point in~(b) now follows
from the independence in~(a).
\end{proof}

Next we identify the cohomology of $\SL_2(\ZZ)$ with the one of $\PSL_2(\ZZ)$.

\begin{proposition}
Let $\Gamma \le \SL_2(\ZZ)$ be a subgroup of finite index and let $R$
be a ring in which $2$ is invertible. Let $V$ be a left $R[\Gamma]$-module.
Assume that either $-1  \not\in \Gamma$ or $-1 \in \Gamma$ acts trivially on~$V$.
Then the inflation map
$$\h^1(\PSL_2(\ZZ),\Hom_{R[\Gammabar]}(R[\PSL_2(\ZZ)],V)) \xrightarrow{\infl} \h^1(\SL_2(\ZZ),\Hom_{R[\Gamma]}(R[\SL_2(\ZZ)],V))$$
is an isomorphism. We shall identify these two $R$-modules from now on.
\end{proposition}

\begin{proof}
If $-1 \not\in \Gamma$, then $\Gamma \cong \Gammabar$ and
$\Hom_{R[\Gamma]}(R[\SL_2(\ZZ)],V)^{\langle -1 \rangle}$ consists of all the
functions satisfying $f(g) = f(-g)$ for all $g \in \SL_2(\ZZ)$, which are precisely
the functions in $\Hom_{R[\Gammabar]}(R[\PSL_2(\ZZ)],V)$.

If $-1 \in \Gamma$ and $-1$ acts trivially on~$V$, then $f(-g) = (-1).f(g) = f(g)$
and so $-1$ already acts trivially on $\Hom_{R[\Gamma]}(R[\SL_2(\ZZ)],V)$. This
$R[\SL_2(\ZZ)]$-module is then naturally isomorphic to $\Hom_{R[\Gammabar]}(R[\PSL_2(\ZZ)],V)$
since any function is uniquely determined on its classes modulo $\langle -1 \rangle$.

Due to the invertibility of~$2$, the Hochschild-Serre exact sequence (Theorem~\ref{thm:hochschild-serre})
shows that inflation indeed gives the desired isomorphism because the third term
$\h^1(\langle \pm 1\rangle, \Hom_{R[\Gamma]}(R[\SL_2(\ZZ)],V))$ in the inflation-restriction sequence is zero
(see Proposition~\ref{corres}).
\end{proof}

\begin{proposition}\label{esres}
The kernel of the Eichler-Shimura map composed with the restriction
$$\Mkg k\Gamma\CC \oplus \overline{\Skg k\Gamma\CC}  
\to \h^1(\Gamma,V_{k-2}(\CC)) \to 
\prod_{c \in \Gamma \backslash \PP^1(\QQ)} \h^1(\Gamma_c,V_{k-2}(\CC))$$
is equal to $\Skg k\Gamma\CC \oplus \overline{\Skg k\Gamma\CC}$.
In particular, the image of $\Skg k\Gamma\CC \oplus \overline{\Skg k\Gamma\CC}$
under the Eichler-Shimura map lies in the parabolic cohomology
$\Hpar^1(\Gamma, V_{k-2}(\CC))$.
\end{proposition}

\begin{proof}
In order to simplify the notation of the proof, we shall only prove the case of a modular form $f \in \Mkg k\Gamma\CC$.
The statement for anti-homolorphic forms is proved in the same way.
The composition maps the modular form~$f$ to the $1$-cocycle (for $\gamma \in \Gamma_c$)
$$ \gamma \mapsto \int_{z_0}^{\gamma z_0} f(z) (Xz+Y)^{k-2} dz$$
with a fixed base point~$z_0 \in \HH$. The aim is now to move the
base point to the cusps. We cannot just replace $z_0$ by~$\infty$,
as then the integral might not converge any more (it converges on cusp
forms).
Let $c = M\infty$ be any cusp with $M = \mat abcd \in \SL_2(\ZZ)$.
We then have
$\Gamma_c = \langle MTM^{-1} \rangle \cap \Gamma = \langle MT^rM^{-1} \rangle$
for some $r \ge 1$.
Since $f$ is holomorphic in the cusps, we have
$$ f|_M(z) = \sum_{n=0}^\infty a_n e^{2\pi i n/r z} = a_0 + g(z)$$
and thus
$$ f(z) = a_0|_{M^{-1}}(z) + g|_{M^{-1}}(z) = \frac{a_0}{(-cz+a)^k} + g|_{M^{-1}}(z).$$
Now we compute the cocycle evaluated at $\gamma = MT^rM^{-1}$:
$$\int_{z_0}^{\gamma z_0} f(z) (Xz+Y)^{k-2} dz 
= a_0 \int_{z_0}^{\gamma z_0} \frac{(Xz+Y)^{k-2}}{(-cz+a)^k} dz +
\int_{z_0}^ {\gamma z_0} g|_{M^{-1}}(z)(Xz+Y)^{k-2} dz. $$
Before we continue by evaluating the right summand, we remark that the integral
$$ I_{g |_{M^{-1}}} (z_0, M \infty) 
= \int_{z_0}^{M\infty} g|_{M^{-1}}(z) (Xz+Y)^{k-2} dz
= M. \int_{M^{-1}z_0}^{\infty} g(z) (Xz+Y)^{k-2} dz$$
converges. We have
\begin{align*}
\int_{z_0}^{\gamma z_0} g|_{M^{-1}}(z)(Xz+Y)^{k-2} dz
&= (\int_{z_0}^{M \infty} + \int_{\gamma M\infty}^{\gamma z_0}) g|_{M^{-1}}(z)(Xz+Y)^{k-2} dz\\
&= (1-\gamma). \int_{z_0}^{M \infty} g|_{M^{-1}}(z) (Xz+Y)^{k-2} dz
\end{align*}
since $g|_{M^{-1}\gamma}(z) = g|_{T^rM^{-1}}(z) = g|_{M^{-1}}(z)$.
The $1$-cocycle
$\gamma \mapsto \int_{z_0}^{\gamma z_0} g|_{M^{-1}}(z)(Xz+Y)^{k-2} dz$
is thus a $1$-coboundary. Consequently, the class of the image of~$f$
is equal to the class of the $1$-cocycle
$$ \gamma \mapsto a_0 \int_{z_0}^{\gamma z_0} \frac{(Xz+Y)^{k-2}}{(-cz+a)^k} dz.$$

We have the isomorphism (as always for cyclic groups)
$$ \h^1(\Gamma_c,V_{k-2}(\CC)) \xrightarrow{\phi \mapsto \phi(MT^rM^{-1})} V_{k-2}(\CC)_{\Gamma_c}.$$
Furthermore, we have the isomorphism
$$ V_{k-2}(\CC)_{\Gamma_c} \xrightarrow {P \mapsto M^{-1}P} 
V_{k-2}(\CC)_{\langle T^r \rangle} \xrightarrow{P \mapsto P(0,1)} \CC$$
with polyomials~$P(X,Y)$.
Note that the last map is an isomorphism by the explicit description of
$V_{k-2}(\CC)_{\langle T^r \rangle}$.
Under the composition the image of the cocycle coming from the modular form~$f$
is
\begin{multline*} a_0 M^{-1}. \int_{z_0}^{\gamma z_0} \frac{(Xz+Y)^{k-2}}{(-cz+a)^k} dz (0,1)
= a_0 \int_{z_0}^{\gamma z_0} \frac{(Xz+Y)^{k-2}}{(-cz+a)^k} dz (-c,a)\\
= a_0 \int_{z_0}^{\gamma z_0} \frac{1}{(-cz+a)^2} dz
= a_0 \int_{M^{-1}z_0}^{T^rM^{-1}z_0} dz
= a_0 (M^{-1}z_0 + r - M^{-1}z_0) = r a_0,
\end{multline*}
as $(0,1) M^{-1} = (0,1) \mat d{-b}{-c}a = (-c,a)$.
This expression is zero if and only if $a_0 = 0$, i.e.\ if and only if $f$
vanishes at the cusp~$c$.

A similar argument works for anti-holomorphic cusp forms.
\end{proof}

\subsection{Theory: Cup product and Petersson scalar product}

This part owes much to the treatment of the Petersson scalar product by Haberland in~\cite{Haberland}
(see also \cite[\S 12]{CS}).

\begin{definition}\label{def:cup}
Let $G$ be a group and $M$ and $N$ be two left $R[G]$-modules. We
equip $M \otimes_R N$ with the diagonal left $R[G]$-action. Let $m,n \ge 0$.
Then we define the {\em cup product}
$$ \cup: \h^n(G,M) \otimes_R \h^m(G,N) \to \h^{n+m}(G,M \otimes_R N)$$
by
$$ \phi \cup \psi := ((g_1,\dots,g_n,g_{n+1},\dots,g_{n+m}) 
\mapsto \phi(g_1,\dots,g_n) \otimes (g_1 \cdots g_n).\psi(g_{n+1},\dots,g_{n+m})$$
on cochains of the bar resolution.
\end{definition}

This description can be derived easily from the natural one on the standard resolution.
For instance, \cite[\S5.3]{Brown} gives the above formula up to a sign (which does not matter in our application anyway
because we work in fixed degree).
In Exercise~\ref{excupwd} it is checked that the cup product is well-defined.

\begin{lemma}\label{lem:swap}
Keep the notation of Definition~\ref{def:cup} and let $\phi \in \h^n(G,M)$ and $\psi \in \h^m(G,N)$.
Then
$$\phi \cup \psi = (-1)^{mn} \psi \cup \phi$$
via the natural isomorphism $M \otimes_R N \cong N \otimes_R M$.
\end{lemma}

\begin{proof}
Exercise~\ref{ex:cup}.
\end{proof}

We are now going to formulate a pairing on cohomology, which will turn out
to be a version of the Petersson scalar product. We could introduce compactly
supported cohomology for writing it in more conceptual terms, but have decided
not to do this in order not to increase the amount of new material even more.

\begin{definition}
Let $M$ be an $R[\PSL_2(\ZZ)]$-module.
The {\em parabolic $1$-cocycles} are defined as
$$ \Zpar^1(\Gamma,M) = \ker(\Z^1(\Gamma,M) \xrightarrow{\res} 
\prod_{g \in \Gamma \backslash \PSL_2(\ZZ) / \langle T \rangle}
\Z^1(\Gamma \cap \langle g T g^{-1} \rangle, M) ).$$
\end{definition}

\begin{proposition}\label{defpairing}
Let $R$ be a ring in which $6$ is invertible. Let $M$ and $N$ be left $R[\PSL_2(\ZZ)]$-modules together with a $R[\PSL_2(\ZZ)]$-module homomorphism
$\pi: M \otimes_R N \to R$
where we equip $M \otimes_R N$ with the diagonal action. Write $G$ for $\PSL_2(\ZZ)$.
We define a pairing
$$ \langle,\rangle: \Z^1(G,M) \times \Z^1(G,N) \to R$$
as follows: Let $(\phi,\psi)$ be a pair of $1$-cocycles.
Form their cup product $\rho := \pi_*(\phi \cup \psi)$ in $\Z^2(G,R)$ via
$\Z^2(G,M \otimes_R N) \xrightarrow{\pi_*} \Z^2(G,R)$. As
$\h^2(G,R)$ is zero (Corollary~\ref{corhzweinull}), $\rho$ must be a $2$-coboundary, i.e.\
there is $a: G \to R$ (depending on $(\phi,\psi)$) such that
$$\rho (g,h) = \pi(\phi(g) \otimes g.\psi(h)) = a(h) - a(gh) + a(g).$$
We define the pairing by 
$$\langle\phi,\psi\rangle := a(T).$$

\begin{enumerate}[(a)]
\item The pairing is well-defined and bilinear. It can be expressed as
$$ \langle \phi,\psi \rangle = -\rho(\tau,\sigma) + \frac{1}{2} \rho(\sigma,\sigma) +
\frac{1}{3} (\rho(\tau,\tau)+\rho(\tau,\tau^2)).$$

\item If $\phi \in \Zpar^1(G,M)$, then $\rho(\tau,\sigma)=\rho(\sigma,\sigma)$ and
$$ \langle \phi,\psi \rangle = - \frac{1}{2} \rho(\sigma,\sigma) + 
\frac{1}{3}(\rho(\tau,\tau) + \rho(\tau,\tau^2)).$$
Moreover, $\langle \phi,\psi \rangle$ only depends on the class of~$\psi$
in $\h^1(G,N)$.

\item If $\psi \in \Zpar^1(G,N)$, then $\rho(\tau,\sigma) = \rho(\tau,\tau^2)$ and
$$ \langle \phi,\psi \rangle = \frac{1}{2} \rho(\sigma,\sigma) +
\frac{1}{3}\rho(\tau,\tau) - \frac{2}{3} \rho(\tau,\tau^2).$$
Moreover, $\langle \phi,\psi \rangle$ only depends on the class of~$\phi$
in $\h^1(G,M)$.

\item If $\phi \in \Zpar^1(G,M)$ and $\psi \in \Zpar^1(G,N)$, 
then $\rho(\sigma,\sigma) = \rho(\tau,\tau^2)$
and
$$ \langle \phi,\psi \rangle = -\frac{1}{6} \rho(\sigma,\sigma) +
\frac{1}{3}\rho(\tau,\tau).$$
\end{enumerate}
\end{proposition}

\begin{proof}
(a) We first have
$$ 0 = \pi(\phi(1) \otimes \psi(1)) = \rho(1,1) = a(1) - a(1) + a(1) = a(1),$$
since $\phi$ and $\psi$ are $1$-cocycles. Recall that the value of a
$1$-cocycle at~$1$ is always~$0$ due to $\phi(1) = \phi(1\cdot 1) = \phi(1) + \phi(1)$.
Furthermore, we have
\begin{align*}
\rho(\tau,\sigma) &= a(\sigma) - a(T) + a(\tau)\\
\rho(\sigma,\sigma) &= a(\sigma) - a(1) + a(\sigma) = 2 a(\sigma)\\
\rho(\tau,\tau^2) &= a(\tau^2) - a(1) + a(\tau) = a(\tau) + a(\tau^2)\\
\rho(\tau,\tau) &= a(\tau) - a(\tau^2) + a(\tau) = 2a(\tau) - a(\tau^2)
\end{align*}
Hence, we get
$a(T) = - \rho(\tau,\sigma) + a(\sigma) + a(\tau)$ and
$a(\sigma) = \frac{1}{2} \rho(\sigma,\sigma)$ as well as
$a(\tau) = \frac{1}{3}(\rho(\tau,\tau) + \rho(\tau,\tau^2))$,
from which the claimed formula follows. The formula also shows the
independence of the choice of~$a$ and the bilinearity.

(b) Now assume $\phi(T) = 0$. Using $T = \tau \sigma$ we obtain
\begin{multline*}
\rho(\tau,\sigma) = \pi(\phi(\tau) \otimes \tau \psi(\sigma))
= - \pi(\tau.\phi(\sigma) \otimes \tau \psi(\sigma))\\
= - \pi(\phi(\sigma)\otimes \psi(\sigma))
= \pi((\phi(\sigma)\otimes \sigma \psi(\sigma))) = \rho(\sigma,\sigma)
\end{multline*}
because $0 = \phi(T) = \phi(\tau\sigma) = \tau.\phi(\sigma) + \phi(\tau)$
and $0 = \psi(1) = \psi(\sigma^2) = \sigma.\psi(\sigma) + \psi(\sigma)$.
This yields the formula.

We now show that the pairing does not depend on the choice of
$1$-cocycle in the class of~$\psi$. To see this, let $\psi(g) = (g-1)n$ with $n\in N$
be a $1$-coboundary. Put $b(g) := \pi(-\phi(g) \otimes gn)$. Then,
using $\phi(gh)=g(\phi(h))+\phi(g)$, one immediately checks the equality
$$ \rho(g,h) = \pi(\phi(g)\otimes g(h-1)n) = g.b(h)-b(gh)+b(g).$$
Hence, $(\phi,\psi)$ is mapped to $b(T) = \pi(-\phi(T) \otimes T n) = \pi(0 \otimes Tn) = 0$.

(c) Let now $\psi(T) = 0$. Then $0=\psi(T) = \psi(\tau\sigma) = \tau\psi(\sigma) + \psi(\tau)$
and $0=\psi(\tau^3) = \tau \psi(\tau^2) + \psi(\tau)$, whence
$\tau\psi(\tau^2) = \tau \psi(\sigma)$. Consequently,
$$\rho(\tau,\sigma) = \pi(\phi(\tau) \otimes \tau \psi(\sigma))
= \pi(\phi(\tau) \otimes \tau \psi(\tau^2)) =  \rho(\tau,\tau^2),$$
implying the formula.

The pairing does not depend on the choice
of $1$-cocycle in the class of~$\phi$. Let $\phi(g) = (g-1)m$ be a $1$-coboundary
and put $c(g) := \pi(m \otimes \psi(g))$. Then the equality
$$ \rho(g,h) = \pi((g-1)m \otimes g\psi(h)) = g.c(h)-c(gh)+c(g)$$
holds. Hence, $(\phi,\psi)$ is mapped to $c(T) = \pi(m \otimes \psi(T)) = \pi(m \otimes 0) = 0$.

(d) Suppose now that $\phi(T)=0=\psi(T)$, then by what we have just seen
$$\rho(\tau,\sigma)=\rho(\sigma,\sigma)=\rho(\tau,\tau^2).$$
This implies the claimed formula.
\end{proof}

Our next aim is to specialise this pairing to the cocycles coming from modular
forms under the Eichler-Shimura map. We must first define a pairing
on the modules used in the cohomology groups.

On the modules $\Sym^{k-2}(R^2)$ we now define the {\em symplectic pairing}
over any ring~$R$ in which $(k-2)!$ is invertible.
Let $n=k-2$ for simplicity. The pairing for $n=0$ is just the
multiplication on~$R$. We now define the pairing for $n=1$ as
$$ R^2 \times R^2 \to R, \;\;\; \vect ac \bullet \vect bd := \det \mat abcd.$$
For any $g \in \SL_2(\ZZ)$ we have
$$ g{\vect a c} \bullet g{\vect b d}
= {\det g \mat abcd} = \det \mat abcd
= {\vect a c} \bullet {\vect b d}.$$
As the next step, we define a pairing on the $n$-th tensor power
of~$R^2$
$$(R^2 \otimes_R \dots \otimes_R R^2) \times (R^2 \otimes_R \dots \otimes_R R^2) \to R$$
by
$$ (\vect {a_1} {c_1} \otimes \dots \otimes \vect{a_n}{c_n}) \bullet
(\vect {b_1} {d_1} \otimes \dots \otimes \vect{b_n}{d_n})
:= \prod_{i=1}^n \vect {a_i} {c_i} \bullet \vect {b_i} {d_i}.$$
This pairing is still invariant under the $\SL_2(\ZZ)$-action.

Now we use the assumption on the invertibility of~$n!$ in order
to embed $\Sym^n(R^2)$ as an $R[S_n]$-module in the $n$-th tensor power,
where the action of the symmetric group~$S_n$ is on the indices. We have
that the map (in fact, $1/n!$ times the norm)
$$ \Sym^n(R^2) \to R^2 \otimes_R \dots \otimes_R R^2, \;\;\;
[ \vect {a_1}{c_1} \otimes \dots \otimes \vect{a_n}{c_n} ] \mapsto
\frac{1}{n!} \sum_{\sigma \in S_n} \vect {a_{\sigma(1)}} {c_{\sigma(1)}} 
\otimes \dots \otimes \vect{a_{\sigma(n)}}{c_{\sigma(n)}}$$
is injective (one can use Tate cohomology groups to see this) as the
order of $S_n$ is invertible in the ring.

Finally, we define the pairing on $\Sym^n(R^2)$ as the restriction of the
pairing on the $n$-th tensor power to the image of $\Sym^n(R^2)$ under
the embedding that we just described. This pairing is, of course,
still $\SL_2(\ZZ)$-invariant.

We point to the important special case 
$${\vect a c}^{\otimes(k-2)} \bullet{\vect b d}^{\otimes(k-2)}=(ad-bc)^{k-2}.$$
Hence, after the identification $\Sym^{k-2}(R^2) \cong V_{k-2}(R)$ from
Exercise~\ref{exsym}, the resulting pairing on $V_{k-2}(R)$ has
the property
$$ (aX+cY)^{k-2} \bullet (bX+dY)^{k-2} \mapsto (ad-bc)^{k-2}.$$
This pairing extends to a paring on coinduced modules
\begin{equation}\label{eq:coind-pair}
 \pi: \Hom_{R[\Gamma]}(R[\PSL_2(\ZZ)],V_{k-2}(R)) \otimes_R \Hom_{R[\Gamma]}(R[\PSL_2(\ZZ)],V_{k-2}(R)) \to R
\end{equation}
by mapping $(\alpha,\beta)$ to 
$\sum_{\gamma \in \Gamma \backslash \PSL_2(\ZZ)} \alpha(\gamma) \bullet \beta(\gamma)$.

\begin{proposition}\label{cuppet}
Let $k \ge 2$. Assume $-1 \not\in \Gamma$ (whence we view $\Gamma$ as a subgroup of $\PSL_2(\ZZ)$).
Let $f,g \in \Skg k\Gamma\CC$ be cusp forms. Denote by
$\phi_f$ the $1$-cocycle associated with~$f$
under the Eichler-Shimura map for the base point $z_0 = \infty$, i.e.\
$$ \phi_f (a) = (b \mapsto I_f(b \infty,ba \infty)) \in
\Z^1(\PSL_2(\ZZ),\Coind_\Gamma^{\PSL_2(\ZZ)} (V_{k-2}(\CC))).$$
Further denote
$$ \overline{\phi_f} (a) = (b \mapsto \overline{I_f(b \infty,ba \infty)}) \in
\Z^1(\PSL_2(\ZZ),\Coind_\Gamma^{\PSL_2(\ZZ)} (V_{k-2}(\CC))).$$
Similarly, denote by $\psi_g$ the $1$-cocycle
associated with~$g$ for the base point $z_1 = \zeta_6$.
Define a bilinear pairing as in Proposition~\ref{defpairing}
$$\langle,\rangle: \big(\Z^1(\PSL_2(\ZZ), \Coind_\Gamma^{\PSL_2(\ZZ)} (V_{k-2}(\CC))) \big)^2 \to \CC$$
with the product on the coinduced modules described in~\eqref{eq:coind-pair}.
Then the equation
$$ \langle \phi_f, \overline{\psi_g} \rangle = (2i)^{k-1}\mu(f,g) $$
holds where $(f,g)$ denotes the Petersson scalar product and $\mu$ the index
of $\Gamma$ in $\PSL_2(\ZZ)$.
\end{proposition}

\begin{proof}
Note that the choice of base point~$\infty$ is on the one hand well-defined
(the integral converges, as it is taken over a cusp form) and on the other hand
it ensures that $\phi_f(T) = \overline{\phi_f}(T) = 0$.
But note that $\psi_g$ is not a parabolic cocycle in general since the chosen base point is not~$\infty$ even
though $g$ is also a cusp form.

Now consider $\langle \phi_f,\overline{\psi_g} \rangle$.
Let $\rho(a,b) := \pi(\phi_f(a) \otimes a \overline{\psi_g}(b))$, where $\pi$ is from~\eqref{eq:coind-pair}.
We describe $\rho(a,b)$:
\begin{align*}
 \rho(a,b) & = \sum_\gamma 
\big( \int_{\gamma \infty}^{\gamma a \infty} f(z) (Xz+Y)^{k-2} dz\big) \bullet
\big( \int_{\gamma a \zeta_6}^{\gamma ab \zeta_6} \overline{g(z)} (X\zbar+Y)^{k-2} d\zbar\big)\\
&= \sum_\gamma
  \int_{\gamma a \zeta_6}^{\gamma ab \zeta_6} \int_{\gamma \infty}^{\gamma a \infty} 
 f(z) \overline{g(z)} \big( (Xz+Y)^{k-2} \bullet (X\zbar+Y)^{k-2} \big) dz d\zbar \\
&= \sum_\gamma 
  \int_{\gamma a \zeta_6}^{\gamma ab \zeta_6} \int_{\gamma \infty}^{\gamma a \infty} 
f(z)\overline{g(z)} (z-\zbar)^{k-2} dz d\zbar \\
&= \sum_\gamma 
  \int_{a \zeta_6}^{ab \zeta_6} \int_{\infty}^{a \infty} 
f|_\gamma (z)\overline{g|_\gamma (z)} (z-\zbar)^{k-2} dz d\zbar.
\end{align*}
where the sums run over a system of representatives of $\Gamma \backslash \PSL_2(\ZZ)$.
We obtain
\begin{align*}
&\rho(\sigma,\sigma)\\
 =& \sum_\gamma
  \int_{\sigma \zeta_6}^{\sigma^2 \zeta_6} \int_{\infty}^{\sigma \infty} 
f|_\gamma (z)\overline{g|_\gamma (z)} (z-\zbar)^{k-2} dz d\zbar\\ 
=& \sum_\gamma 
  \int_{\zeta_3}^{\zeta_6} \int_{\infty}^{0}
f|_\gamma (z)\overline{g|_\gamma (z)} (z-\zbar)^{k-2} dz d\zbar,\\
=& \sum_\gamma \big[ \int_{\zeta_3}^i \int_\infty^0 f|_\gamma (z) \overline{g|_\gamma (z)} (z-\zbar)^{k-2} dz d\zbar
+ \int_{\sigma \zeta_3}^{\sigma i} \int_{\sigma \infty}^{\sigma 0} f|_\gamma (z)
\overline{g|_\gamma (z)} (z-\zbar)^{k-2} dz d\zbar \big]\\
=& \sum_\gamma \big[\int_{\zeta_3}^i \int_\infty^0 f|_\gamma (z)\overline{g|_\gamma (z)} (z-\zbar)^{k-2} dz d\zbar
+ \int_{\zeta_3}^{i} \int_{\infty}^{0} f|_{\gamma\sigma} (z)\overline{g|_{\gamma\sigma} (z)} (z-\zbar)^{k-2} dz d\zbar \big]\\
=& 2 \sum_\gamma \int_{\zeta_3}^i \int_\infty^0 f|_\gamma (z)\overline{g|_\gamma (z)} (z-\zbar)^{k-2} dz d\zbar,
\end{align*}
and
\begin{align*}
\rho(\tau,\tau) =& \sum_\gamma  
\int_{\tau \zeta_6}^{\tau^2 \zeta_6} \int_{\infty}^{\tau \infty} 
 f|_\gamma (z)\overline{g|_\gamma (z)} (z-\zbar)^{k-2} dz d\zbar = 0\\
\rho(\tau,\tau^2) =& \sum_\gamma 
\int_{\tau \zeta_6}^{\tau^3 \zeta_6} \int_{\infty}^{\tau \infty}
f|_\gamma (z)\overline{g|_\gamma (z)} (z-\zbar)^{k-2} dz d\zbar = 0,
\end{align*}
since $\tau$ stabilises~$\zeta_6$.
It now suffices to compare with the formulas computed before
(Propositions \ref{defpairing} and~\ref{petexplicit}) to obtain the claimed formula.
\end{proof}

\subsection{Theory: The Eichler-Shimura theorem}

We can now, finally, prove that the Eichler-Shimura map is an isomorphism.
It should be pointed out again that the cohomology groups can be replaced by
modular symbols according to Theorem~\ref{compthm}.

\begin{theorem}[Eichler-Shimura]\label{esgammaeins}
Let $N \ge 4$ and $k\ge 2$. The Eichler-Shimura map and the induced Eichler-Shimura
map (Proposition~\ref{esmap}) are isomorphisms for $\Gamma = \Gamma_1(N)$.
The image of $\Skone kN\CC \oplus \overline{\Skone kN\CC}$ is isomorphic
to the parabolic subspace.
\end{theorem}

\begin{proof}
We first assert that the dimensions of both sides of the Eichler-Shimura
map agree and also that twice the dimension of the space of cusp forms
equals the dimension of the parabolic subspace. The dimension of the
cohomology group and its parabolic subspace was computed in
Proposition~\ref{dimheins}.
For the dimension of the left-hand side we refer to~\cite[\S 6.2]{Stein}.

Suppose that $(f,g)$ are in the kernel of the Eichler-Shimura map.
Then by Proposition~\ref{esres} it follows that $f$ and $g$ are both cuspidal.
Hence, it suffices to prove that the restriction
of the Eichler-Shimura map to $\Skone kN\CC \oplus \overline{\Skone kN\CC}$
is injective. In order to do this we choose $z_0=z_1=\infty$ as
base points for the Eichler-Shimura map, which is possible as the integrals
converge on cusp forms (as in Proposition~\ref{esmap} one sees that this
choice of base point does not change the cohomology class).
As in Proposition~\ref{cuppet}, we write $\phi_f$ for
the $1$-cocycle associated with a cusp form~$f$ for the base point~$\infty$.

We now make use of the pairing from Proposition~\ref{cuppet}
on $$\Z^1(\PSL_2(\ZZ),\Coind_\Gamma^{\PSL_2(\ZZ)} (V_{k-2}(\CC))),$$
where we put $\Gamma := \Gamma_1(N)$ for short.
This pairing induces a $\CC$-valued pairing $\langle \;, \;\rangle$
on
$$\Hpar^1(\PSL_2(\ZZ),\Coind_\Gamma^{\PSL_2(\ZZ)} (V_{k-2}(\CC))).$$
Next observe that the map
$$ \Skone kN\CC \oplus \overline{\Skone kN\CC} \xrightarrow{(f,\gbar) \mapsto (f+g,\fbar-\gbar)} \Skone kN\CC \oplus \overline{\Skone kN\CC}$$
is an isomorphism.
Let $f,g \in \Skone kN\CC$ be cusp forms and assume now that $(f+g,\fbar-\gbar)$ is sent to the zero-class in
$\Hpar^1(\PSL_2(\ZZ),\Coind_\Gamma^{\PSL_2(\ZZ)} (V_{k-2}(\CC)))$.
In that cohomology space, we thus have
$$ 0 = \phi_f + \phi_g + \overline{\phi_f} - \overline{\phi_g} = (\phi_f + \overline{\phi_f}) + (\phi_g - \overline{\phi_g})
= 2\Real(\phi_f) + 2i \Imag(\phi_g).$$
We conclude that the cohomology classes of $\phi_f + \overline{\phi_f}$ and $\phi_g - \overline{\phi_g}$ are both zero.

Now we apply the pairing as follows:
$$ 0=\langle \phi_f, \phi_f + \overline{\phi_f} \rangle = \langle \phi_f, \phi_f \rangle + \langle \phi_f, \overline{\phi_f} \rangle
= (2i)^{k-1} \mu (f,f)$$
where we used $\langle \phi_f, \phi_f \rangle = 0$ because of Lemma~\ref{lem:swap} (since the pairing is
given by the cup product), as well as Proposition~\ref{cuppet}.
Hence, $(f,f)=0$ and, thus, $f=0$ since the Petersson scalar product is positive definite.
Similar arguments with $0=\langle \phi_g, \phi_g - \overline{\phi_g} \rangle$ show $g=0$.
This proves the injectivity.
\end{proof}

\begin{remark}
The Eichler-Shimura map is in fact an isomorphism for all subgroups~$\Gamma$
of $\SL_2(\ZZ)$ of finite index. The proof is the same, but must use more involved
dimension formulae for the cohomology group (see Remark~\ref{remdimcoh}) and
modular forms.

In Corollary~\ref{esgammanull} we will see that there also is an Eichler-Shimura
isomorphism with a Dirichlet character.
\end{remark}

\subsection{Theoretical exercises}

\begin{exercise}\label{excupwd}
Check that the cup product is well-defined.

Hint: this is a standard exercise that can be found in many textbooks (e.g.\ \cite{Brown}).
\end{exercise}

\begin{exercise}\label{ex:cup}
Prove Lemma~\ref{lem:swap}.

Hint: \cite[(5.3.6)]{Brown}.
\end{exercise}

\section{Hecke operators}\label{sec:Hecke}

In this section we introduce Hecke operators on group cohomology using the double cosets approach
and we prove that the Eichler-Shimura isomorphism is compatible with the Hecke action on group
cohomology and modular forms.

\subsection{Theory: Hecke rings}

\begin{definition}
Let $N,n \in \NN$. We define
\begin{align*}
 \Delta_0^n(N) & = \{\mat abcd \in M_2(\ZZ) | \mat abcd \equiv \mat **0* \mod N, (a,N)=1, \det \mat abcd = n\},\\
 \Delta_1^n(N) & = \{\mat abcd \in M_2(\ZZ) | \mat abcd \equiv \mat 1*0* \mod N, \det \mat abcd = n\},\\
 \Delta_0(N) & = \bigcup_{n \in \NN} \Delta_0^n(N),\\
 \Delta_1(N) & = \bigcup_{n \in \NN} \Delta_1^n(N).
\end{align*}
\end{definition}

From now on, let
$(\Delta,\Gamma) = (\Delta_1(N),\Gamma_1(N))$ or $(\Delta,\Gamma) = (\Delta_0(N),\Gamma_0(N))$.

\begin{lemma}
Let $\alpha \in \Delta$. We put
$$ \Gamma_\alpha = \Gamma \cap \alpha^{-1} \Gamma \alpha \textnormal{ and }
\Gamma^\alpha = \Gamma \cap \alpha \Gamma \alpha^{-1}.$$
Then $\Gamma_\alpha$ has finite index in $\Gamma$ and $\alpha^{-1}\Gamma\alpha$
(one says that $\Gamma$ and $\alpha^{-1}\Gamma\alpha$ are commensurable), and also
$\Gamma^\alpha$ has finite index in $\Gamma$ and $\alpha\Gamma\alpha^{-1}$
(hence, $\Gamma$ and $\alpha \Gamma \alpha^{-1}$ are commensurable).
\end{lemma}

\begin{proof}
Let $n = \det \alpha$. One checks by matrix calculation that
$$\alpha^{-1}\Gamma(Nn)\alpha \subset \Gamma(N).$$
Thus,
$$ \Gamma(Nn) \subset \alpha^{-1} \Gamma(N) \alpha \subset \alpha^{-1} \Gamma \alpha.$$
Hence, we have $\Gamma(Nn) \subset \Gamma_\alpha$ and the first claim follows.
For the second claim, one proceeds similarly.
\end{proof}

\begin{example}
Let $\Gamma = \Gamma_0(N)$ and $p$ a prime. The most important case for the
sequel is $\alpha = \mat 100p$. An elementary calculation shows
$\Gamma^\alpha = \Gamma_0(Np)$.
\end{example}

\begin{definition}
Let $\alpha \in \Delta$. We consider the diagram
$$\xymatrix@=1cm{
  \Gamma_\alpha \backslash \HH \ar@{->}[r]^\alpha_{\tau \mapsto \alpha \tau} \ar@{->}[d]^{\pi_\alpha} &
  \Gamma^\alpha \backslash \HH \ar@{->}[d]^{\pi^\alpha} \\
  \Gamma \backslash \HH & \Gamma \backslash \HH,
}$$
in which $\pi^\alpha$ and $\pi_\alpha$ are the natural projections. One checks
that this is well defined by using
$\alpha \Gamma_\alpha \alpha^{-1} = \Gamma^\alpha$.

The {\em group of divisors} $\Div(S)$ on a Riemann surface~$S$ consists of all formal $\ZZ$-linear combinations of points of~$S$.
For a morphism $\pi: S \to T$ of Riemann surfaces, define the {\em pull-back} $\pi^*: \Div(T) \to \Div(S)$ and the
{\em push-forward} $\pi_*:\Div(S) \to \Div(T)$ uniquely by the rules $\pi^*(t) = \sum_{s \in S: \pi(s)=t} s$ and $\pi_*(s)=\pi(s)$
for points $t \in T$ and $s \in S$.

The {\em modular correspondence} or {\em Hecke correspondence} $\tau_\alpha$ is
defined as
$$ \tau_\alpha: \Div(Y_\Gamma) \xrightarrow{\pi_\alpha^*}
\Div(Y_{\Gamma_\alpha}) \xrightarrow{\alpha_*} \Div(Y_{\Gamma^\alpha}) \xrightarrow{\pi^\alpha_*} \Div(Y_\Gamma).$$
\end{definition}

These modular correspondences will be described more explicitly in a moment.
First a lemma:

\begin{lemma}\label{lem:doub}
Let $\alpha_i \in \Gamma$ for $i \in I$ with some index set~$I$. Then we have
$$ \Gamma = \bigsqcup_{i \in I} \Gamma_\alpha \alpha_i \; \Leftrightarrow \;
\Gamma \alpha \Gamma = \bigsqcup_{i \in I} \Gamma \alpha \alpha_i.$$ 
\end{lemma}

\begin{proof}
This is proved by quite a straight forward calculation.
\end{proof}

\begin{corollary}
Let $\alpha \in \Delta$ and $\Gamma \alpha \Gamma = \bigsqcup_{i \in I} \Gamma \alpha \alpha_i$.
Then the Hecke correspondence $\tau_\alpha: \Div(Y_\Gamma) \to \Div(Y_\Gamma)$ is given
by
$\tau \mapsto \sum_{i \in I} \alpha \alpha_i \tau$ 
for representatives $\tau \in \HH$.
\end{corollary}

\begin{proof}
It suffices to check the definition using Lemma~\ref{lem:doub}.
\end{proof}

\begin{remark}
We have $\Delta^n = \bigcup_{\alpha \in \Delta, \det \alpha = n} \Gamma \alpha \Gamma$
and one can choose finitely many $\alpha_i$ for $i \in I$ such that
$\Delta^n = \bigsqcup_{i \in I} \Gamma \alpha_i \Gamma$.
\end{remark}

\begin{definition}
Let $\Delta^n = \bigsqcup_{i \in I} \Gamma \alpha_i \Gamma$.
The Hecke operator $T_n$ on $\Div(Y_\Gamma)$ is defined as
$$ T_n = \sum_{i \in I} \tau_{\alpha_i}.$$
\end{definition}

Let us recall from equation~\eqref{sigmaa} the matrix  $\sigma_a \in \Gamma_0(N)$ (for $(a,N)=1$) which satisfies
$$\sigma_a \equiv \mat {a^{-1}}00a \mod N.$$

\begin{proposition}\label{zerlegung}
\begin{enumerate}[(a)]
\item We have the decomposition
$$\Delta_0^n(N) = \bigsqcup_a \bigsqcup_b \Gamma_0(N)\mat ab0d,$$
where $a$ runs through the positive integers with
$a \mid n$ and $(a,N) = 1$ and $b$ runs through the integers such that $0 \le b < d := n/a$.
\item We have the decomposition
$$\Delta_1^n(N) = \bigsqcup_a \bigsqcup_b \Gamma_1(N)\sigma_a \mat ab0d$$
with $a,b,d$ as in~(a).
\end{enumerate}
\end{proposition}

\begin{proof}
This proof is elementary.
\end{proof}

Note that due to $\sigma_a \in \Gamma_0(N)$, the matrices $\sigma_a  \mat ab0d$ used in part~(b)
also work in part~(a). One can thus use the same representatives regardless if one works with $\Gamma_0(N)$
or $\Gamma_1(N)$.
Note also that for $n=\ell$ a prime, these representatives are exactly the elements of $\calR_\ell$
from equation \eqref{setrp}.

Next, we turn to the important description of the Hecke algebra as a double coset algebra.

\begin{definition}
The {\em Hecke ring} $R(\Delta,\Gamma)$ is the free abelian group on the
double cosets $\Gamma \alpha \Gamma$ for $\alpha \in \Delta$.
\end{definition}

As our next aim we would like to define a multiplication, which then
also justifies the name "ring".
First let
$\Gamma \alpha \Gamma = \bigsqcup_{i=1}^n \Gamma \alpha_i$ und
$\Gamma \beta \Gamma = \bigsqcup_{j=1}^m \Gamma \beta_j$.
We just start computing.
$$ \Gamma \alpha \Gamma \cdot \Gamma \beta \Gamma = 
\bigcup_j \Gamma \alpha \Gamma \beta_j = \bigcup_{i,j} \Gamma \alpha_i \beta_j.$$
This union is not necessarily disjoint.
The left hand side can be written as a disjoint union of double cosets
$\bigsqcup_{k=1}^r \Gamma \gamma_k \Gamma$. Each of these double
cosets is again of the form
$$\Gamma \gamma_k \Gamma = \bigsqcup_{l=1}^{n_k} \Gamma \gamma_{k,l}.$$
We obtain in summary
$$ \Gamma \alpha \Gamma \cdot \Gamma \beta \Gamma = \bigcup_{i,j} \Gamma \alpha_i \beta_j
= \bigsqcup_k \bigsqcup_l \Gamma \gamma_{k,l}.$$
We will now introduce a piece of notation for the multiplicity with which every
coset on the right appears in the centre of the above equality. For fixed $k$ we define for every~$l$
$$ m_{k,l} = \# \{ (i,j) | \Gamma \gamma_{k,l} = \Gamma \alpha_i \beta_j \}.$$

The important point is the following lemma.

\begin{lemma}\label{lemmkl}
The number $m_{k,l}$ is independent of~$l$. We put $m_k := m_{k,l}$.
\end{lemma}

\begin{proof}
The proof is combinatorial and quite straight forward.
\end{proof}

\begin{definition}
We define the multiplication on $R(\Delta,\Gamma)$ by
$$ \Gamma \alpha \Gamma \cdot \Gamma \beta \Gamma 
= \sum_{k=1}^n m_k \Gamma\gamma_k\Gamma,$$
using the preceding notation.
\end{definition}

In Exercise~\ref{exheckering} you are asked to check that the Hecke
ring is indeed a ring.
The definition of the multiplication makes sense,
as it gives for Hecke correspondences:
$$ \tau_\alpha \circ \tau_\beta = \sum_{k=1}^n m_k \tau_{\gamma_k}.$$

\begin{definition}
For $\alpha \in \Delta$ let $\tau_\alpha = \Gamma \alpha \Gamma$.
We define (as above)
$$T_n = \sum_{\alpha} \tau_\alpha \in R(\Delta,\Gamma),$$
where the sum runs over a set of $\alpha$ such that
$\Delta^n = \bigsqcup_\alpha \Gamma \alpha \Gamma$.
For $a \mid d$ and $(d,N) = 1$ we let
$$ T(a,d) = \Gamma \sigma_a \mat a00d \Gamma \in R(\Delta,\Gamma).$$
\end{definition}

From Exercise~\ref{aufghecke}, we obtain the the following important corollary.

\begin{corollary}
We have $T_m T_n = T_n T_m$ and hence $R(\Delta,\Gamma)$ is a
commutative ring.
\end{corollary}

\subsection{Theory: Hecke operators on modular forms}

In this section we again let
$(\Delta,\Gamma) = (\Delta_0(N),\Gamma_0(N))$
or $(\Delta_1(N),\Gamma_1(N))$.
We now define an action of the Hecke ring $R(\Delta,\Gamma)$ on modular forms.

\begin{definition}
Let $\alpha \in \Delta$.
Suppose $\Gamma \alpha \Gamma = \bigsqcup_{i=1}^n \Gamma \alpha_i$ and let $f \in M_k(\Gamma)$.
We put
$$ f.\tau_\alpha := \sum_{i=1}^n f |_{\alpha_i}.$$
\end{definition}

\begin{lemma}
The function $f.\tau_\alpha$ again lies in $M_k(\Gamma)$.
\end{lemma}

\begin{proof}
For $\gamma \in \Gamma$ we check the transformation rule:
$$ \sum_i f |_{\alpha_i} |_{\gamma} = \sum_i f |_{\alpha_i \gamma} =
\sum_i f |_{\alpha_i},$$
since the cosets $\Gamma (\alpha_i \gamma)$ are a permutation
of the cosets $\Gamma \alpha_i$.
The holomorphicity of $f.\tau_\alpha$ is clear and the holomorphicity in the
cusps is not difficult.
\end{proof}

This thus gives the desired operation of $R(\Delta,\Gamma)$
on $M_k(\Gamma)$.

\begin{proposition}
Let $(\Delta,\Gamma) = (\Delta_0(N),\Gamma_0(N))$ and
$f \in M_k(\Gamma)$. The following formulae hold:
\begin{enumerate}[(a)]
\item $(f.T_m)(\tau) 
= \frac{1}{m} \sum_{{a \mid m}, {(a,N)=1}} \sum_{b=0}^{\frac{m}{a} - 1} a^k f(\frac{a\tau+b}{m/a})$,
\item $a_n(f.T_m) = \sum_{a \mid (m,n), (a,N)=1} a^{k-1} a_{\frac{mn}{a^2}}$.
\end{enumerate}
Similar formulae hold for $(\Delta_1(N),\Gamma_1(N))$, if one includes
a Dirichlet character at the right places.
\end{proposition}

\begin{proof}
(a) follows directly from Proposition~\ref{zerlegung}.

(b) is a simple calculation using
$ \sum_{b=0}^{d-1} e^{2 \pi i \frac{b}{d} n} =
\begin{cases}
0, & \textnormal{ if } d \nmid n\\
d, & \textnormal{ if } d \mid n.
\end{cases}$
\end{proof}

\begin{remark}
The Hecke ring $R(\Delta,\Gamma)$ also acts on $S_k(\Gamma)$.
\end{remark}

\begin{corollary}\label{corformel}
Let $(\Delta,\Gamma) = (\Delta_0(N),\Gamma_0(N))$.
For the action of the Hecke operators on $M_k(\Gamma)$ and $S_k(\Gamma)$
the following formulae hold:
\begin{enumerate}[(a)]
\item $T_n T_m = T_{nm}$ for $(n,m)=1$,
\item $T_{p^{r+1}} = T_p T_{p^r} - p^{k-1} T_{p^{r-1}}$, if $p \nmid N$, and
\item $T_{p^{r+1}} = T_p T_{p^r}$, if $p \mid N$.
\end{enumerate}
Here, $p$ always denotes a prime number.
Similar formulae hold for $(\Delta_1(N),\Gamma_1(N))$, if one includes
a Dirichlet character at the right places.
\end{corollary}

\begin{proof}
These formulae follow from Exercise~\ref{aufghecke} and the
definition of the action.
\end{proof}

Even though it is not directly relevant for our purposes, we include Euler products, which allow
us to express the formulae from the corollary in a very elegant way.

\begin{proposition}[Euler product]\label{euler}
The action of the Hecke operators $T_n$ on modular forms satisfies
the formal identity:
$$ \sum_{n=1}^\infty T_n n^{-s} = \prod_{p \nmid N} (1 - T_p p^{-s} + p^{k-1-2s})^{-1} \cdot
\prod_{p \mid N} (1 - T_p p^{-s})^{-1}.$$
\end{proposition}

That the identity is formal means that we can arbitrarily
permute terms in sums and products without considering questions
of convergence.

\begin{proof}
The proof is carried out in three steps.

\underline{1st~step}: Let $g: \ZZ \to \CC$ be any function.
Then we have the formal identity
$$ \prod_{p\textnormal{ prime}} \sum_{r=0}^\infty g(p^r) = 
\sum_{n=1}^\infty \prod_{p^r \parallel n} g(p^r).$$

For its proof, let first $S$ be a finite set of prime numbers. Then
we have the formal identity:
$$ \prod_{p \in S} \sum_{r=0}^\infty g(p^r) = 
\sum_{n=1, n \textnormal{ only has prime factors in }S}^\infty
\prod_{p^r \parallel n} g(p^r),$$
which one proves by multiplying out the left hand side (Attention! Here one permutes
the terms!). We finish the first step by letting $S$ run through arbitrarily
large sets.

\underline{2nd~step:} For $p \nmid N$ we have
$$(\sum_{r=0}^\infty T_{p^r} p^{-rs}) (1-T_pp^{-s} + p^{k-1-2s}) = 1$$
and for $p \mid N$:
$$(\sum_{r=0}^\infty T_{p^r} p^{-rs}) (1-T_pp^{-s}) = 1.$$

The proof of the second step consists of multiplying out these expressions
and to identify a ``telescope''.

\underline{3rd~step:} The proposition now follows by using the
first step with $g(p^r) = T_{p^r} p^{-rs}$ and plugging in
the formulae from the second step.
\end{proof}

\subsection{Theory: Hecke operators on group cohomology}

In this section we again let
$(\Delta,\Gamma) = (\Delta_0(N),\Gamma_0(N))$
or $(\Delta_1(N),\Gamma_1(N))$.
Let $R$ be a ring and $V$ a left $R[\Gamma]$-module which extends
to a semi-group action by the semi-group consisting of all $\alpha^\iota$
for $\alpha \in \Delta^n$ for all~$n$. 
Recall that $\mat abcd^\iota = \mat d{-b}{-c}a$.

We now give the definition of the Hecke operator $\tau_\alpha$ on
$\Div(\Gamma \backslash \HH)$ (see, for instance, \cite{DiamondIm} or \cite{faithful}).

\begin{definition}\label{defheckegp}
Let $\alpha \in \Delta$.
The {\em Hecke operator} $\tau_\alpha$ acting on group
cohomology is the composite
$$ \h^1(\Gamma,V) \xrightarrow{\res} \h^1(\Gamma^\alpha, V)
 \xrightarrow{\conj_\alpha} \h^1(\Gamma_\alpha, V)
 \xrightarrow{\cores} \h^1(\Gamma,V).$$
The first map is the {\em restriction}, and the third one is the {\em corestriction}.
We explicitly describe the second map on cocycles:
$$ \conj_\alpha: \h^1(\Gamma^\alpha, V) \to \h^1(\Gamma_\alpha, V), \;\;
c \mapsto \big( g_\alpha \mapsto \alpha^{\iota}.c(\alpha g_\alpha \alpha^{-1}) \big).$$
There is a similar description on the parabolic subspace 
and the two are compatible, see Exercise~\ref{exheckepar}.
\end{definition}

\begin{proposition}\label{shdesc}
Let $\alpha \in \Delta$.
Suppose that $\Gamma \alpha \Gamma = \bigcup_{i=1}^n \Gamma \delta_i$ is a disjoint
union. Then the Hecke operator $\tau_\alpha$ acts on $\h^1(\Gamma,V)$ 
and $\Hpar^1(\Gamma,V)$ by
sending the cocycle $c$ to $\tau_\alpha c$ defined by
$$ (\tau_\alpha c)(g) = \sum_{i=1}^n \delta_i^\iota c(\delta_i g \delta_{\sigma_g(i)}^{-1})$$
for $g \in \Gamma$. Here $\sigma_g(i)$ is the index such that
$\delta_i g \delta_{\sigma_g(i)}^{-1} \in \Gamma$.
\end{proposition}

\begin{proof}
We only have to describe the corestriction explicitly. For that we use that
$\Gamma = \bigcup_{i=1}^n \Gamma_\alpha g_i$
with $\alpha g_i = \delta_i$. Furthermore, by Exercise~\ref{excorestriction}
the corestriction of a cocycle $u \in \h^1(\Gamma_\alpha, V)$ is the cocycle $\cores(u)$ uniquely
given by 
\begin{equation}\label{eqcorestriction}
\cores(u)(g) = \sum_{i=1}^n g_i^{-1} u(g_i g g_{\sigma_g(i)}^{-1})
\end{equation}
for $g \in \Gamma$. Combining with the explicit description of the map $\conj_\alpha$
yields the result.
\end{proof}

\begin{definition}
For a positive integer~$n$, the {\em Hecke operator} $T_n$ is 
defined as $\sum_\alpha \tau_\alpha$,
where the sum runs through a system of representatives of the double
cosets $\Gamma \backslash \Delta^n / \Gamma$.

Let $a$ be an integer coprime to~$N$.
The {\em diamond operator}
$\diam {a}$ is defined as $\tau_\alpha$ for the matrix $\sigma_a \in \Gamma_0(N)$,
defined in Equation~\ref{sigmaa} (if the $\Gamma$-action on $V$ extends to
an action of the semi-group generated by $\Gamma$ and~$\alpha^\iota$; note
that $\alpha \in \Delta_0^1$, but in general not in $\Delta_1^1$).
\end{definition}

It is clear that the Hecke and diamond operators satisfy the
``usual'' Euler product.

\begin{proposition}
The Eichler-Shimura isomorphism is compatible with the Hecke
operators.
\end{proposition}

\begin{proof}
We recall the definition of Shimura's main involution:
${\mat abcd}^\iota = \mat d{-b}{-c}a$. In other words, for matrices
with a non-zero determinant, we have
$$ {\mat abcd}^\iota = (\det \mat abcd) \cdot {\mat abcd}^{-1}.$$
Let now $f \in \Mkg k\Gamma\CC$ be a modular form, $\gamma \in \Gamma$ and $z_0 \in \HH$.
For any matrix $g$ with non-zero determinant, Lemma~\ref{esmaplem} yields
$$ I_{f|_g} (z_0,\gamma z_0) = g^\iota I_f (g z_0,g \gamma z_0).$$

Let $\alpha \in \Delta$.
We show the compatibility of the Hecke operator $\tau_\alpha$ with the map
$$ f \mapsto (\gamma \mapsto I_f(z_0,\gamma z_0))$$
between $\Mkg k\Gamma\CC$ and $\h^1(\Gamma,V_{k-2}(\CC))$. The same arguments
will also work, when $I_f(z_0,\gamma z_0)$ is replaced by $J_\gbar(z_1,\gamma z_1))$
with anti-holomorphic cusp forms~$\gbar$.

Consider a coset decomposition $\Gamma \alpha \Gamma = \bigsqcup_i \Gamma \delta_i$.
We use notation as in Proposition~\ref{shdesc} and compute:
\begin{align*}
&I_{\tau_\alpha f} (z_0,\gamma z_0) \\
= &I_{\sum_i f|_{\delta_i}}(z_0,\gamma z_0)
= \sum_i I_{f|_{\delta_i}}(z_0,\gamma z_0) 
= \sum_i \delta_i^\iota I_f (\delta_i z_0, \delta_i\gamma z_0)\\
=& \sum_i \delta_i^\iota \big( 
I_f(\delta_i z_0, z_0) + I_f(z_0,\delta_i \gamma \delta_{\sigma_\gamma(i)}^{-1} z_0)
+ I_f(\delta_i \gamma \delta_{\sigma_\gamma(i)}^{-1} z_0, \delta_i \gamma
\delta_{\sigma_\gamma(i)}^{-1} \delta_{\sigma_\gamma(i)} z_0) \big)\\
=& \sum_i \delta_i^\iota I_f(z_0,\delta_i \gamma \delta_{\sigma_\gamma(i)}^{-1} z_0)
+ \sum_i \delta_i^\iota I_f(\delta_i z_0, z_0)
- \sum_i \delta_i^\iota \delta_i \gamma \delta_{\sigma_\gamma(i)}^{-1}
I_f(\delta_{\sigma_\gamma(i)} z_0, z_0)\\
=& \sum_i \delta_i^\iota I_f(z_0,\delta_i \gamma \delta_{\sigma_\gamma(i)}^{-1} z_0)
+ (1-\gamma) \sum_i \delta_i^\iota I_f(\delta_i z_0, z_0),
\end{align*}
since $\delta_i^\iota \delta_i \gamma \delta_{\sigma_\gamma(i)}^{-1}
= \gamma \delta_{\sigma_\gamma(i)}^\iota$.
Up to coboundaries, the cocycle $\gamma \mapsto I_{\tau_\alpha f} (z_0,\gamma z_0)$
is thus equal to the cocycle $\gamma \mapsto \sum_i \delta_i^\iota I_f(z_0,\delta_i \gamma \delta_{\sigma_\gamma(i)}^{-1} z_0)$, which by Proposition~\ref{shdesc}
is equal to $\tau_\alpha$ applied to the cocycle $\gamma \mapsto I_f(z_0,\gamma z_0)$,
as required.
\end{proof}

\begin{remark}
The conceptual reason why the above proposition is correct, is, of course,
that the Hecke operators come from Hecke correspondences.
\end{remark}

\subsection{Theory: Hecke operators and Shapiro's lemma}

We now prove that the Hecke operators are compatible with Shapiro's lemma.
This was first proved by Ash and Stevens~\cite{AshStevens}.
We need to say what the action
of $\alpha \in \Delta$ on the coinduced module $\Hom_{R[\Gamma]}(R[\SL_2(\ZZ)],V)$
should be. Here we are assuming that $V$ carries an action by the semi-group~$\Delta^\iota$
(that is, $\iota$ applied to all elements of~$\Delta$).

Let $U_N$ be the image of $\Delta^\iota$ in $\Mat_2(\ZZ/N\ZZ)$.
The natural map
$$ \Gamma \backslash \SL_2(\ZZ) \to U_N \backslash \Mat_2(\ZZ/N\ZZ)$$
is injective. Its image consists of those $U_N g$ such that
\begin{equation}\label{eq:star}
(0,1) g = (u,v) \textnormal{ with } \langle u,v \rangle = \ZZ/N\ZZ.
\end{equation}
If that is so, then we say for short that $g$ satisfies~\eqref{eq:star}.
Note that this condition does not depend on the choice of~$g$ in $U_N g$.
Define the $R[\Delta^\iota]$-module $\calC(N,V)$ as
$$ \{f \in \Hom_R (R[U_N \backslash \Mat_2(\ZZ/N\ZZ)],V) \; | \; 
f(g) = 0 \textnormal{ if } g \textnormal{ does not satisfy \eqref{eq:star}}\}$$
with the action of $\delta \in \Delta^\iota$ given by
$(\delta.f)(g) = \delta.(f(g\delta))$.
The module $\calC(N,V)$ is isomorphic to the coinduced module
$\Hom_{R[\Gamma]}(R[\SL_2(\ZZ)],V)$ as an $R[\Gamma]$-module by
\begin{align*}
\Hom_{R[\Gamma]}(R[\SL_2(\ZZ)],V) &\to \calC(N,V),\\
f &\mapsto \begin{cases}
(g \mapsto g f(g^{-1})) & \textnormal{for any $g \in \SL_2(\ZZ)$,}\\
0 & \textnormal{if $g$ does not satisfy \eqref{eq:star}.}
\end{cases}
\end{align*}
One might wonder why we introduce the new module $\calC(N,V)$ instead of working directly with $\Hom_{R[\Gamma]}(R[\SL_2(\ZZ)],V)$.
The point is that we cannot directly act on the latter with a matrix of determinant different from~$1$.
Hence we need a way to naturally extend the action. We do this by embedding 
$\Gamma\backslash \SL_2(\ZZ)$ into $U_N \backslash \Mat_2(\ZZ/N\ZZ)$. Of course, we then
want to work on the image of this embedding, which is exactly described by~\eqref{eq:star}.
The module $\calC(N,V)$ is then immediately written down in view of the identification
between $\Hom_{R[\Gamma]}(R[\SL_2(\ZZ)],V)$ and $\Hom_R(R[\Gamma\backslash \SL_2(\ZZ)],V)$
given by sending $f$ to $(g \mapsto g.f(g^{-1}))$ (which is clearly independent of the choice of $g$
in the coset $\Gamma g$).

\begin{proposition}
The Hecke operators are compatible with Shapiro's Lemma. More precisely,
for all $n \in \NN$ the following diagram commutes:
$$\xymatrix@=.8cm{
\h^1(\Gamma,V) \ar@{->}[r]^{T_n} & 
\h^1(\Gamma,V) \\
\h^1(\SL_2(\ZZ),\calC(N,V)) \ar@{->}[r]^{T_n} \ar@{->}[u]^{\textnormal{Shapiro}} & 
\h^1(\SL_2(\ZZ),\calC(N,V)). \ar@{->}[u]^{\textnormal{Shapiro}}
}$$
\end{proposition}

\begin{proof}
Let $j\in\{0,1\}$ indicate whether we work with $\Gamma_0$ or $\Gamma_1$.
Let $\delta_i$, for $i=1,\dots,r$ be the representatives of 
$\SL_2(\ZZ) \backslash \Delta_j^n(1)$ provided by Proposition~\ref{zerlegung}.
Say, that they are ordered such that $\delta_i$ for $i=1,\dots,s$ with $s \le r$
are representatives for $\Gamma \backslash \Delta_j^n(N)$. This explicitly means
that the lower row of $\delta_i^\iota$ is $(0,a)$  with $(a,N)=1$
(or even $(0,1)$ if $j=1$) for $i=1,\dots,s$.
If $s < i \le r$, then the lower row is $(u,v)$ with $\langle u,v \rangle \lneq \ZZ/N\ZZ$.

Let $c \in \h^1(\SL_2(\ZZ),\calC(N,V))$ be a $1$-cochain. Then, as required, we find
\begin{align*}
\textnormal{Shapiro}(T_n(c))(\gamma) 
&= \sum_{i=1}^r (\delta_i^\iota.c(\delta^i \gamma \delta_{\sigma_\gamma(i)}^{-1}))(\mat 1001)
= \sum_{i=1}^r \delta_i^\iota(c(\delta^i \gamma \delta_{\sigma_\gamma(i)}^{-1})(\delta_i^\iota))\\
&= \sum_{i=1}^s (\delta_i^\iota.c(\delta^i \gamma \delta_{\sigma_\gamma(i)}^{-1}))(\mat 1001))
= T_n(\textnormal{Shapiro}(c))(\gamma),
\end{align*}
where the second equality is due to the definition of the action and the third one holds since
$c(\delta^i \gamma \delta_{\sigma_\gamma(i)}^{-1})$ lies in $\calC(N,V)$ and thus evaluates
to~$0$ on $\delta_i^\iota$ for $i>s$.
\end{proof}

\begin{remark}
A very similar description exists involving $\PSL_2(\ZZ)$.
\end{remark}

\begin{remark}
It is possible to give an explicit description of Hecke operators on Manin symbols from Theorem~\ref{ManinSymbols}
by using Heilbronn matrices and variations as, for instance, done in~\cite{MerelUniversal}.
\end{remark}

\begin{remark}
One can show that the isomorphisms from Theorem~\ref{compthm} are compatible with Hecke operators.
\end{remark}

\subsection{Theory: Eichler-Shimura revisited}

In this section we present some corollaries and extensions
of the Eichler-Shimura theorem.
We first come to modular symbols with a character and, thus, also to modular symbols
for~$\Gamma_0(N)$.

\begin{corollary}[Eichler-Shimura]\label{esgammanull}
Let $N \ge 1$, $k \ge 2$ and $\chi: (\ZZ/N\ZZ)^\times \to \CC^\times$
be a Dirichlet character. Then the Eichler-Shimura map gives isomorphisms
$$ \Mk kN\chi\CC \oplus \overline{\Sk kN\chi\CC} \to \h^1(\Gamma_0(N), V_{k-2}^{\iota,\chi}(\CC)),$$
and
$$ \Sk kN\chi\CC \oplus \overline{\Sk kN\chi\CC} \to \Hpar^1(\Gamma_0(N), V_{k-2}^{\iota,\chi}(\CC)),$$
which are compatible with the Hecke operators.
\end{corollary}

\begin{proof}
Recall that the $\sigma_a$ form a system of coset representatives for
$\Gamma_0(N) / \Gamma_1(N) =: \Delta$ and that the group~$\Delta$
acts on $\h^1(\Gamma_0(N), V)$ by sending a cocycle $c$ to the
cocycle $\delta c$ (for $\delta \in \Delta$) which is defined by
$$ \gamma \mapsto \delta. c(\delta^{-1} \gamma \delta).$$
With $\delta = \sigma_a^{-1} = \sigma_a^\iota $, this reads
$$ \gamma \mapsto \sigma_a^\iota. c(\sigma_a \gamma \sigma_a^{-1}) = \tau_{\sigma_a} c = \langle a \rangle c.$$
Hence, $\sigma_a \in \Delta$-action acts through the inverse of the diamond operators.

We now appeal to the Hochschild-Serre exact sequence, using that
the cohomology groups (from index $1$ onwards) vanish if the group order
is finite and invertible. We get the isomorphism
$$ \h^1(\Gamma_0(N),V_{k-2}^{\iota,\chi}(\CC)) \xrightarrow{\res}
   \h^1(\Gamma_1(N),V_{k-2}^{\iota,\chi}(\CC))^\Delta.$$
Moreover, the Eichler-Shimura isomorphism is an isomorphism
of Hecke modules
$$ \Mkone kN\CC \oplus \overline{\Skone kN\CC} \to
\h^1(\Gamma_1(N),V_{k-2}^{\iota, \chi}(\CC)),$$
since for matrices in $\Delta_1(N)$ acting through the Shimura main involution
the modules $V_{k-2}^{\iota,\chi}(\CC)$ and $V_{k-2}(\CC)$ coincide.
Note that it is necessary to take $V_{k-2}^{\iota,\chi}(\CC)$ because the action
on group cohomology involves the Shimura main involution. Moreover, with this
choice, the Eichler-Shimura isomorphism is $\Delta$-equivariant.

To finish the proof, it suffices to take $\Delta$-invariants on both sides,
i.e.\ to take invariants for the action of the diamond operators.
The result on the parabolic subspace is proved in the same way.

Since Hecke and diamond operators commute, the Hecke action is compatible with
the decomposition into $\chi$-isotypical components.
\end{proof}

Next we  consider the action of complex conjugation.

\begin{corollary}
Let $\Gamma = \Gamma_1(N)$. The maps
$$ \Skg k\Gamma\CC \to \Hpar^1(\Gamma, V_{k-2}(\RR)), \;\;\;
f \mapsto (\gamma \mapsto \Real( I_f(z_0,\gamma z_0)))$$
and
$$ \Skg k\Gamma\CC \to \Hpar^1(\Gamma, V_{k-2}(\RR)), \;\;\;
f \mapsto (\gamma \mapsto \Imag( I_f(z_0,\gamma z_0)))$$
are isomorphisms (of real vector spaces) compatible with the Hecke operators.
A similar result holds in the presence of a Dirichlet character.
\end{corollary}

\begin{proof}
We consider the composite
$$ \Skg k\Gamma\CC \xrightarrow{f \mapsto \frac{1}{2}(f + \fbar)} 
\Skg k\Gamma\CC \oplus \overline{\Skg k\Gamma\CC} \xrightarrow{\text{Eichler-Shimura}}
\Hpar^1(\Gamma,V_{k-2}(\CC)).$$
It is clearly injective.
As $J_\fbar(z_0,\gamma z_0) = \overline{I_f (z_0,\gamma z_0)}$,
the composite map coincides with the first map in the statement.
Its image is thus already contained in the real vector space $\Hpar^1(\Gamma, V_{k-2}(\RR))$.
Since the real dimensions coincide, the map is an isomorphism.
In order to prove the second isomorphism, we use $f \mapsto \frac{1}{2i} (f - \fbar)$
and proceed as before.
\end{proof}

We now treat the $+$ and the $-$-space for the involution attached
to the matrix~$\eta = \mat {-1}001$ from equation~\eqref{defeta}.
The action of $\eta$ on $\h^1(\Gamma,V)$ is the action of the Hecke
operator $\tau_\eta$; strictly speaking, this operator is not defined because
the determinant is negative, however we use the same definition. To be precise
we have
$$ \tau_\eta: \h^1(\Gamma,V) \to \h^1(\Gamma,V), \;\;\; 
c \mapsto (\gamma \mapsto \eta^\iota. c(\eta \gamma \eta)),$$
provided, of course, that $\eta^\iota$ acts on~$V$ (compatibly with
the $\Gamma$-action).

We also want to define an involution $\tau_\eta$ on 
$\Skg k\Gamma\CC \oplus \overline{\Skg k\Gamma\CC}$.
For that recall that if $f(z) = \sum a_n e^{2\pi i n z}$,
then $\tilde{f}(z) := \sum \overline{a_n} e^{2\pi i n z}$
is again a cusp form in $\Skg k\Gamma\CC$ since we only applied
a field automorphism (complex conjugation) to the coefficients
(think of cusp forms as maps from the Hecke algebra over~$\QQ$ to~$\CC$).
We define $\tau_\eta$ as the composite
$$ \tau_\eta: \Skg k\Gamma\CC \xrightarrow{f \mapsto (-1)^{k-1}\tilde{f}}
           \Skg k\Gamma\CC \xrightarrow{\tilde{f} \mapsto \overline{\tilde{f}}}
            \overline{\Skg k\Gamma\CC}.$$
Similarly, we also define $\tau_\eta: \overline{\Skg k\Gamma\CC} \to\Skg k\Gamma\CC$
and obtain in consequence an involution $\tau_\eta$ on
$\Skg k\Gamma\CC \oplus \overline{\Skg k\Gamma\CC}$.
We consider the function $(-1)^{k-1}\overline{\tilde{f}(z)}$ as a function of~$\zbar$.
We have
\begin{multline*}
 \tau_\eta(f)(\zbar)=(-1)^{k-1}\overline{\tilde{f}(z)} 
= (-1)^{k-1}\overline{\sum_n \overline{a_n} e^{2\pi i n z}}
= (-1)^{k-1}\sum_n a_n e^{2 \pi i n (-\zbar)}\\
= (-1)^{k-1}f(-\zbar) = f|_\eta (\zbar).
\end{multline*}

\begin{proposition}
The Eichler-Shimura map commutes with $\tau_\eta$.
\end{proposition}

\begin{proof}
Let $f \in \Skg k\Gamma\CC$ (for simplicity). We have to check whether $\tau_\eta$
of the cocycle attached to~$f$ is the same as the cocycle attached to $\tau_\eta (f)$.
We evaluate the latter at a general $\gamma \in \Gamma$ and compute:
\begin{align*}
J_{(-1)^{k-1}\overline{\tilde{f}}}(\infty,\gamma \infty) 
&= (-1)^{k-1} \int_{\infty}^{\gamma \infty} f(-\zbar) (X\zbar+Y)^{k-2} d\zbar\\
&= - \int_{\infty}^{\gamma \infty} f(-\zbar) (X(-\zbar)-Y)^{k-2} d\zbar\\
&= \int_{\gamma \infty}^{\infty} f(-\zbar) (X(-\zbar)-Y)^{k-2} d\zbar\\
&= \int_{0}^{\infty} f(-\overline{(\gamma \infty + it)}) (X(- \overline{(\gamma \infty + it)}  )-Y)^{k-2} (-i) dt\\
&= - \int_{0}^{\infty} f(-\gamma \infty + it) (X(- \gamma \infty + it)  -Y)^{k-2} i dt\\
&= \int_{\infty}^{-\gamma \infty } f(z) (Xz-Y)^{k-2} dz\\
&= \eta^\iota. I_f(\infty,-\gamma \infty)= \eta^\iota. I_f(\infty,\eta \gamma \eta \infty).
\end{align*} 
This proves the claim.
\end{proof}

\begin{corollary}
Let $\Gamma = \Gamma_1(N)$. The maps
$$ \Skg k\Gamma\CC \to \Hpar^1(\Gamma, V_{k-2}(\CC))^+, \;\;\;
f \mapsto (1+\tau_\eta).(\gamma \mapsto I_f(z_0,\gamma z_0))$$
and
$$ \Skg k\Gamma\CC \to \Hpar^1(\Gamma, V_{k-2}(\CC))^-, \;\;\;
f \mapsto (1-\tau_\eta).(\gamma \mapsto I_f(z_0,\gamma z_0))$$
are isomorphisms compatible with the Hecke operators,
where the~$+$ (respectively the~$-$) indicate the subspace invariant
(respectively anti-invariant) for the involution~$\tau_\eta$.
A similar result holds in the presence of a Dirichlet character.
\end{corollary}

\begin{proof}
Both maps are clearly injective (consider them as being given
by $f \mapsto f + \tau_\eta f$ followed by the Eichler-Shimura map)
and so dimension considerations show that they are isomorphisms.
\end{proof}

\subsection{Theoretical exercises}

\begin{exercise}\label{exheckering}
Check that $R(\Delta,\Gamma)$ is a ring (associativity and
distributivity).
\end{exercise}

\begin{exercise}\label{aufghecke}
Show the formula
$$ T_m T_n = \sum_{d \mid (m,n), (d,N)=1} d T(d,d) T_{\frac{mn}{d^2}}.$$
Also show that $R(\Delta,\Gamma)$ is generated by $T_p$ and $T(p,p)$
for $p$ running through all prime numbers.
\end{exercise}

\begin{exercise}\label{exheckepar}
Check that the Hecke operator $\tau_\alpha$ from Definition~\ref{defheckegp}
restricts to $\Hpar^1(\Gamma,V)$.
\end{exercise}

\begin{exercise}\label{excorestriction}
Prove Equation~\ref{eqcorestriction}.
\end{exercise}

\subsection{Computer exercises}

\begin{cexercise}
Implement Hecke operators.
\end{cexercise}

\bigskip
\noindent Gabor Wiese\\
University of Luxembourg\\
Mathematics Research Unit\\
Maison du nombre\\
6, avenue de la Fonte\\
L-4364 Esch-sur-Alzette\\
Grand-Duchy of Luxembourg\\
\url{gabor.wiese@uni.lu}

\end{document}